\documentclass[a4paper,11pt]{amsart}
\usepackage[margin=1.2in]{geometry}
\pdfoutput=1

\usepackage{mathpazo,amsmath,amssymb,amscd,color,xcolor,mathrsfs}
\usepackage{tikz}
\usetikzlibrary{decorations.markings}
\usepackage{float}
\usepackage{comment}
\definecolor{green}{RGB}{0,144,0}
\definecolor{bluegreen}{RGB}{17,100,180}
\usepackage[linkcolor = bluegreen, citecolor = bluegreen, colorlinks = True]{hyperref}
\numberwithin{equation}{section}
\numberwithin{figure}{section}
\numberwithin{table}{section}
\usepackage[all]{hypcap}
\usepackage{mathtools}

\usepackage{pythonhighlight}

\newtheorem{theorem}{Theorem}[section]
\newtheorem*{theorem*}{Theorem}

\newtheorem{lemma}[theorem]{Lemma}

\newtheorem{proposition}[theorem]{Proposition}

\newtheorem{remark}[theorem]{Remark}

\newcommand{\teichmuller}{Teichm{\"u}ller }
\newcommand{\calH}{\mathcal{H}}

\newcommand{\C}{\mathbb{C}}
\newcommand{\Z}{\mathbb{Z}}
\newcommand{\R}{\mathbb{R}}

\newcommand{\N}{\mathbb{N}}
\DeclareMathOperator{\SL}{\text{SL}}
\DeclareMathOperator{\GL}{\text{GL}}

\author{Luke Jeffreys}
\address{School of Mathematics , University of Bristol, Fry Building, Woodland Road, Bristol BS8 1UG}
\curraddr{}
\email{luke.jeffreys@bristol.ac.uk}
\thanks{The first author is a Leverhulme Early Career Fellow (ECF-2023-553) and so thanks the Leverhulme Trust for their support.}

\author{Carlos Matheus}
\address{Centre de Math{\'e}matiques Laurent Schwartz, {\'E}cole Polytechnique, 91128 Palaiseau Cedex, France}
\curraddr{}
\email{carlos.matheus@math.cnrs.fr}

\title[Diameter bounds for $\SL(2,\Z)$-orbits of origamis]{Diameter bounds for $\boldsymbol{\SL(2,\Z)}$-orbits of origamis in $\boldsymbol{\calH(2)}$ and the Prym loci in $\boldsymbol{\calH(4)}$ and $\boldsymbol{\calH(6)}$}

\subjclass[2020]{Primary: 32G15, 30F30, 30F60. Secondary: 05C12.}
\keywords{origamis, orbit graphs, graph diameter}

\begin{document}

\begin{abstract}
    Using algorithms implicit in the classification of $\SL(2,\Z)$-orbits of primitive origamis in the stratum $\calH(2)$ due to Hubert--Leli\`evre and McMullen, we give diameter bounds on the resulting orbit graphs. Since the machinery of McMullen from $\calH(2)$ is generalised and reused in Lanneau and Nguyen's classification of the orbits of Prym eigenforms in $\calH(4)$ and $\calH(6)$, we are also able to obtain diameter bounds for the orbit graphs in this setting as well. In each stratum, we obtain diameter bounds of the form $O(N^{2/3}\log N)$, where $N$ is the size of the orbit graph.
\end{abstract}

\maketitle

\section{Introduction}

A square-tiled surface is an orientable connected surface obtained from a finite cover of the unit square torus $\mathbb{T}^2:=\mathbb{C}/(\mathbb{Z}\oplus i\mathbb{Z})$ which is possibly branched only at $0\in\mathbb{T}^2$. In the literature, square-tiled surfaces are also called origamis (for short). 

Origamis play a special role in the study of the natural $\SL(2,\mathbb{R})$ action on moduli spaces of Abelian differentials. For example, since they correspond to integral points of such moduli spaces, one can compute Masur--Veech volumes of these spaces by counting origamis (cf. \cite{Zo} and \cite{EsOk}). Moreover, their $\SL(2,\mathbb{R})$-orbits are closed subvarieties of the moduli spaces known as \emph{arithmetic Teichm\"uller curves}. We refer the reader to the recent book \cite{AtMa} and survey \cite{Fil} for further explanations about moduli spaces of Abelian differentials and its connections to several topics in Mathematics (including Dynamics and Hodge Theory).  

As it turns out, origamis can be organised into $\SL(2,\mathbb{Z})$-orbits leading to Schreier graphs sitting on the corresponding arithmetic Teichm\"uller curves. That is, one turns an $\SL(2,\Z)$-orbit into a combinatorial graph by setting the vertices to be the origamis in the orbit with the edges corresponding to the action of a generating set $S$ for $\SL(2,\Z)$ --- this is isomorphic to the coset Schreier graph $\text{Sch}(\SL(2,\Z), H ,S)$ with $H$ being the stabliser of an origami in the orbit and $S$ as before. In this context, McMullen conjectured that the family of such Schreier graphs associated to primitive origamis of genus $2$ with a single conical singularity is an expander family. In a previous paper \cite{JM}, we gave indirect evidence towards McMullen's conjecture by showing that these graphs are eventually non-planar. In the present paper, we give another piece of indirect evidence by establishing the following bound on the diameter of these graphs: 

\begin{theorem}\label{thm:main-H2}
    If $X$ is a primitive origami in the stratum $\mathcal{H}(2)$ of the moduli space of translation surfaces of genus two with a single conical singularity, then the graph $\mathcal{G}(X)$ associated to its $\SL(2,\mathbb{Z})$-orbit has a diameter $O(|V|^{2/3}\log |V|) = O(n^2\log n)$, where $|V|$ is the number of vertices of $\mathcal{G}(X)$ and $n$ is the number of unit squares tiling $X$ (i.e., the degree of the branched covering map $X\to\mathbb{T}^2$ defining $X$). 
\end{theorem}

This article is organised along the following lines. In Section \ref{sec:prelim}, we quickly review some basic material about moduli spaces of translation surfaces, $\SL(2,\mathbb{Z})$-orbits of primitive origamis, and the surface parameters introduced by Hubert--Leli\`evre that we use to describe a given origami. In Section~\ref{sec:butterfly}, we describe McMullen's classification result and introduce the butterfly moves he used in his proof. In Section \ref{sec:H2-McM}, we prove Theorem \ref{thm:main-H2} by effectivising McMullen's results on butterfly moves. Finally, we take advantage of the results of Lanneau--Nguyen extending the technology of butterfly moves to the Prym loci of the minimal strata of moduli spaces of translation surfaces in genus 3 and 4 (resp.) in order to establish in Sections \ref{sec:H4} and \ref{sec:H6} (resp.) the analogs of Theorem \ref{thm:main-H2} in these settings. We obtain the following.

\begin{theorem}\label{thm:main-H4-H6}
    If $X$ is a primitive origami contained in the Prym eigenform loci of $\calH(4)$ or $\calH(6)$, then the graph $\mathcal{G}(X)$ associated to its $\SL(2,\Z)$-orbit has diameter $O(|V|^{2/3}\log|V|) = O(n^2\log n)$, where $|V|$ is the number of vertices of $\mathcal{G}(X)$ and $n$ is the number of unit squares tiling $X$.
\end{theorem}

\begin{remark}
    The diameter bounds in this paper are far from optimal: for instance, a positive answer to McMullen's conjecture would imply a diameter bound of the form $O(\log n)$. For this reason, we include in Appendix~\ref{app:improvements} a short discussion of potential sources of improvement (partly supported by numerical experiments) in the basic arguments.

    In Appendix~\ref{sec:H2-HL}, we analyse the algorithm implicit in the classification proof of Hubert--Leli\`evre. We do this is in an appendix since, in this preliminary analysis, it produces a weaker diameter bound.
\end{remark}
     
\subsection*{Acknowledgements} For the purpose of open access, the authors have applied a Creative Commons Attribution (CC BY) licence to any Author Accepted Manuscript version arising from this submission. We also thank the anonymous referees for their valuable comments and suggestions, which have improved the readability of the paper.

\section{Preliminaries}\label{sec:prelim}

Here, we remind the reader of the necessary background on origamis and their $\SL(2,\Z)$-orbits as well as the language required to discuss the works of Hubert--Leli\`evre and McMullen.

\subsection{Translation surfaces and their moduli spaces} A compact Riemann surface $M$ of genus $g\geq 1$ equipped with a non-trivial Abelian differential $\omega$ is called a \textit{translation surface}. This nomenclature is justified by the fact that the local primitives of $\omega$ away from the set $\Sigma$ of its zeroes yields a collection of charts on $M\setminus\Sigma$ whose changes of coordinates are given by translations of the complex plane $\mathbb{C}$. By the Riemann-Roch Theorem, $\omega$ has $2g-2$ zeroes counted with multiplicities, i.e., if $\sigma=|\Sigma|$ is the cardinality of $\Sigma$ and $k_1,\dots,k_{\sigma}$ are the vanishing orders of $\omega$ at the elements of $\Sigma$, then $\sum\limits_{j=1}^{\sigma}k_j=2g-2$. 

By gathering together translation surfaces $X=(M,\omega)$ with a prescribed list $\kappa=(k_1,\dots,k_{\sigma})$ of orders of zeroes of $\omega$, one obtains a stratum $\mathcal{H}(\kappa)$ of the moduli space of Abelian differentials of genus $g$. This is a complex orbifold of dimension $2g+\sigma-1$ carrying a natural $\SL(2,\mathbb{R})$ action (consisting of post-composing the local primitives of $\omega$ with the usual action of $\SL(2,\mathbb{R})$ on $\mathbb{R}^2=\mathbb{C}$). For further information and references about these objects, the reader is encouraged to consult Athreya--Masur's book \cite{AtMa}. 

\subsection{Origamis and their $\boldsymbol{\SL(2,\Z)}$-orbits} An \textit{origami} (or square-tiled surface) is a translation surface $(M,\pi^*(dz))$ obtained from a finite cover $\pi:M\to \mathbb{T}^2=\mathbb{C}/(\mathbb{Z}+i\mathbb{Z})$ possibly branched only at $0\in\mathbb{T}^2$. Alternatively, an origami is constructed from a pair $(h,v)$ of permutations acting transitively on $n$ symbols by taking unit squares $sq(i)$, $i=1,\dots,n$, and gluing the rightmost vertical side of $sq(i)$ to the leftmost vertical side of $sq(h(i))$ and the topmost horizontal side of $sq(i)$ to the bottommost horizontal side of $sq(v(i))$. An origami is said to be \textit{primitive} if it is not a cover of another origami different from itself or the unit-square torus. Equivalently, an origami is primitive if the group $\langle h, v\rangle$ is primitive as a permutation group.

The group $\SL(2,\mathbb{Z})$ leaves invariant the set of primitive origamis. In fact, $\SL(2,\mathbb{Z})$ is generated by the matrices 
\[T = \begin{pmatrix}
    1 & 1 \\
    0 & 1
\end{pmatrix} \quad \textrm{and} \quad 
S = \begin{pmatrix}
    1 & 0 \\
    1 & 1
\end{pmatrix}\] and one can check that their actions on pairs of permutations are $T(h,v)=(h,vh^{-1})$ and $S(h,v)=(hv^{-1},v)$ (see, for example,~\cite[Figure 2.2]{JM}). In particular, the $\SL(2,\mathbb{Z})$-orbits of origamis are coded by graphs whose vertices are the elements of these orbits and edges corresponding to elements deduced one from the other by applying $T$ or $S$.

\begin{remark}
    Later, we will also use the matrix $R=\begin{pmatrix}
    0 & -1 \\
    1 & 0
\end{pmatrix}$ satisfying $R=ST^{-1}S$. Note that $R$ rotates the origami by $\frac{\pi}{2}$ anti-clockwise.
\end{remark}

For a primitive origami $X$, the graph $\mathcal{G}(X)$ of the main theorem is the graph whose vertex set is the set of origamis in the $\SL(2,\Z)$-orbit of $X$ with edges corresponding to the action of $T$ and $S$.

The origamis $\pi:(M,\omega)\to(\mathbb{T}^2,dz)$ studied in this paper come equipped with an involution $\iota$ of $M$ taking $\omega$ to $-\omega$. In this case, one can count the number $l_0$ of fixed points of $\iota$ over $0\in\mathbb{T}^2$ that are distinct from the zero of $\omega$, and the numbers $l_1, l_2, l_3$ of fixed points of $\iota$ over the other three $2$-torsion points of $\mathbb{T}^2$. As it turns out, the pair $(l_0,[l_1,l_2,l_3])$, where $[l_1,l_2,l_3]$ is an unordered triple, is an invariant of the $\SL(2,\mathbb{Z})$-orbit of $(M,\omega)$ called its HLK-invariant.  

\subsection{Surface parameters}\label{subsec:params} The HLK-invariant was originally used by Hubert and Leli\`evre \cite{HL} to distinguish between two $\SL(2,\mathbb{Z})$-orbits (called A and B) of primitive origamis in $\mathcal{H}(2)$ tiled by a prime number of unit squares. In general, by putting together the works \cite{HL}, \cite{McM} and \cite{LR}, if $X$ is a primitive origami in $\mathcal{H}(2)$ tiled by $n\geq 3$ squares, then 
\begin{itemize}
    \item $X$ falls into a single $\SL(2,\mathbb{Z})$-orbit whenever $n=3$ or $n\geq 4$ is even, and 
    \item $X$ falls into one of two possible $\SL(2,\mathbb{Z})$-orbits (called A and B) whenever $n\geq 5$ is odd: the A orbit has cardinality $\frac{3}{16}(n-1)n^2\prod\limits_{p|n, p \textrm{ prime}}(1-p^{-2})$ and the B orbit has cardinality $\frac{3}{16}(n-3)n^2\prod\limits_{p|n, p \textrm{ prime}}(1-p^{-2})$.
\end{itemize}
In the language of HLK-invariants, the orbit for even $n\geq 4$ corresponds to the HLK-invariant $(1,[2,2,0])$, the A orbit and the case of $n=3$ corresponds to the HLK-invariant $(0,[3,1,1])$, and the B orbit corresponds to the HLK-invariant $(2,[1,1,1])$.

Note that in all cases the size of the orbit is $\Theta(n^{3})$.

Moreover, one can exhibit explicit representatives of the A and B orbits by using the surface parameters introduced by Hubert--Leli\`evre. More concretely, an origami $X$ in $\mathcal{H}(2)$ tiled by $n$ unit squares can decompose into one or two cylinders in the horizontal direction. If $X$ has two cylinders then it is characterised by the widths $w_1$, $w_2$ and heights $h_1$, $h_2$ of these cylinders, and the twists $t_1$, $t_2$ (i.e., the relative positions of the conical singularity in the boundaries of these cylinders): cf. Figure \ref{fig:H2-params} below. Note that these parameters satisfy $h_1w_1+h_2w_2=n$ (because the total area of $X$ is $n$), and, if $n\geq 5$ is odd, one gets a representative of the A orbit, resp. B orbit, by setting $t_1=0=t_2$, $h_1=1=h_2$, $w_1=1$, $w_2=n-1$, resp. $t_1=0=t_2$, $h_1=2$, $h_2=1$, $w_1=1$, $w_2=n-2$. For a primitive origami, we require $\gcd(h_1,h_2) = 1$.

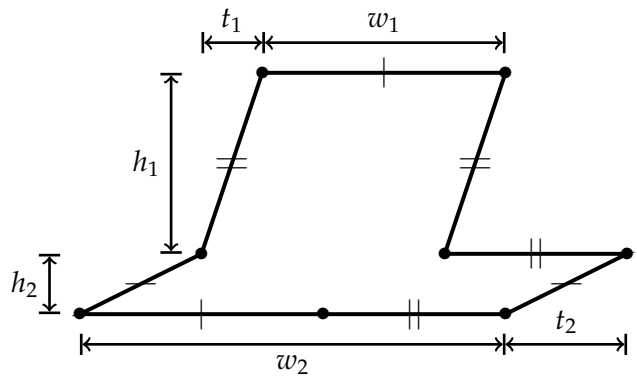
\begin{figure}[t]
    \centering
    \begin{tikzpicture}[scale = 0.8, line width = 1.5pt]
        \draw (0,0) -- node[rotate = 90]{$|$}(2,1) -- node[rotate = 90]{$||$}(3,4) -- node{$|$}(7,4) -- node[rotate = 90]{$||$}(6,1) -- node{$||$}(9,1) -- node[rotate = 90]{$|$}(7,0) -- node{$||$}(4,0) -- node{$|$} cycle;
        \draw[|<->|,line width = 1pt] (1.5,1) -- node[left]{$h_{1}$}(1.5,4);
        \draw[|<->|,line width = 1pt] (-0.5,0) -- node[left]{$h_{2}$}(-0.5,1);
        \draw[|<->|,line width = 1pt] (3,4.5) -- node[above]{$w_{1}$}(7,4.5);
        \draw[|<->|,line width = 1pt] (0,-0.5) -- node[below]{$w_{2}$}(7,-0.5);
        \draw[|<->,line width = 1pt] (2,4.5) -- node[above]{$t_{1}$}(3,4.5);
        \draw[<->|,line width = 1pt] (7,-0.5) -- node[above]{$t_{2}$}(9,-0.5);
        \foreach \x/\y in {0/0,2/1,3/4,7/4,6/1,9/1,7/0,4/0}{
            \node at (\x,\y) {$\bullet$};
        }
    \end{tikzpicture}
    \caption{Two-cylinder surface parameters in $\calH(2)$.}
    \label{fig:H2-params}
\end{figure}

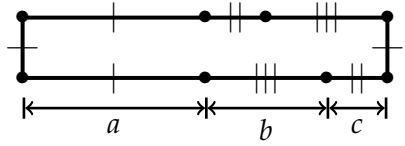
\begin{figure}[t]
    \centering
    \begin{tikzpicture}[scale = 0.8, line width = 1.5pt]
        \draw (0,0) -- node[rotate = 90]{$|$}(0,1) -- node{$|$}(3,1) -- node{$||$}(4,1) -- node{$|||$}(6,1) -- node[rotate = 90]{$|$}(6,0) -- node{$||$}(5,0) -- node{$|||$}(3,0) -- node{$|$} cycle;
        \draw[|<->,line width = 1pt] (0,-0.5) -- node[below]{$a$}(3,-0.5);
        \draw[|<->,line width = 1pt] (3,-0.5) -- node[below]{$b$}(5,-0.5);
        \draw[|<->|,line width = 1pt] (5,-0.5) -- node[below]{$c$}(6,-0.5);
        \foreach \x/\y in {0/0,0/1,3/1,4/1,6/1,6/0,5/0,3/0}{
            \node at (\x,\y) {$\bullet$};
        }
    \end{tikzpicture}
    \caption{A one-cylinder cusp representative in $\calH(2)$.}
    \label{fig:H2-1-cyl}
\end{figure}

The \emph{cusp} of an origami, $X$, is its orbit under the action of the horizontal shear $T$ and the \emph{cusp width} is the size of this orbit (i.e., the minimal $i$ for which $T^{i}(X) = X$).

If $X$ is a two-cylinder surface with parameters $(w_1,h_1,t_1,w_2,h_2,t_2)$ then it has cusp width
\[\text{lcm}\left(\frac{w_{1}}{\gcd(w_{1},h_{1})},\frac{w_{2}}{\gcd(w_{2},h_{2})}\right).\]
This follows from the fact that $T$ acts on the twist parameters as $t_i \mapsto (t_i + h_i)\!\!\mod w_i$. The \textit{cusp representative} of $X$ is the unique surface in the cusp with $0\leq t_{i} < \gcd(w_{i},h_{i})$ (cf. \cite[Lemma 3.1]{HL}).

Every cusp of a one-cylinder origami has a representative, as shown in Figure~\ref{fig:H2-1-cyl}, with saddle connections of lengths $a,b,c>0$ such that $n = a+b+c$ and $\gcd(a,b,c) = 1$. For $n\geq 4$, the cusp width is $n$ (for $n = 3$ the surface with $(a,b,c) = (1,1,1)$ has cusp width 1). Note that the cusp representative has two cylinders in the vertical direction so that applying $R$ (rotation by $\pi/2$) reaches a two-cylinder surface.

\section{Butterfly moves and McMullen's algorithm}\label{sec:butterfly}

In this section, we remind the reader of the details of McMullen's classification and describe how we turn this into an algorithm for moving within the $\SL(2,\Z)$-orbit of an origami.

\subsection{Prototypes and butterfly moves}\label{subsec:proto}

A \teichmuller curve is an algebraic and isometric immersion of a finite-volume hyperbolic Riemann surface into the moduli space $\mathcal{M}_{g}$ of Riemann surfaces of genus $g$. McMullen~\cite{McM,McM06} classified all of the primitive \teichmuller curves in genus two. The main source of such \teichmuller curves are the so-called Weierstrass curves $W_{D}$, parameterised by integers (called \emph{discriminants}) $D \geq 5$ with $D \equiv 0$ or 1 modulo 4. These curves consist of those Riemann surfaces $M \in \mathcal{M}_{2}$ whose Jacobians admit real multiplication by the quadratic order $\mathcal{O}_{D} :=\Z[x]/\langle x^2+bx+c\rangle, b,c\in\Z$ with $D=b^2-4c$, and for which there exists a holomorphic one-form $\omega$ on $M$ such that $(M,\omega) \in\calH(2)$ and $\mathcal{O}_{D}\cdot\omega \subset \C\cdot\omega$ (such forms are said to be \emph{eigenforms} for real multiplication by $\mathcal{O}_{D}$).

A translation surface $X = (M,\omega) \in\calH(2)$ projects to $W_{n^{2}}$ if and only if $X$ is in the $\GL^+(2,\R)$-orbit of an $n$-squared origami. The classification of the $\SL(2, \Z)$-orbits of primitive origamis in $\calH(2)$ is equivalent to the fact that $W_{n^{2}}$ is connected for $n = 3$ and $n \geq 4$ even, and has two connected components for $n \geq 5$ odd.

McMullen~\cite[Section 3]{McM} defines a 4-tuple $(a,b,c,e)$ to be a \emph{prototype of discriminant $D$} if we have
\begin{itemize}
    \item $D = e^{2}+4bc$;
    \item $0<b,c$;
    \item $0\leq a < \gcd(b,c)$;
    \item $c+e < b$; and
    \item $\gcd(a,b,c,e) = 1$.
\end{itemize}
We collect these in the set $\mathcal{P}_{D}$. A prototype is said to be \emph{reduced} if it has the form $(0,b,1,e)$. The set of reduced prototypes is denoted by $\mathcal{S}_{D}$.

Any translation surface $(M,\omega)\in\Omega W_{D}$ splits in infinitely many ways as a connected sum of elliptic curves:
$$(M,\omega) = (E_{1},\omega_{1})\underset{I}{\#}(E_{2},\omega_{2}),$$
where $I\subset \C$ is an interval that projects to a closed loop in $E_{1}$ and to an embedded interval in $E_{2}$.

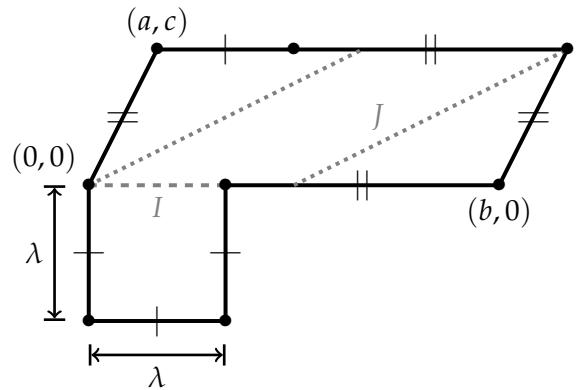
\begin{figure}[b]
    \centering
    \begin{tikzpicture}[scale = 0.9, line width = 1.5pt]
        \draw (0,0) -- (2,0) -- (2,2) -- (6,2) -- (7,4) -- (1,4) -- (0,2) -- cycle;
        \draw[dashed, gray] (0,2)--(2,2);
        \draw[dotted, gray] (0,2)--(4,4);
        \draw[dotted, gray] (3,2)--(7,4);
        \draw[|<->|,line width = 1pt] (0,-0.5) -- node[below]{$\lambda$}(2,-0.5);
        \draw[|<->|,line width = 1pt] (-0.5,0) -- node[left]{$\lambda$}(-0.5,2);
        \node[above left] at (0,2) {$(0,0)$};
        \node[below, gray] at (1,2) {$I$};
        \node[left, gray] at (4.5,3) {$J$};
        \node[below] at (6,2) {$(b,0)$};
        \node[above] at (1,4) {$(a,c)$};
        \node at (1,0) {$|$};
        \node at (2,4) {$|$};
        \node at (4,2) {$||$};
        \node at (5,4) {$||$};
        \node[rotate = 90] at (0,1) {$|$};
        \node[rotate = 90] at (2,1) {$|$};
        \node[rotate = 90] at (0.5,3) {$||$};
        \node[rotate = 90] at (6.5,3) {$||$};
        \foreach \x/\y in {0/0,2/0,2/2,6/2,7/4,3/4,1/4,0/2}{
            \node at (\x,\y) {$\bullet$};
        }
    \end{tikzpicture}
    \caption{The surface associated to the prototypical splitting $(a,b,c,e)$. The splitting realises the surface $(M,\omega)$ as $(E_{1},\omega_{1})\underset{I}{\#}(E_{2},\omega_{2})$. A butterfly move changes the splitting to $(F_{1},\eta_1)\underset{J}{\#}(F_{2},\eta_2)$ with the interval $J = [0,(b,0) + q(a,c)]\subset\C$.}
    \label{fig:prototype}
\end{figure}

A prototype $(a,b,c,e)$ corresponds to a \textit{prototypical splitting} of an eigenform $(M,\omega)$ as a connected sum of elliptic curves in the following way. First, let $\lambda = (e+\sqrt{D})/2$ and consider the interval $I = [0,\lambda]\subset \C$. Define $E_1 = \C/\Lambda_{1}$ and $E_2 = \C/\Lambda_{2}$ by setting
\[\Lambda_1 = \Z(\lambda,0)\oplus\Z(0,\lambda)\,\,\,\,\,\,\text{and}\,\,\,\,\,\,\Lambda_2 = \Z(b,0)\oplus\Z(a,c).\]
The conditions on the the prototype $(a,b,c,e)$ above give us that $I$ projects to a closed curve in $E_{1}$ and is embedded in $E_{2}$. We then define the prototypical splitting associated to $(a,b,c,e)$ to be the eigenform obtained as a connected sum of the elliptic curves over $I$:
$$(M,\omega) = (E_{1},\omega_{1})\underset{I}{\#}(E_{2},\omega_{2}).$$
See Figure~\ref{fig:prototype} (ignore the interval $J$ in the figure for now).

Given a prototypical splitting $(M,\omega) = (E_{1},\omega_1)\underset{I}{\#}(E_{2},\omega_2)$, McMullen defines the \emph{butterfly move} $B_{q}$, for some $q\in\{1,2,\ldots\}\cup\{\infty\}$, to be a change of splitting to $$(M,\omega) = (F_{1},\eta_1)\underset{J}{\#}(F_{2},\eta_2)$$ with the interval $J = [0,(b,0) + q(a,c)]$ for finite $q$ or $J = [0,(a,c)]$ for $q=\infty$. See Figure~\ref{fig:prototype}. Note that this operation is only possible for certain admissible $q$ and, for such $q$, the interval $J$ projects to a closed loop now in $E_{2}$. By applying elements of $\GL^+(2,\R)$, we can then move to the prototypical splitting corresponding to this new splitting interval and obtain the associated prototype $(a',b',c',e')$. The corresponding prototypical splitting exists and is unique by~\cite[Theorem 3.1]{McM}.

For a given prototype, the set of admissible $q$ values is
\[\{q\,\mid\,q\in\N, (e+2qc)^{2} < D\}\cup\{\infty\}.\]
Observe that $q = 1$ is always admissible since
\[(e+2c)^{2} = e^2 + 4ec+4c^{2} = e^{2}+4(e+c)c < e^{2}+4bc = D.\] Admissibility is equivalent to requiring that the projection of the interval $[0,(b,0) + q(a,c)]$ (or $[0,(a,c)]$ for $q = \infty$) only intersects the projection of $I = [0,\lambda]$ at 0.

Given a $q$ value that is admissible for $(a,b,c,e)$, McMullen~\cite[Section 7]{McM} determined that the butterfly move $B_{q}$ has the following effects:

\begin{itemize}
    \item if $q$ is finite, then $B_{q}(a,b,c,e) = (a',b',c',e')$ with $c' = \gcd(qc,b+qa)$ and $e' = -e-2qc$, with $b'$ determined from the condition that $D = (e')^2 + 4b'c'$. McMullen does not determine a formula for $a'$.
    \item if $q = \infty$, then $B_{q}(a,b,c,e) = (a',b',c',e')$ with $c' = \gcd(a,c)$ and $e' = -e-2c$, with $b'$ determined from these values as above. Once again, no formula for $a'$ is given.
\end{itemize}

Since we need to understand $a'$ in order to bound the number of butterfly moves, we will unpack further the action of the butterfly move on the prototype $(a,b,c,e)$. In doing so, we will also see the proof of the above formulae. Indeed, the following is essentially a rewriting of the proof of~\cite[Theorem 7.3]{McM}.

Firstly, recall that the butterfly move $B_{q}$ for finite $q$ sends a splitting $$(M,\omega) = (E_{1},\omega_1)\underset{I}{\#}(E_{2},\omega_2)$$ over $I = [0,\lambda]$ with homology
$$H_{1}(M,\Z) = H_{1}(E_{1},\Z)\oplus H_{1}(E_{2},\Z) = \large\left(\Z(\lambda,0)\oplus\Z(0,\lambda)\large\right)\oplus\large\left(\Z(b,0)\oplus\Z(a,c)\large\right)$$
to a splitting $(M,\omega) = (F_{1},\eta_1)\underset{J}{\#}(F_{2},\eta_2)$ over $J = [0,(b,0) + q(a,c)]$ with homology
$$H_{1}(M,\Z) = H_{1}(F_{1},\Z)\oplus H_{1}(F_{2},\Z).$$
We then apply elements of $\GL^+(2,\R)$ to move to the corresponding prototypical splitting which will have an associated prototype $(a',b',c',e')$.  The homology of this prototypical splitting will be
$$\large\left(\Z(\lambda',0)\oplus\Z(0,\lambda')\large\right)\oplus\large\left(\Z(b',0)\oplus\Z(a',c')\large\right)$$
with $\lambda' = (e'+\sqrt{D}/2)$.

To carry out the above, we first take the basis of $H_{1}(F_{1},\Z)$ to be the columns of the matrix
\[L_{1} = \begin{pmatrix}
    b+qa &  \lambda+a \\
    qc & c
\end{pmatrix}\]
and the basis of $H_{1}(F_{2},\Z)$ to be the columns of the matrix
\[L_{2} = \begin{pmatrix}
    \lambda &  b+qa \\
    0 & \lambda+qc
\end{pmatrix}.\]
See the left of Figure~\ref{fig:F-bases}.

\begin{figure}[t]
    \centering
    \begin{minipage}{0.45\textwidth}
    \centering
    \begin{tikzpicture}[scale = 0.9, line width = 1.5pt]
        \draw (0,0) -- (2,0) -- (2,2) -- (6,2) -- (7,4) -- (1,4) -- (0,2) -- cycle;
        \draw[line width = 0pt, fill = blue, opacity = 0.3] (0,2) -- (1/3,2+2/3) -- (3,4) -- (4,4) -- cycle;
         \draw[line width = 0pt, fill = blue, opacity = 0.3] (2,2) -- (6,4) -- (7,4) -- (3,2) -- cycle;
         \draw[line width = 0pt, fill = blue, opacity = 0.3] (5,2) -- (6+1/3,2+2/3) -- (6,2) -- cycle;
        \draw[dotted, gray] (0,2)--(4,4);
        \draw[dotted, gray] (3,2)--(7,4);
        \draw[dotted, gray] (1/3,2+2/3)--(3,4);
        \draw[dotted, gray] (2,2)--(6,4);
        \draw[dotted, gray] (5,2)--(6+1/3,2+2/3);
        \draw[blue] (0,2) -- (4,4);
        \draw[blue] (3,2) -- (6,3.5);
        \draw[blue, ->] (0,2) -- (2.9,3.95);
        \draw[red, ->, bend left = 30] (0,0) to (1.95,0.05);
        \draw[red, bend left = 13, postaction={decorate, decoration={markings, mark=at position 0.6 with {\arrow{>}}}}] (0,0) to (4.5,4);
        \draw[red, postaction={decorate, decoration={markings, mark=at position 0.5 with {\arrow{>}}}}] (3.5,2) -- (6.85,3.925);
        \draw[blue, ->] (6,3.5) -- (6.9,3.95);
        \draw[|<->|,line width = 1pt] (0,-0.5) -- node[below]{$\lambda$}(2,-0.5);
        \draw[|<->|,line width = 1pt] (-0.5,0) -- node[left]{$\lambda$}(-0.5,2);
        \node[above left] at (0,2) {$(0,0)$};
        \node[below] at (6,2) {$(b,0)$};
        \node[above] at (1,4) {$(a,c)$};
        \node at (1,0) {$|$};
        \node at (2,4) {$|$};
        \node at (4,2) {$||$};
        \node at (5,4) {$||$};
        \node[rotate = 90] at (0,1) {$|$};
        \node[rotate = 90] at (2,1) {$|$};
        \node[rotate = 90] at (0.5,3) {$||$};
        \node[rotate = 90] at (6.5,3) {$||$};
        \foreach \x/\y in {0/0,2/0,2/2,6/2,7/4,3/4,1/4,0/2}{
            \node at (\x,\y) {$\bullet$};
        }
    \end{tikzpicture}
    \end{minipage}\hfill
    \begin{minipage}{0.45\textwidth}
    \centering
    \hspace*{-1.35cm}
    \begin{tikzpicture}[scale = 0.9, line width = 1.5pt]
        \draw (0,0) -- (2,0) -- (2,2) -- (6,2) -- (7,4) -- (1,4) -- (0,2) -- cycle;
        \draw[line width = 0pt, fill = blue, opacity = 0.3] (2,2) -- (3,4) -- (7,4) -- (6,2) -- cycle;
        \draw[dotted, gray] (2,2)--(3,4);
        \draw[blue, ->, bend right = 20] (6,2) to (2.1,2.1);
        \draw[blue, ->, bend left  = 30] (6,2) to (6.9,3.95);
        \draw[red, ->, bend right = 30] (2,0) to (0.05,0.05);
        \draw[red, ->, bend left = 35] (2,0) to (2.9,3.95);
        \draw[|<->|,line width = 1pt] (0,-0.5) -- node[below]{$\lambda$}(2,-0.5);
        \draw[|<->|,line width = 1pt] (-0.5,0) -- node[left]{$\lambda$}(-0.5,2);
        \node[above left] at (0,2) {$(0,0)$};
        \node[below] at (6,2) {$(b,0)$};
        \node[above] at (1,4) {$(a,c)$};
        \node at (1,0) {$|$};
        \node at (2,4) {$|$};
        \node at (4,2) {$||$};
        \node at (5,4) {$||$};
        \node[rotate = 90] at (0,1) {$|$};
        \node[rotate = 90] at (2,1) {$|$};
        \node[rotate = 90] at (0.5,3) {$||$};
        \node[rotate = 90] at (6.5,3) {$||$};
        \foreach \x/\y in {0/0,2/0,2/2,6/2,7/4,3/4,1/4,0/2}{
            \node at (\x,\y) {$\bullet$};
        }
    \end{tikzpicture}
    \end{minipage}
    \caption{\textit{Left:} The torus $F_1$ of the new splitting is shaded. The basis $(b+qa,qc), (\lambda+a,c)$ for $H_1(F_1,\Z)$ is shown in blue. The basis $(\lambda,0), (b+qa,\lambda+qc)$ for $H_1(F_2,\Z)$ is shown in red. \textit{Right:} A similar diagram for the choice of $q = \infty$.}
    \label{fig:F-bases}
\end{figure}

Set $e' = -e-2qc$ and $\lambda' = (e'+\sqrt{D})/2$. To obtain the homological splitting $\Z(\lambda',0)\oplus\Z(0,\lambda')$ for the first summand, we apply $\lambda'L_{1}^{-1}$ to our surface as an element of $\GL^+(2,\R)$. This sends the basis $L_{2}$ to
\[N = \lambda'L_{1}^{-1}L_{2} = \begin{pmatrix}
    c & -a-e-qc \\
    -qc & b+qa
\end{pmatrix}.\]
Recall that we are aiming for a basis on this second summand of the form
\[\begin{pmatrix}
    b' & a' \\
    0 & c'
\end{pmatrix},\]
with $(a',b',c',e')$ satisfying the conditions of being a prototype.

We can make the lower left entry of $N$ equal to 0 and the lower right entry equal to $\gcd(qc,b+qa)$ by applying \textit{basis reduction}. That is, for a basis matrix $A = (a_{ij})_{i,j = 1,2}$, if $|a_{21}| \geq |a_{22}|$, let $u = -\text{sgn}(a_{21}\cdot a_{22})$ and apply $\begin{pmatrix}1 & 0 \\ u & 1 \end{pmatrix}$ \textit{on the right}, otherwise apply $\begin{pmatrix}1 & u \\ 0 & 1 \end{pmatrix}$ \textit{on the right}. Terminate the process when the lower left entry is 0. We will be left with the lower right entry equal to $\gcd(a_{21},a_{22})$. Following this, an application \textit{on the right} of $\begin{pmatrix} 1 & n \\ 0 & 1 \end{pmatrix}$, for an appropriate choice of $n$, will allow us to reduce the upper right entry modulo the upper left (which will always be positive in our case). Since the action is on the right, all of these operations are just change of basis operations within the homology of the surface. We are not applying elements of $\GL^+(2,\R)$ to our surface here.

Applying this process to $N$, we obtain a basis of the form
\[\begin{pmatrix}
    b' & a^*\\
    0 & c'
\end{pmatrix}\]
with $c' = \gcd(qc,b+qa)$ and with $0\leq a^*< b'$. Finally, by applying powers of $\begin{pmatrix} 1 & 1 \\ 0 & 1 \end{pmatrix}$ on the left (so acting on the surface by $\GL^+(2,\R)$) and also on the right (as a change of basis operator), we can obtain a basis of the form
\[\begin{pmatrix}
    b' & a' \\
    0 & c'
\end{pmatrix}\]
with $0 \leq a' < \gcd(b',c')$ and the remaining prototype conditions for $(a',b',c',e')$ being satisfied.

If $q = \infty$, then (see the right of Figure~\ref{fig:F-bases}) we have, with $e' = -e -2c$ and $\lambda' = (e'+\sqrt{D})/2$, 
\[L_{1} = \begin{pmatrix}
    a & \lambda - b \\ 
    c & 0
\end{pmatrix}, L_{2} = \begin{pmatrix}
    a & -\lambda \\
    \lambda + c & 0
\end{pmatrix}, \text{ and }N = \lambda'L_{1}^{-1}L_{2} = \begin{pmatrix}
    b-e-c & 0 \\
    a & c
\end{pmatrix}.\]
Note that, if $a = 0$, then we have $a' = 0$, $b' = b-e-c$ and $c' = c$. Otherwise, we again apply the basis reduction and horizontal shearing procedure described above to get a basis
\[\begin{pmatrix}
    b' & a' \\
    0 & c'
\end{pmatrix}\]
with $c' = \gcd(a,c)$ and $0\leq a' < \gcd(b',c')$.

Let us consider the example of the prototype $(1,24,2,2)$. We can determine that $D = e^2+4bc = 296 = 14^2$ and that, for $q\geq 1$, $(e+2qc)^2 = (2+4q)^2 < 14^2$ if and only if $q = 1,2$. So the set of admissible $q$ is $\{1,2,\infty\}$. We also have $\lambda = (e+\sqrt{D})/2 = 8$. Let us apply the butterfly move $B_{2}$. From McMullen's formulae, we know that $e' = -e-2qc = -10$ and $\lambda' = (e'+\sqrt{D})/2 = 2$. The initial basis matrices are 
\[L_1 = \begin{pmatrix}
    b+qa & \lambda + a \\
    qc & c
\end{pmatrix} = \begin{pmatrix}
    26 & 9 \\
    4 & 2
\end{pmatrix}\,\,\,\text{and}\,\,\,L_{2} = \begin{pmatrix}
    \lambda & b + qa \\
    0 & \lambda + qc
\end{pmatrix} = \begin{pmatrix}
    8 & 26 \\
    0 & 12
\end{pmatrix}.\]
The matrix $N = \lambda'L_{1}^{-1}L_{2}$ considered above is
$$N = \begin{pmatrix}
    c & -a-e-qc \\
    -qc & b+qa
\end{pmatrix} = \begin{pmatrix}
    2 & -7 \\
    -4 & 26
\end{pmatrix}.$$
Applying basis reduction gives
\[\begin{pmatrix}
    2 & -7 \\
    -4 & 26
\end{pmatrix}\rightarrow\begin{pmatrix}
    2 & -7 \\
    -4 & 26
\end{pmatrix}\begin{pmatrix} 1 & 6 \\ 0 & 1\end{pmatrix} = \begin{pmatrix}
    2 & 5 \\
    -4 & 2
\end{pmatrix}\rightarrow\begin{pmatrix}
    2 & 5 \\
    -4 & 2
\end{pmatrix}\begin{pmatrix} 1 & 0 \\ 2 & 1\end{pmatrix}=\begin{pmatrix}
    12 & 5 \\
    0 & 2
\end{pmatrix}.\]
Finally, we apply $\begin{pmatrix} 1 & -2 \\ 0 & 1\end{pmatrix}$ to our surface and obtain the basis
\[\begin{pmatrix} 1 & -2 \\ 0 & 1\end{pmatrix}\begin{pmatrix}
    12 & 5 \\
    0 & 2
\end{pmatrix} = \begin{pmatrix} 12 & 1 \\ 0 & 2\end{pmatrix}.\]
Hence, $B_2(1,24,2,2) = (1,12,2,-10)$.

\begin{remark}\label{rem:GL-vs-SL}
    Notice that the butterfly move in this example is carried out by the matrix
    \[\begin{pmatrix} 1 & -2 \\ 0 & 1\end{pmatrix}\cdot\lambda'L_{1}^{-1} = \frac{1}{8}\begin{pmatrix} 10 & -61 \\ -4 & 26\end{pmatrix}\in\GL^{+}(2,\R).\]
    In particular, notice that this matrix is not in $\SL(2,\Z)$. This will be true in general since the prototypical splittings for $(a,b,c,e)$ and $(a',b',c',e')$ have areas $\lambda^2 + bc = \lambda\sqrt{D}$ and $(\lambda')^2 + b'c' = \lambda'\sqrt{D}$ which will not typically be equal. So, the determinant of the butterfly move matrix $\frac{\lambda'}{\lambda}$ will not typically be one. In fact, you can check that this happens only when $e < 0$ and $q = -e/c$ or $e = -c$ and $q = \infty$. This means that we will not be able to move around the $\SL(2,\Z)$-orbits by applying butterfly move matrices directly.
\end{remark}

\subsection{An outline of McMullen's proof}

Now that we have setup the butterfly move machinery, we will remind the reader of how McMullen uses butterfly moves in his classification of the $\GL^{+}(2,\R)$-orbits in $\calH(2)$. In the next section, we describe how to turn this proof into a way of moving within an $\SL(2,\Z)$-orbit of an origami.

McMullen's proof can be summarised along the following lines:
\begin{enumerate}
    \item Show that every translation surface in $\Omega W_{D}$ has a prototypical splitting corresponding to some prototype $(a,b,c,e)\in\mathcal{P}_{D}$ in its $\GL^{+}(2,\R)$-orbit.
    \item Prove that there exists a sequence of butterfly moves taking this prototype $(a,b,c,e)$ to some reduced prototype $(0,b',1,e')\in \mathcal{S}_{D}$.
    \item Finally, show that all reduced prototypes in $\mathcal{S}_{D}$ can be connected by butterfly moves (or that $\mathcal{S}_{D}$ has two components when $D\equiv 1 \!\!\mod 8$). 
\end{enumerate}

\subsection{The implied $\boldsymbol{\SL(2,\Z)}$-orbit algorithm}

Here, we describe our algorithm for using the above proof of McMullen to connect origamis in an $\SL(2,\Z)$-orbit.

We consider some $n$-squared origami $X$ and will connect it to $$\left(\left(\frac{n}{2}+1,\frac{n}{2}+2,\ldots,n\right),\left(1,\frac{n}{2}+1,\frac{n}{2},\frac{n}{2}-1,\ldots,2\right)\right)$$ if $n$ is even; to $$\left(\left(\frac{n+1}{2}+1,\frac{n+1}{2}+2,\ldots,n\right),\left(1,\frac{n+1}{2}+1,\frac{n+1}{2},\frac{n+1}{2}-1,\ldots,2\right)\right)$$ if $n \equiv 1\!\!\mod 4$ with HLK-invariant $(2,[1,1,1])$ or if $n \equiv 3\!\!\mod 4$ with HLK-invariant $(0,[3,1,1])$; or to $$\left(\left(\frac{n-1}{2}+1,\frac{n-1}{2}+2,\ldots,n\right),\left(1,\frac{n-1}{2}+1,\frac{n-1}{2},\frac{n-1}{2}-1,\ldots,2\right)\right)$$ if $n \equiv 1\!\!\mod 4$ with HLK-invariant $(0,[3,1,1])$ or if $n \equiv 3\!\!\mod 4$ with HLK-invariant $(2,[1,1,1])$.

The algorithm will go as follows:
\begin{enumerate}
    \item If $X$ has one horizontal cylinder, move in the cusp to the cusp representative. This origami now has two cylinders in the vertical direction (see Figure~\ref{fig:H2-1-cyl}). Apply rotation by $\pi/2$, $R = ST^{-1}S$, to move to a two-cylinder origami.
    \item Determine the prototype $(a,b,c,e)$ corresponding to this two-cylinder cusp.
    \item From McMullen's proof, there exists a sequence of butterfly moves $B_{q}$ taking this prototype to a reduced prototype $(0,b',1,e')$. For each butterfly move, \textit{shadow the butterfly move} inside the $\SL(2,\Z)$-orbit. We describe how to do this in the following subsection.
    \item Finally, shadow a sequence of butterfly moves to arrive at the reduced prototype $(0,\frac{n^{2}}{4},1,0)$ if $n$ is even; or to one of $(0,\frac{n^2-1}{4},1,\pm1)$ if $n$ is odd. This will leave us inside the cusp of one of the origamis listed above. We finish by moving inside the cusp to arrive at the target origami.
\end{enumerate}

That the origamis listed above correspond to the prototypes $(0,\frac{n^{2}}{4},1,0)$, $(0,\frac{n^2-1}{4},1,1)$ and $(0,\frac{n^2-1}{4},1,-1)$, respectively, follows from Lemma~\ref{lem:H2-red} below. We shall call these prototypes the \textit{target prototypes} and the associated origamis the \textit{target origamis}.

\subsection{Shadowing butterfly moves inside an $\boldsymbol{\SL(2,\Z)}$-orbit}\label{subsec:shadow}

Recall from Remark~\ref{rem:GL-vs-SL} that the butterfly move matrices applied to prototypical splittings will rarely live in $\SL(2,\Z)$. We now describe, via a concrete example, how to carry out the \textit{shadowing of a butterfly move} alluded to above. This applies a sequence of matrices in $\SL(2,\Z)$ to move between a pair of two-cylinder origamis that correspond to prototypes being acted upon by a butterfly move.

Consider the origami $$O = ((1,2)(3,4)(5,6)(7,8)(9,10,11,12,13,14),(1,3,5,7,9)(2,4,6,8,14,13,12,11,10))$$ shown on the left of Figure~\ref{fig:sts-example-1}. We must first determine the corresponding prototype.

\begin{figure}[b]
    \centering
    \begin{tikzpicture}[scale = 0.8, line width = 1.5pt]
        \draw (0,0) -- (2,0) -- (3,4) -- (7,4) -- (8,5) -- (2,5) -- (1,4) -- cycle;
        \draw[gray,dashed] (1,0) -- (1,4);
        \draw[gray,dashed] (2,0) -- (2,5);
        \foreach \i in {3,4,5,6,7}{
            \draw[gray,dashed] (\i,4) -- (\i,5);
        }
        \foreach \j in {1,2,3,4}{
            \draw[gray,dashed] (\j*0.25,\j) -- (\j*0.25+2,\j);
        }
        \foreach \x/\y in {0/0,2/0,3/4,7/4,8/5,4/5,2/5,1/4}{
            \node at (\x,\y) {$\bullet$};
        }
        \node[gray] at (0.5,0.5) {\footnotesize $1$};
        \node[gray] at (1.5,0.5) {\footnotesize $2$};
        \node[gray] at (0.5,1.5) {\footnotesize $3$};
        \node[gray] at (1.5,1.5) {\footnotesize $4$};
        \node[gray] at (1.5,2.5) {\footnotesize $6$};
        \node[gray] at (2.5,2.5) {\footnotesize $5$};
        \node[gray] at (1.5,3.5) {\footnotesize $8$};
        \node[gray] at (2.5,3.5) {\footnotesize $7$};
        \node[gray] at (1.5,4.5) {\footnotesize $14$};
        \node[gray] at (2.5,4.5) {\footnotesize $9$};
        \foreach \k in {10,11,12,13}{
            \node[gray] at (\k-6.5,4.5) {\footnotesize $\k$};
        }
        \draw (0+9,0) -- (2+9,0) -- (3+9,4) -- (7+9,4) -- (7+9,5) -- (1+9,5) -- (1+9,4) -- (0+9,0);
        \draw[gray,dashed] (1+9,0) -- (1+9,4);
        \draw[gray,dashed] (2+9,0) -- (2+9,5);
        \foreach \i in {2,3,4,5,6}{
            \draw[gray,dashed] (\i+9,4) -- (\i+9,5);
        }
        \foreach \j in {1,2,3,4}{
            \draw[gray,dashed] (\j*0.25+9,\j) -- (\j*0.25+2+9,\j);
        }
        \foreach \x/\y in {9/0,11/0,12/4,16/4,16/5,12/5,10/5,10/4}{
            \node at (\x,\y) {$\bullet$};
        }
        \node[gray] at (0.5+9,0.5) {\footnotesize $1$};
        \node[gray] at (1.5+9,0.5) {\footnotesize $2$};
        \node[gray] at (0.5+9,1.5) {\footnotesize $3$};
        \node[gray] at (1.5+9,1.5) {\footnotesize $4$};
        \node[gray] at (1.5+9,2.5) {\footnotesize $6$};
        \node[gray] at (2.5+9,2.5) {\footnotesize $5$};
        \node[gray] at (1.5+9,3.5) {\footnotesize $8$};
        \node[gray] at (2.5+9,3.5) {\footnotesize $7$};
        \node[gray] at (1.5+9,4.5) {\footnotesize $14$};
        \node[gray] at (2.5+9,4.5) {\footnotesize $9$};
        \foreach \k in {10,11,12,13}{
            \node[gray] at (\k-6.5+9,4.5) {\footnotesize $\k$};
        }
        \draw[line width = 0pt, fill = blue, opacity = 0.3] (12,4) -- (15,5) -- (16,5) -- (13,4) -- (12,4);
        \draw[line width = 0pt, fill = blue, opacity = 0.3] (15,4) -- (16,4+1/3) -- (16,4) -- (15,4);
        \draw[line width = 0pt, fill = blue, opacity = 0.3] (10,4) -- (10,4+1/3) -- (12,5) -- (13,5) -- (10,4);
        \draw[purple] (10,4) -- (13,5);
        \draw[purple] (13,4) -- (16,5);
    \end{tikzpicture}
    \caption{The origami $O$ and its cusp representative. The direction corresponding to the butterfly move $B_{2}$ is also shown, with the resulting torus $F_{1}$ shaded in blue.}
    \label{fig:sts-example-1}
\end{figure}
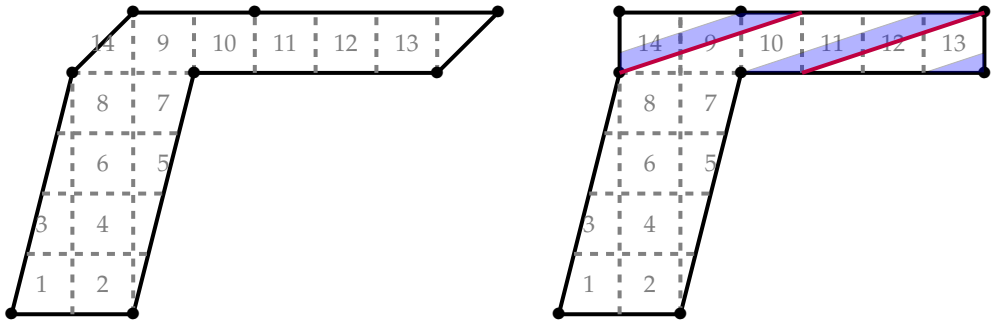

\begin{figure}[b]
    \centering
    \begin{tikzpicture}[scale = 0.5, line width = 1.5pt]
        \draw (0,0) -- (4,0) -- (6,4) -- (14,4) -- (14,5) -- (2,5) -- (2,4) -- cycle;
        \node at (2,-0.5) {$\lambda = 4$};
        \node at (-0.25,2) {$\lambda = 4$};
        \node[left] at (2,4) {$(0,0)$};
        \node[below] at (14,4) {$(12,0)$};
        \node[above] at (2,5) {$(0,1)$};
        \foreach \x/\y in {0/0,4/0,6/4,14/4,14/5,6/5,2/5,2/4}{
            \node at (\x,\y) {$\bullet$};
        }
    \end{tikzpicture}
    \begin{tikzpicture}[scale = 0.5, line width = 1.5pt]
        \draw (0,0-14) -- (8,0-14) -- (8,8-14) -- (24,8-14) -- (25,10-14) -- (1,10-14) -- (0,8-14) -- (0,0-14);
        \node at (4,-0.25-14-0.25) {$\lambda = 8$};
        \node at (-0.25-1,4-14) {$\lambda = 8$};
        \node[left] at (0,8-14) {$(0,0)$};
        \node[below] at (24,8-14) {$(24,0)$};
        \node[above] at (1,10-14) {$(1,2)$};
        \foreach \x/\y in {0/-14,8/-14,8/-6,24/-6,25/-4,9/-4,1/-4,0/-6}{
            \node at (\x,\y) {$\bullet$};
        }
        \draw[purple] (0,-6) -- (13,-4);
        \draw[purple] (12,-6) -- (25,-4);
    \end{tikzpicture}
    \caption{Realising the prototype corresponding to $O$. The direction corresponding to the butterfly move $B_{2}$ is also shown.}
    \label{fig:sts-example-2}
\end{figure}
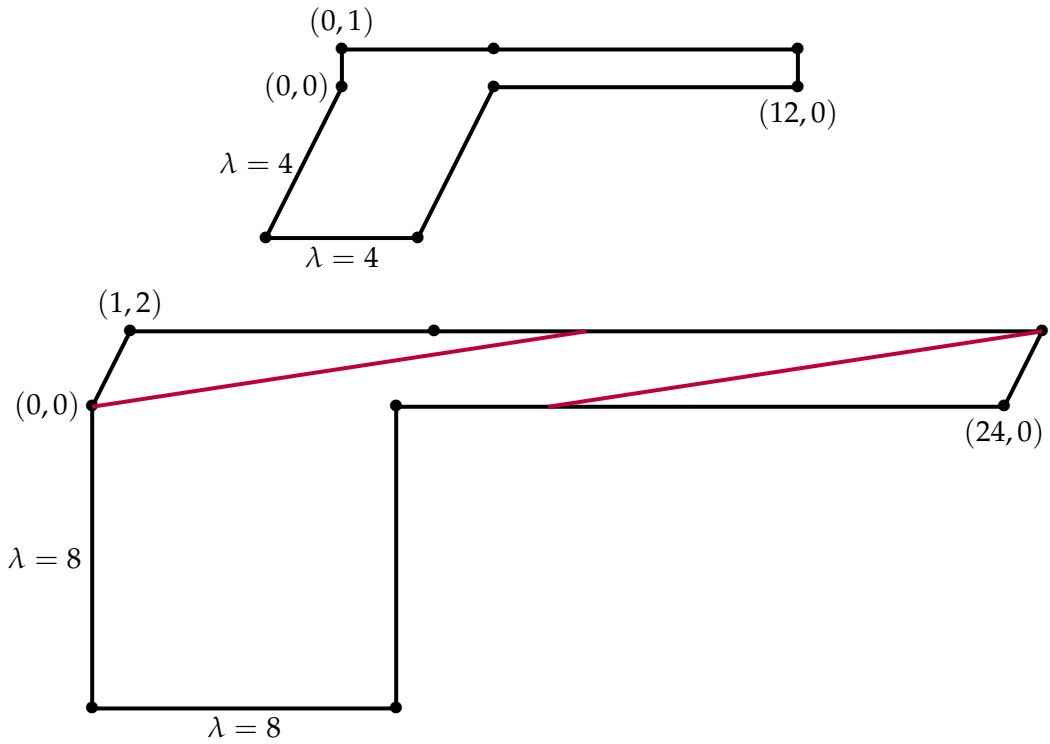

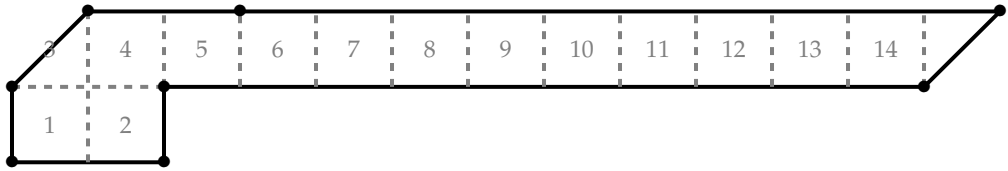
\begin{figure}[t]
    \centering
    \begin{tikzpicture}[scale = 1, line width = 1.5pt]
        \draw (0,0) -- (2,0) -- (2,1) -- (12,1) -- (13,2) -- (1,2) -- (0,1) -- cycle;
        \draw[gray,dashed] (1,0) -- (1,2);
        \draw[gray,dashed] (0,1) -- (2,1);
        \foreach \i in {2,...,12}{
            \draw[gray,dashed] (\i,1) -- (\i,2);
        }
        
        \foreach \x/\y in {0/0,2/0,2/1,12/1,13/2,3/2,1/2,0/1}{
            \node at (\x,\y) {$\bullet$};
        }
        \node[gray] at (0.5,0.5) {\footnotesize $1$};
        \node[gray] at (1.5,0.5) {\footnotesize $2$};
        \foreach \k in {3,...,14}{
            \node[gray] at (\k-2.5,1.5) {\footnotesize $\k$};
        }
    \end{tikzpicture}
    \caption{The cusp representative corresponding to the prototype $(1,12,2,-10)$.}
    \label{fig:sts-example-3}
\end{figure}

First, we find the cusp representative (i.e., the origami in the $T$-orbit with $0\leq t_i < \gcd(w_{i},h_{i})$). Here, we have $(w_1,h_1,t_1,w_2,h_2,t_2) = (2,4,1,6,1,1)$. Recalling that $T^{k}$ sends $t_i$ to $t_i + k\,h_i\!\!\mod w_i$, we arrive at the cusp representative by applying $T^{-1}$ to get the origami $$((1,2)(3,4)(5,6)(7,8)(9,10,11,12,13,14),(1,3,5,7,9,2,4,6,8,14))$$ shown on the right of Figure~\ref{fig:sts-example-1}. Indeed, $t_1$ is fixed and $t_2$ is reduced by 1.

To realise the surface in the form of a prototypical splitting as in Figure~\ref{fig:prototype}, we must make the bottom cylinder square. So we apply the matrix $\begin{pmatrix}
    2 & 0 \\ 
    0 & 1
\end{pmatrix}$. The resulting surface is shown in the top of Figure~\ref{fig:sts-example-2}. We now also need to undo the twist in the bottom cylinder. We would need to apply $\begin{pmatrix}
    1 & -\frac{1}{2} \\ 
    0 & 1
\end{pmatrix}$ to do this. However, this will not preserve the integer values of the vertices of the top cylinder. So we must first scale the surface by $\begin{pmatrix}
    2 & 0 \\ 
    0 & 2
\end{pmatrix}$. We must also then normalise the twist of the top cylinder by applying powers of $T$ to satisfy the requirements on the value $a$ in a prototype. The resulting surface is shown in Figure~\ref{fig:sts-example-2}. At this point, we must check that $e^{2}+4bc = D = n^{2}$. Here, this is true. However, it is possible that the surface requires further scaling before this condition is achieved. We then see that the prototype of the surface is $(1,24,2,2)$ with admissible $q$ values $\{1,2,\infty\}$. This is the prototype of the example considered at the end of Subsection~\ref{subsec:proto}.

Suppose that we want to shadow the butterfly move $B_{2}$. The new splitting interval lies in the direction $(b+2a,2c)$ inside the prototypical splitting surface. Inside the cusp representative of the original origami, this is the direction $(w_{2}+2t_{2},2h_{2}) = (6,2)$. If we make this direction horizontal, then the resulting two-cylinder origami will lie in the cusp associated to the prototype $B_2(1,24,2,2) = (1,12,2,-10)$. Indeed, a butterfly move always makes this direction horizontal before then applying some further scaling and shearing. To make this direction horizontal in the $\SL(2,\Z)$-orbit, we apply $T^{-2}$ followed by $S^{-1}$ and obtain the origami $$((1,2)(3,4,5,6,7,8,9,10,11,12,13,14),(1,3,10,5,12,7,14,9,2,4,11,6,13,8))$$ whose cusp representative
$$((1,2)(3,4,5,6,7,8,9,10,11,12,13,14),(1,3,14,13,12,11,10,9,8,7,6,5,2,4))$$
is shown in Figure~\ref{fig:sts-example-3}. It can be checked that, after applying $\begin{pmatrix}
    1 & 0 \\ 
    0 & 2
\end{pmatrix}$, the corresponding prototype is indeed $(1,12,2,-10)$, as desired.

So the full butterfly move $B_{2}$ was shadowed by applying $T^{-1}$ to reach the cusp representative followed by $S^{-1}\circ T^{-2}$ in order to make the butterfly move direction horizontal and in doing so we landed in the cusp associated to the image prototype.

Such a procedure works in general:
\begin{enumerate}
    \item Move to the cusp representative.
    \item Find the butterfly move direction $(w_{2}+qt_{2},qh_{2})$.
    \item Apply a Euclidean algorithm procedure to make this direction horizontal; i.e., if the vector $(u,v)$ has $u>v$ then apply $T^{-1}$ otherwise apply $S^{-1}$. We end up in the cusp associated to the image prototype.
\end{enumerate}

\section{Diameter bounds}\label{sec:H2-McM}

Here, we bound the number of applications of $T$ or $S$ required to carry out the algorithm of the previous section. In doing so, we obtain diameter bounds for the $\SL(2,\Z)$-orbit graphs.

We will first bound the number of butterfly moves that we will be required to shadow. After this, we bound the number of applications of $T$ or $S$ required to shadow each butterfly move.

\subsection{Bounding the number of butterfly moves}\label{subsec:butterfly-bound}

Here we will argue that McMullen's algorithm requires $O(\log n)$ butterfly moves to reach a reduced prototype and then a further $O(n)$ butterfly moves to reach any target reduced prototype.

\subsubsection{Connecting to reduced prototypes}

Consider a prototype $(a,b,c,e)$. We will argue that $O(\log n)$ repeated applications of $B_{1}$ will reach a prototype of the form $(0,b,1,e)$; i.e., a reduced prototype. This will follow from the fact that at most three applications of $B_1$ achieves a prototype $(a',b',c',e')$ with $c' \leq c/2$.

\indent\paragraph{\textbf{The first application of $\boldsymbol{B_1}$}} After applying $B_{1}$ we obtain the prototype $(a',b',c',e')$ with $c' = \gcd(b+a,c)$. Now either $c' < c$, in which case $c' \leq c/2$ as desired, or $c' = c$. In the latter case, the basis reduction in the butterfly move is carried out as follows:
\[N = \begin{pmatrix}
    c & -a-e-qc \\
    -qc & b+qa
\end{pmatrix} = \begin{pmatrix}
    c & -a-e-c \\
    -c & b+a
\end{pmatrix}\rightarrow\begin{pmatrix}
    b-e-c & -c \\
    0 & c
\end{pmatrix} \rightarrow\begin{pmatrix}
    b-e-c & 0 \\
    0 & c
\end{pmatrix},\]
where the first ``$\rightarrow$" denotes the initial basis reduction and the second corresponds to the reduction $-c \equiv 0 \mod \gcd(c,b-e-c)$. So we have either arrived at a prototype $(a',b',c',e')$ with $c' \leq c/2$, as required, or at a prototype of the form $(0,b',c',e') = (0,b-e-c,c,-e-2c)$.

\indent\paragraph{\textbf{The second application of $\boldsymbol{B_1}$}} Consider the prototype $(0,b',c',e') = (0,b-e-c,c,-e-2c)$ from the previous step. We apply $B_{1}$ again to obtain a prototype $(a'',b'',c'',e'')$. Again, we either have $c'' = \gcd(b',c') < c'=c$, giving us $c'' \leq c/2$, or we have $c'' = c' = c$. In the first case, we are done. In the latter case, by a similar argument to the above, we will have that $(a'',b'',c'',e'') = (0,b'-e'-c',c',-e'-2c')$.

\indent\paragraph{\textbf{The third application of $\boldsymbol{B_{1}}$}} Consider the prototype $$(a'',b'',c'',e'') = (0,b'-e'-c',c',-e'-2c')$$ in the previous step. Recall from the properties of prototypes that $\gcd(a'',b'',c'',d'') = 1$. Therefore, $$1 = \gcd(a'',b'',c'',d'') = \gcd(0,b'-e'-c',c',-e'-2c') = \gcd(c',e'),$$ where the last equality follows from the fact that $\gcd(b',c') = c'$. This then gives us that $\gcd(b'',c'') = \gcd(b'-e'-c',c') = \gcd(c',e') = 1$ and so a third application of $B_{1}$ will arrive at a prototype of the form $(0,b''',1,e''')$. Trivially, $1\leq c/2$ if $c$ was not already equal to 1. \\

To summarise, we can apply $B_{1}$ at most three times to arrive at a prototype whose $c$ value has been divided by at least two. Hence, we can arrive at a reduced prototype in at most $O(\log c)$ applications of $B_{1}$. Now, $n^{2} = D = e^{2}+4bc > c$, so that $\log c = O(\log n)$. Hence, this process requires $O(\log n)$ butterfly moves, all of which are $B_{1}$.

\subsubsection{Connecting reduced prototypes to the target prototype}

The reduced prototypes of discriminant $D$, $\mathcal{S}_{D} := \{(0,b,1,e)\,:\,e^2 + 4b = D\}$, are parameterised by the set $$\{e\equiv D\!\!\mod 2\,:\,e^2 < D\,\,\text{and}\,\,(e+2)^2 < D\}.$$ We will abuse notation and refer to the latter also as $S_{D}$ when there is no confusion. As a result of this parameterisation, and since $D = n^2$ in our case, there are $O(\sqrt{D}) = O(n)$ reduced prototypes. McMullen equips $\mathcal{S}_{D}$ with the equivalence relation produced by butterfly moves that preserve being reduced. That is, $e\sim e'$ if there exists a sequence $e = e_{0}, e_{1},\ldots, e_{k} = e'$, so that $e_{i}\in \mathcal{S}_{D}$ for all $0\leq i\leq k$ and for all $0\leq i \leq k-1$ there exists some $q>0$ such that
$$(0,b_{i+1},1,e_{i+1}) = B_{q}(0,b_{i},1,e_{i})\,\,\text{or}\,\,(0,b_{i},1,e_{i}) = B_{q}(0,b_{i+1},1,e_{i+1})$$
With this relation in mind, McMullen proved the following theorem.

\begin{theorem}[{\cite[Theorem 10.1]{McM}}]
    Assume $D\geq 5$ and $D\neq 9, 49, 73, 121$ or $169$. Then $S_D$ has exactly two components when $D \equiv 1 \!\!\mod 8$, and otherwise just one.
\end{theorem}

The two components for $D\equiv 1\!\!\mod 8$ correspond to the orbits distinguished by the HLK-invariant (Hubert--Leli\`evre's $A$ and $B$ orbits).

So we see that, in the general case, each reduced prototype can be connected to the target prototype by a path of undirected butterfly moves between reduced prototypes. Note that, even though butterfly moves are directed, the underlying $\GL^{+}(2,\R)$ matrix (and, for our sake, the shadowing $\SL(2,\Z)$ moves) can be performed in the reverse order. There are $O(n)$ prototypes in each component of $\mathcal{S}_{D}$ and so we require at most $O(n)$ butterfly moves to connect any reduced prototype to the target prototype.

\subsection{The resulting diameter bound}

We now cost the number of applications of $T$ or $S$ required to connect a given origami $X$ to one of the target origamis.

If $X$ has one cylinder, we send it to a two-cylinder origami by first travelling in the cusp to the cusp representative --- this takes $O(n)$ applications of $T$ --- and then we rotate the origami by $\pi/2$ by applying $R = ST^{-1}S$. So, $X$ is taken to a two-cylinder origami in $O(n)$ applications of $T$ or $S$.

The two-cylinder origami is associated to some prototype $(a,b,c,e)$. From the previous subsection, we know that $O(\log n)$ applications of $B_1$ sends this prototype to a reduced prototype. We must cost, in $T$ and $S$, the process of shadowing these butterfly moves.

As discussed in Subsection~\ref{subsec:shadow}, to shadow a butterfly move $B_{q}$ we move inside the two-cylinder cusp to reach the cusp representative --- we do this naively by applying powers of $T$. This process then takes $O(n^{2})$ moves since the cusp has width
\[\text{lcm}\left(\frac{w_{1}}{\gcd(w_{1},h_{1})},\frac{w_{2}}{\gcd(w_{2},h_{2})}\right)\]
and each of $w_1$ and $w_2$ are $O(n)$. This is then followed by a Euclidean algorithm operation to make the butterfly move direction $(w_{2}+qt_{2},qh_{2})$ (or $(t_{2},h_{2})$ for $q = \infty$) horizontal --- this takes $O(\max\{w_{2}+qt_{2},qh_{2}\})$ (or $O(h_2)$ for $q = \infty$) applications of $T$ or $S$. In our case, since we are applying $B_{1}$ and have $\max\{w_{2}+t_{2},h_{2}\} = O(n)$, this process takes $O(n)$ applications of $T$ or $S$. So each $B_1$ takes $O(n^2) + O(n) = O(n^2)$ applications of $T$ or $S$ to shadow.

Since we shadow $B_{1}$ $O(\log n)$ times, the whole process takes $O(n^{2}\log n)$ steps.

Now, we must cost the shadowing of the butterfly moves required to connect the reduced prototypes to the target reduced prototype. Above, we argued that there are at most $O(n)$ butterfly moves required to carry out this step.

Note that, a priori, we might expect that each of these $O(n)$ butterfly moves also costs $O(n^{2})$ applications of $T$ or $S$ to shadow, as above. If this was true, then we would only obtain the trivial diameter bound $O(n^3)$ (recall that $|V| = \Theta(n^{3})$). Thankfully, the following lemma proves that the cusps of two-cylinder origamis corresponding to reduced prototypes have width only $O(n)$.

\begin{lemma}\label{lem:H2-red}
    A primitive two-cylinder $n$-squared origami in $\calH(2)$ with surface parameters $(w_1,h_1,t_1,w_2,h_2,t_2)$ corresponding to a reduced prototype $(0,b,1,e)$ has cusp width equal to $w_{2} = O(n)$. In fact, $(w_1,h_1,w_2,h_2) = (1,\frac{n+e}{2},\frac{n-e}{2},1)$.
\end{lemma}

\begin{proof}
    Consider Figure~\ref{fig:reduced-proto}. The prototype $(0,b,1,e)$ has $\lambda = \frac{e + \sqrt{D}}{2} = \frac{e+n}{2}$. The surface therefore has area $\lambda^{2}+b = \left(\frac{e+n}{2}\right)^{2}+b = \frac{en+n^{2}}{2} = n\lambda$. Recall that the prototype is obtained from the origami of area $n$ by scaling and shearing. Suppose that the origami was vertically scaled by $L$. Then the horizontal scaling factor is $\frac{\lambda}{L}$, and we have
    \[\lambda = Lh_{1},\,\,\lambda = \frac{\lambda}{L}w_{1},\,\,1 = Lh_{2},\,\,b = \frac{\lambda}{L}w_{2},\]
    from which we obtain that $1 = w_{1}h_{2}$. Recall that our surface parameters $w_i$ and $h_i$ are integers and that $\lambda = \frac{e+n}{2}$ is also an integer since $e\equiv n \!\!\mod 2$. Hence, the origami has parameters $w_1 = 1$, $h_1 = \lambda$, $w_{2} = \frac{b}{\lambda} = \frac{n-e}{2}$ and $h_{2} = 1$. Therefore, it lies in a cusp of width
    \[\text{lcm}\left(\frac{w_{1}}{\gcd(w_{1},h_{1})},\frac{w_{2}}{\gcd(w_{2},h_{2})}\right) = \text{lcm}\left(\frac{1}{\gcd(1,\lambda)},\frac{w_{2}}{\gcd(w_{2},1)}\right) = w_{2} = O(n).\]
\end{proof}

\begin{figure}[t]
    \centering
    \begin{tikzpicture}[scale = 0.381, line width = 1.5pt]
        \draw (0,0) -- (4,0) -- (4,4) -- (8,4) -- (8,5) -- (0,5) -- (0,4) -- cycle;
        \node at (2,-1) {$\lambda = \frac{e+n}{2}$};
        \node at (-2,2) {$\lambda = \frac{e+n}{2}$};
        \node[left] at (0,4) {$(0,0)$};
        \node[below right] at (8,4) {$(b,0)$};
        \node[above left] at (0,5) {$(0,1)$};
        \foreach \x/\y in {0/0,4/0,4/4,8/4,8/5,4/5,0/5,0/4}{
            \node at (\x,\y) {$\bullet$};
        }
    \end{tikzpicture}
    \caption{The prototypical splitting corresponding to prototype $(0,b,1,e)$.}
    \label{fig:reduced-proto}
\end{figure}

We briefly remark that, for a reduced prototype, all finite admissible $q$ satisfy $q = O(n)$. Indeed, recall that a finite $q$ is admissible if and only if $(e+2qc)^2 < D$. In this case, since $c = 1$ and $D = n^2$, we need $(e+2q)^2 < n^2$. Solving the equality for the positive solution gives $q = \frac{n-e}{2}$, which is indeed positive since $-n < e < n$. So the largest admissible $q$ is then $\frac{n-e}{2} - 1 = O(n)$. In fact, this shows that the largest admissible $q$ is at most $n-2$.

So, when shadowing butterfly moves between reduced prototypes, the cusp representative with $t_{2} = 0$ can be arranged in $O(n)$ applications of $T$ and, since $\max\{w_2+qt_2,qh_2\} = \max\{w_2,q\} = O(n)$, the Euclidean algorithm part of shadowing the butterfly move only costs $O(n)$ applications of $T$ or $S$. So each butterfly move between reduced prototypes costs $O(n)$ to shadow and we make $O(n)$ such butterfly moves. Hence, reaching the target origami through cusps corresponding to reduced prototypes will cost $O(n^2)$.

Therefore, totalling all of the above steps, we obtain the diameter bound $O(n) + O(n^2\log n) + O(n^2) = O(n^{2}\log n) = O(|V|^{\frac{2}{3}}\log |V|)$, and Theorem~\ref{thm:main-H2} is proved.

\begin{remark}
    We wish to highlight that the most expensive part of the algorithm is moving through the two-cylinder cusps of width $O(n^2)$. At the moment, we just naively apply powers of $T$ to reach the cusp representative. It is entirely possible that there are shortcuts in the graph that allow us to reach the cusp representative by moving outside of the cusp. We expect such an argument to be crucial to any further improvements on diameter bounds. We direct the reader to Appendix~\ref{app:improvements} for further discussion of possible improvements.
\end{remark}

\section{Prym loci in $\calH(4)$}\label{sec:H4}

The Prym locus $\Omega\mathcal{E}_{D}(4)$ is the subset of $\calH(4)$ consisting of those $(M,\omega)$ for which $M$ admits a holomorphic involution $\iota$ with 4 fixed points taking $\omega$ to $-\omega$, and admitting real multiplication by $\mathcal{O}_{D}$ with $\mathcal{O}_{D}\cdot\omega\subset\C\cdot\omega$. Lanneau--Nguyen~\cite{LN14} classify the $\GL^{+}(2,\R)$ connected components of $\Omega\mathcal{E}_{D}(4)$ and we direct the reader to their paper for more details on Prym loci. They prove that, for $D\geq 17$,  $\Omega\mathcal{E}_{D}(4)$ is non-empty if and only if $D\equiv 0,1,4\!\!\mod\! 8$.

Lanneau--Nguyen, as a consequence of their determination of the connected components of $\Omega\mathcal{E}_{D}(4)$, give the following classification of the $\SL(2,\Z)$-orbits of Prym origamis in $\calH(4)$.

\begin{theorem}[{\cite[Proposition B.1 and Corollary B.2]{LN14}}]
    Fix $n\geq 5$ and let $X$ be a primitive $n$-squared origami that is a Prym eigenform in $\calH(4)$. Then
    \begin{itemize}
        \item if $n$ is odd, there is a single $\SL(2,\Z)$-orbit of such origamis and the eigenforms have discriminant $D = n^{2}$.
        \item if $n\equiv 0\mod 4$ or $n = 6$, there is a single $\SL(2,\Z)$-orbit of such origamis and the eigenforms have discriminant $D = n^{2}$.
        \item if $n\equiv 2\mod 4$, $n\geq 10$, there are two $\SL(2,\Z)$-orbits of such origamis: one containing eigenforms of discriminant $D = n^{2}$ and one containing eigenforms of discriminant $D = \frac{n^{2}}{4}$.
    \end{itemize}
\end{theorem}

In the language of HLK-invariants, when $n$ is odd the origamis have HLK-invariant $(0,[1,1,1])$; when $n\equiv 0\mod 4$ or $n = 6$ the origamis have HLK-invariant $(1,[2,0,0])$; and when $n\equiv 2\mod 4$, $n\geq 10$, then the origamis of discriminant $n^2$ have HLK-invariant $(1,[2,0,0])$ while the origamis of discriminant $\frac{n^{2}}{4}$ have HLK-invariant $(3,[0,0,0])$.

\subsection{Surface parameters}

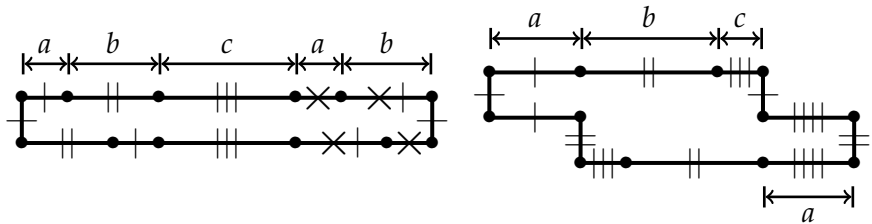
\begin{figure}[b]
    \centering
    \begin{tikzpicture}[scale = 0.6, line width = 1.5pt]
        \draw (0,0) -- node[rotate = 90]{$|$}(0,1) -- node{$|$}(1,1) -- node{$||$}(3,1) -- node{$|||$}(6,1) -- node{$\bigtimes$}(7,1) -- node{$\bigtimes|$}(9,1) -- node[rotate = 90]{$|$}(9,0) -- node{$\bigtimes$}(8,0) -- node{$\bigtimes|$}(6,0) -- node{$|||$}(3,0) -- node{$|$}(2,0) --  node{$||$} cycle;
        \draw[|<->,line width = 1pt] (0,1.75) -- node[above]{$a$}(1,1.75);
        \draw[|<->,line width = 1pt] (1,1.75) -- node[above]{$b$}(3,1.75);
        \draw[|<->,line width = 1pt] (3,1.75) -- node[above]{$c$}(6,1.75);
        \draw[|<->,line width = 1pt] (6,1.75) -- node[above]{$a$}(7,1.75);
        \draw[|<->|,line width = 1pt] (7,1.75) -- node[above]{$b$}(9,1.75);
        \foreach \x/\y in {0/0,0/1,1/1,3/1,6/1,7/1,9/1,9/0,8/0,6/0,3/0,2/0}{
            \node at (\x,\y) {$\bullet$};
        }
        \node at (0,1.75){$\,$};
        \node at (0,-1.75){$\,$};
    \end{tikzpicture}
    \begin{tikzpicture}[scale = 0.6, line width = 1.5pt]
        \draw (0,0) -- node[rotate = 90]{$|$}(0,1) -- node{$|$}(2,1) -- node{$||$}(5,1) -- node{$|||$}(6,1) -- node[rotate = 90]{$|$}(6,0) -- node{$||||$}(8,0) -- node[rotate = 90]{$||$}(8,-1) -- node{$||||$}(6,-1) -- node{$||$}(3,-1) -- node{$|||$}(2,-1) -- node[rotate = 90]{$||$}(2,0) --  node{$|$} cycle;
        \draw[|<->,line width = 1pt] (0,1.75) -- node[above]{$a$}(2,1.75);
        \draw[|<->,line width = 1pt] (2,1.75) -- node[above]{$b$}(5,1.75);
        \draw[|<->|,line width = 1pt] (5,1.75) -- node[above]{$c$}(6,1.75);
        \draw[|<->|,line width = 1pt] (6,-1.75) -- node[below]{$a$}(8,-1.75);
        \foreach \x/\y in {0/0,0/1,2/1,5/1,6/1,6/0,8/0,8/-1,6/-1,3/-1,2/-1,2/0}{
            \node at (\x,\y) {$\bullet$};
        }
    \end{tikzpicture}
    \caption{One-cylinder and two-cylinder cusp representatives in the Prym locus of $\calH(4)$.}
    \label{fig:H4-1-cyl}
\end{figure}

\begin{figure}[t]
    \centering
    \begin{tikzpicture}[scale = 0.7, line width = 1.5pt]
        \draw (0,0) -- node[rotate = 90]{$|$}(2,1) -- node[rotate = 90]{$||$}(3,3) -- node{$||$}(0,3) -- node[rotate = 90]{$|||$}(2,4) -- node{$||$}(5,4) -- node{$|$}(9,4) -- node[rotate = 90]{$|||$}(7,3) -- node[rotate = 90]{$||$}(6,1) -- node{$|||$}(9,1) -- node[rotate = 90]{$|$}(7,0) -- node{$|||$}(4,0) -- node{$|$} cycle;
        \draw[<->|,line width = 1pt] (-0.5,1) -- node[left]{$h_{1}$}(-0.5,3);
        \draw[|<->|,line width = 1pt] (-0.5,0) -- node[left]{$h_{2}$}(-0.5,1);
        \draw[|<->|,line width = 1pt] (3,3.35) -- node[below]{$w_{1}$}(7,3.35);
        \draw[|<->|,line width = 1pt] (0,-0.5) -- node[below]{$w_{2}$}(7,-0.5);
        \draw[|<->,line width = 1pt] (2,0.65) -- node[above]{$t_{1}$}(3,0.65);
        \draw[line width = 1pt] (3.025,0.475) -- (3.025,2.65);
        \draw[<->|,line width = 1pt] (7,-0.5) -- node[below]{$t_{2}$}(9,-0.5);
        \foreach \x/\y in {0/0,2/1,3/3,0/3,2/4,5/4,9/4,7/3,6/1,9/1,7/0,4/0}{
            \node at (\x,\y) {$\bullet$};
        }
    \end{tikzpicture}
    \begin{tikzpicture}[scale = 0.7, line width = 1.5pt]
        \draw (0,0) -- node[rotate = 90]{$||$}(1,2) -- node{$|$}(-1,2) -- node[rotate = 90]{$|$}(1,3) -- node[rotate = 90]{$|||$}(2,5) -- node{$|$}(4,5) -- node[rotate = 90]{$|||$}(3,3) -- node{$||$}(5,3) -- node{$|||$}(8,3) -- node[rotate = 90]{$|$}(6,2) -- node{$|||$}(3,2) -- node[rotate = 90]{$||$}(2,0) -- node{$||$} cycle;
        \draw[|<->|,line width = 1pt] (0.5,3) -- node[left]{$h_{1}$}(0.5,5);
        \draw[|<->|,line width = 1pt] (-1.5,2) -- node[left]{$h_{2}$}(-1.5,3);
        \draw[|<->|,line width = 1pt] (2,5.5) -- node[above]{$w_{1}$}(4,5.5);
        \draw[<->,line width = 1pt] (-1,-0.5) -- node[below]{$w_{2}$}(6,-0.5);
        \draw[|<->,line width = 1pt] (1,5.5) -- node[above]{$t_{1}$}(2,5.5);
        \draw[<->,line width = 1pt] (6,-0.5) -- node[below]{$t_{2}$}(8,-0.5);
        \draw[line width = 1pt] (-1.025,-0.675) -- (-1.025,1.65);
        \draw[line width = 1pt] (6,-0.675) -- (6,1.65);
        \draw[line width = 1pt] (8.025,-0.675) -- (8.025,2.65);
        \foreach \x/\y in {0/0,1/2,-1/2,1/3,2/5,4/5,3/3,5/3,8/3,6/2,3/2,2/0}{
            \node at (\x,\y) {$\bullet$};
        }
    \end{tikzpicture}
    \begin{tikzpicture}[scale = 0.7, line width = 1.5pt]
        \draw (0,0) -- node[rotate = 90]{$||$}(1,2) -- node{$|$}(-1,2) -- node[rotate = 90]{$|$}(1,3) -- node[rotate = 90]{$|||$}(2,5) -- node{$|$}(4,5) -- node{$||$}(5,5) -- node[rotate = 90]{$||$}(4,3) -- node{$|||$}(6,3) -- node[rotate = 90]{$|$}(4,2) -- node[rotate = 90]{$||$}(3,0) -- node{$|||$}(1,0) -- node{$||$} cycle;
        \draw[|<->|,line width = 1pt] (0.5,3) -- node[left]{$h_{1}$}(0.5,5);
        \draw[|<->|,line width = 1pt] (-1.5,2) -- node[left]{$h_{2}$}(-1.5,3);
        \draw[|<->|,line width = 1pt] (2,5.5) -- node[above]{$w_{1}$}(5,5.5);
        \draw[<->,line width = 1pt] (-1,-0.5) -- node[below]{$w_{2}$}(4,-0.5);
        \draw[|<->,line width = 1pt] (1,5.5) -- node[above]{$t_{1}$}(2,5.5);
        \draw[<->,line width = 1pt] (4,-0.5) -- node[below]{$t_{2}$}(6,-0.5);
        \draw[line width = 1pt] (-1.025,-0.675) -- (-1.025,1.65);
        \draw[line width = 1pt] (4,-0.675) -- (4,1.65);
        \draw[line width = 1pt] (6.025,-0.675) -- (6.025,2.65);
        \foreach \x/\y in {0/0,1/2,-1/2,1/3,2/5,4/5,5/5,4/3,6/3,4/2,3/0,1/0}{
            \node at (\x,\y) {$\bullet$};
        }
    \end{tikzpicture}
    \caption{Three-cylinder surface parameters in $\calH(4)$ corresponding to the $A+$, $A-$ and $B$ prototypes.}
    \label{fig:H4-params}
\end{figure}
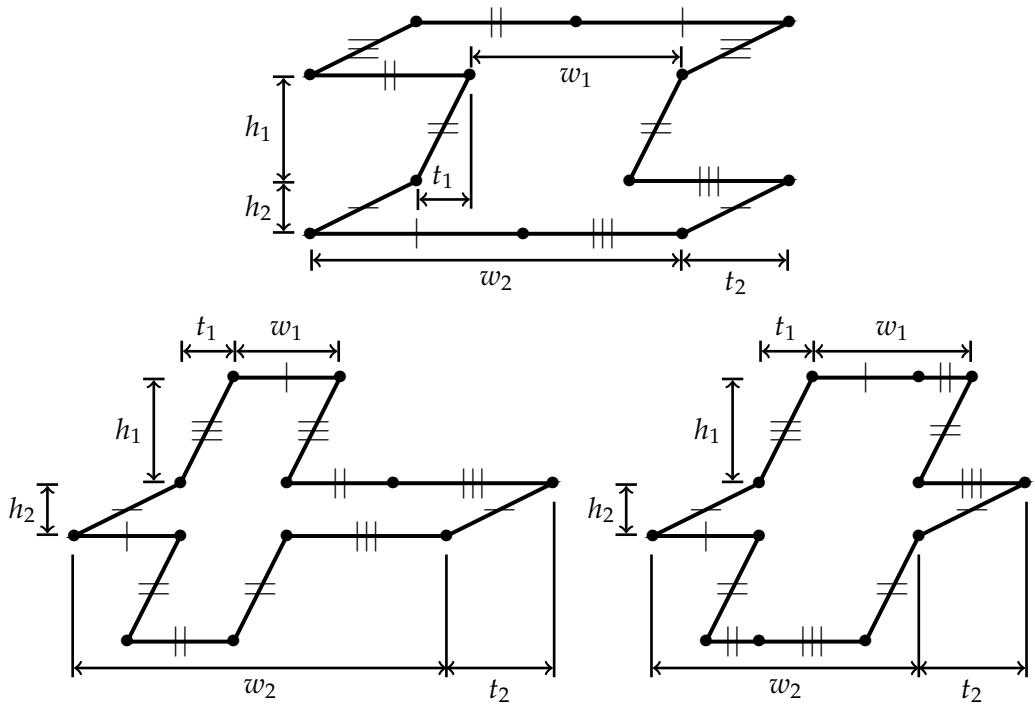

In the Prym locus in $\calH(4)$, an origami can have one, two, or three cylinders. The cusp representatives for one-cylinder and two-cylinder origamis are shown in Figure~\ref{fig:H4-1-cyl}. Here, $(a,b,c)$ is a triple of positive integers with $2a+2b+c = n$ in the one-cylinder case, $2a+2b+2c = n$ in the two-cylinder case, and $\gcd(a,b,c) = 1$ in both cases. It can be seen that two-cylinder origamis do not exist if $n$ is odd, and one-cylinder origamis do not exist for even $n$ with HLK-invariant $(3,[0,0,0])$. One-cylinder cusps have width equal to $n$, while two-cylinder cusps have width equal to $\frac{n}{2}$.

An origami with three cylinders can have one of the three forms shown in Figure~\ref{fig:H4-params}. These correspond, respectively, to the prototypes $A+$, $A-$, and $B$ discussed in the following subsection. The cusp width is given by the same formula used for two-cylinder origamis in $\calH(2)$. Namely, the cusp width is
\[\text{lcm}\left(\frac{w_{1}}{\gcd(w_{1},h_{1})},\frac{w_{2}}{\gcd(w_{2},h_{2})}\right) = O(n^2).\]

Moreover, we see that the orbits will have size $O(n^{3})$.

\subsection{Prototypes and butterfly moves}

Similar to the work of McMullen discussed above, Lanneau--Nguyen define prototypical splittings of discriminant $D$. Here, there are three prototypical splittings: type $A+$, type $A-$, and type $B$. See Figure~\ref{fig:H4-protos}. Prototypes of type $A\pm$ are parameterised by the set
\begin{multline*}
    \mathcal{Q}_{D} := \{(w,h,t,e,\varepsilon)\in\Z^{4}\times\{+,-\}\,:\,w>0,h>0,e+2h<w,\\0\leq t<\gcd(w,h),
    \gcd(w,h,t,e) = 1, D = e^{2}+8wh\}.
\end{multline*}
As before, $\lambda := \frac{e+\sqrt{D}}{2}$.

\begin{figure}[t]
    \centering
    \begin{tikzpicture}[scale = 0.6, line width = 1.5pt]
        \draw (0,0) -- (2,1) -- (2,3) -- (-1,3) -- (1,4) -- (4,4) -- (6,4) -- (4,3) -- (4,1) -- (7,1) -- (5,0) -- (2,0) -- cycle;
        \draw[<->|,line width = 1pt] (-1.5,1) -- node[left]{$\lambda$}(-1.5,3);
        \draw[line width = 1pt] (-1.675,1) -- (1.65,1);
        \draw[|<->|,line width = 1pt] (2,0.65) -- node[above]{$\lambda$}(4,0.65);
        \node[above left] at (-1,3) {$(0,0)$};
        \node[above] at (1,4) {$(t,h)$};
        \node[above right] at (6,4) {$(w+t,h)$};
        \draw[purple] (-1,3) -- (3.5,4);
        \draw[purple] (1.5,3) -- (6,4);
        \foreach \x/\y in {0/0,2/1,2/3,-1/3,1/4,4/4,6/4,4/3,4/1,7/1,5/0,2/0}{
            \node at (\x,\y) {$\bullet$};
        }
    \end{tikzpicture}
    \begin{tikzpicture}[scale = 0.5, line width = 1.5pt]
        \draw (0,0) -- (0,2) -- (-2,2) -- (0,3) -- (0,5) -- (2,5) -- (2,3) -- (4,3) -- (7,3) -- (5,2) -- (2,2) -- (2,0) -- cycle;
        \draw[|<->|,line width = 1pt] (-0.5,3) -- node[left]{$\frac{\lambda}{2}$}(-0.5,5);
        \draw[|<->|,line width = 1pt] (0,5.5) -- node[above]{$\frac{\lambda}{2}$}(2,5.5);
        \node[left] at (-2,2) {$(0,0)$};
        \node[below] at (5,2) {$(w,0)$};
        \node[above right] at (7,3) {$(w+t,h)$};
        \draw[purple] (-2,2) -- (7,3);
        \foreach \x/\y in {0/0,0/2,-2/2,0/3,0/5,2/5,2/3,4/3,7/3,5/2,2/2,2/0}{
            \node at (\x,\y) {$\bullet$};
        }
    \end{tikzpicture}
    \begin{tikzpicture}[scale = 0.5, line width = 1.5pt]
        \draw (0,0) -- (0,3) -- (-2,3) -- (0,4) -- (0,7) -- (2,7) -- (3,7) -- (3,4) -- (5,4) -- (3,3) -- (3,0) -- (1,0) -- cycle;
        \draw[|<->|,line width = 1pt] (-0.5,4) -- node[left]{$\frac{\lambda}{2}$}(-0.5,7);
        \draw[|<->|,line width = 1pt] (0,7.5) -- node[above]{$\frac{\lambda}{2}$}(3,7.5);
        \node[left] at (-2,3) {$(0,0)$};
        \node[below right] at (3,3) {$(w,0)$};
        \node[above right] at (5,4) {$(w+t,h)$};
        \foreach \x/\y in {0/0,0/3,-2/3,0/4,0/7,2/7,3/7,3/4,5/4,3/3,3/0,1/0}{
            \node at (\x,\y) {$\bullet$};
        }
    \end{tikzpicture}
    \caption{Three-cylinder prototypes of type $A+$ (top left), $A-$ (top right) and $B$ (bottom). Directions corresponding to butterfly moves ($q = 2$ and $q = 1$, resp.) are shown in the top prototypes.}
    \label{fig:H4-protos}
\end{figure}
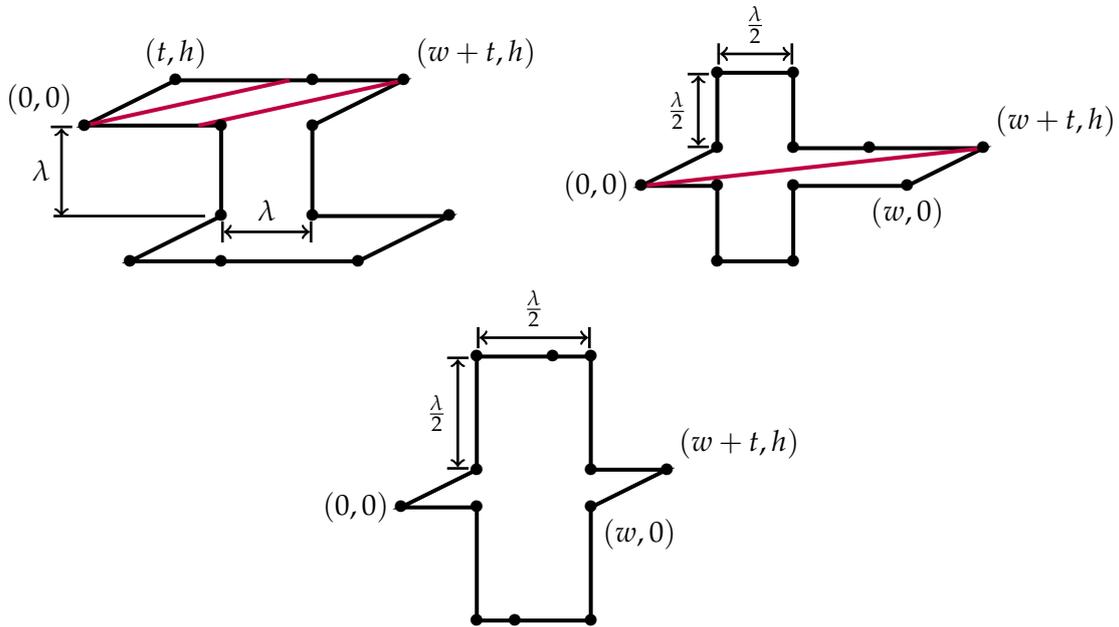

In this setting, butterfly moves exist for such prototypes and send a prototype of type $A+$ to a prototype of type $A-$, and vice versa. A value of $q$ is admissible if $q = \infty$ or $(e + 4qh)^{2} < D$. The quadruples that forget the type ($\pm$) are parameterised by the set
\[\mathcal{P}_{D} := \{(w,h,t,e)\in\Z^{4}\,:\,(w,h,t,e,\pm)\in\mathcal{Q}_{D}\}.\]

In this setting, a reduced prototype is a prototype of the form $(w,1,0,e)$. Such prototypes are parameterised by the set
\[\mathcal{S}_{D} := \{e\in\Z\,:\,e^2\equiv D\!\!\mod\! 8, e^{2},(e+4)^{2} < D\}.\]
As in the previous section, we will at times abuse notation and say that $(w,1,0,e)$ and $e$ are both elements of $\mathcal{S}_{D}$.

Lanneau--Nguyen prove the following result describing the affect of a butterfly move.

\begin{proposition}[{\cite[Propositions 7.5 and 7.6]{LN14}}]
    Let $(w,h,t,e,\pm)\in\mathcal{Q}_{D}$, then, for any admissible $q$, the butterfly move sends $(w,h,t,e,\pm)$ to $(w',h',t',e',\mp)$ with
    \begin{align*}
    e' &= -e - 4qh \\ 
    h' &= \gcd(qh,w+qt)
    \end{align*}
    if $q < \infty$, or
    \begin{align*}
    e' &= -e - 4qh \\ 
    h' &= \gcd(h,t)
    \end{align*}
    if $q = \infty$. In each case, $w'$ is determined by the condition that $D = e^2 + 8wh = (e')^2 + 8w'h'$.
\end{proposition}

Similar to McMullen's proof in genus two, this is proved by considering the basis changes for homology associated to the matrix
\[\begin{pmatrix}
    -2qh & 0 & h & -e-t-2qh \\
    0 & -2qh & -qh & w+qt \\
    2w+2qt & 2e+2t+4qh & e+2qh & 0 \\
    2qh & 2h & 0 & e+2qh
\end{pmatrix}\]
for $q < \infty$, or 
\[\begin{pmatrix}
    -2h & 0 & 0 & -e+w-2h \\
    0 & -2h & -h & t \\
    2t & 2e-2w+4h & e+2h & 0 \\
    2qh & 0 & 0 & e+2h
\end{pmatrix}\]
for $q = \infty$. The key minors in each case are the upper right minors
\[\begin{pmatrix}
    h & -e-t-2qh \\
    -qh & w+qt
\end{pmatrix}\]
and 
\[\begin{pmatrix}
    0 & -e+w-2h \\
    -h & t
\end{pmatrix}\]
which resemble those considered in genus two.

Consider the origami $((1,2,3,4,5,6)(8,9,10,11,12,13),(6,8,7))$ shown in Figure~\ref{fig:prym4-sts-example-1}. It can be checked that this origami corresponds to the prototype $(6,1,0,-11,+)\in\mathcal{Q}_{169}$. In particular, $q = 2$ is admissible. The butterfly move direction $(w_{2}+qt_{2},qh_{2}) = (6,2)$ can be seen in Figure~\ref{fig:prym4-sts-example-1}. To shadow the butterfly move, since we are already at the cusp representative, we need only perform $S^{-1}\circ T^{-2}$ to make the butterfly move direction horizontal. The cusp representative of the resulting origami is $$((1,2,3,4,5)(6,7)(8,9)(10,11)(12,13),(1,6,8,2,7,9)(3,13,11,4,12,10)),$$ shown in Figure~\ref{fig:prym4-sts-example-2}, and it can be checked that it does indeed correspond to the prototype $B_{2}(6,1,0,-11,+) = (10,2,1,3,-)$.

As in the genus two case above, we see that shadowing a butterfly move costs $O(n^{2})$ to move via powers of $T$ to the cusp representative followed by a Euclidean algorithm operation costing $O(\max\{w_{2}+qt_{2},qh_{2}\})$, for finite $q$, or $O(\max\{t_{2},h_{2}\}) = O(h_{2})$ for $q = \infty$.

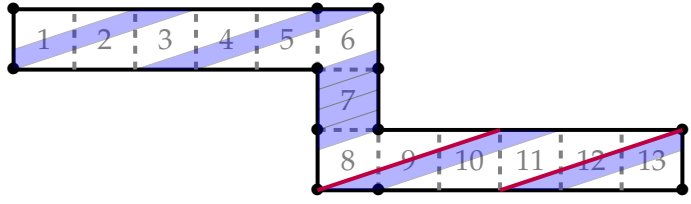
\begin{figure}[h]
    \centering
    \begin{tikzpicture}[scale = 0.8, line width = 1.5pt]
        \draw (0,0) -- (5,0) -- (5,-2) -- (11,-2) -- (11,-1) -- (6,-1) -- (6,1) -- (0,1) -- cycle;
        \draw[gray,dashed] (5,0) -- (6,0);
        \draw[gray,dashed] (5,-1) -- (6,-1);
        \foreach \i in {1,2,3,4,5}{
            \draw[gray,dashed] (\i,0) -- (\i,1);
            \node[gray] at (\i-0.5,0.5) {$\i$};
        }
        \foreach \i in {9,10,11,12,13}{
            \draw[gray,dashed] (\i-3,-2) -- (\i-3,-1);
            \node[gray] at (\i-2.5,-1.5) {$\i$};
        }
        \foreach \i in {6,7,8}{
            \node[gray] at (5.5,6.5-\i) {$\i$};
        }
        \foreach \x/\y in {0/0,0/1,5/1,6/1,6/0,6/-1,11/-1,11/-2,6/-2,5/-2,5/-1,5/0}{
            \node at (\x,\y) {$\bullet$};
        }
        \draw[line width = 0pt, fill = blue, opacity = 0.3] (5,-2) -- (6,-2) -- (9,-1) -- (8,-1) -- (5,-2);
        \draw[line width = 0pt, fill = blue, opacity = 0.3] (8,-2) -- (9,-2) -- (11,-1-1/3) -- (11,-1) -- (8,-2);
        \draw[line width = 0pt, fill = blue, opacity = 0.3] (5,-1-1/3) -- (6,-1) -- (6,-1+1/3) -- (5,-1) -- (5,-1-1/3);
        \draw[line width = 0pt, fill = blue, opacity = 0.3] (6,1) -- (5,1) -- (2,0) -- (3,0) -- (6,1);
        \draw[line width = 0pt, fill = blue, opacity = 0.3] (3,1) -- (2,1) -- (0,1/3) -- (0,0) -- (3,1);
        \draw[line width = 0pt, fill = blue, opacity = 0.3] (6,1/3) -- (5,0) -- (5,-1/3) -- (6,0) -- (6,1/3);
        \draw[line width = 0pt, fill = blue, opacity = 0.3] (6,1/3-1/3) -- (5,0-1/3) -- (5,-1/3-1/3) -- (6,0-1/3) -- (6,1/3-1/3);
        \draw[line width = 0pt, fill = blue, opacity = 0.3] (6,1/3-2/3) -- (5,0-2/3) -- (5,-1/3-2/3) -- (6,0-2/3) -- (6,1/3-2/3);
        \draw[color = purple] (5,-2) -- (8,-1);
        \draw[color = purple] (8,-2) -- (11,-1);
    \end{tikzpicture}
    \caption{An origami corresponding to the prototype $(6,1,0,-11,+)\in\mathcal{Q}_{169}$. The direction corresponding to the butterfly move $B_{2}$ and the resulting simple cylinder are shown in purple and blue, respectively.}
    \label{fig:prym4-sts-example-1}
\end{figure}

\begin{figure}[h]
    \centering
    \begin{tikzpicture}[scale = 0.8, line width = 1.5pt]
        \draw (0,0) -- (0,1) -- (1,3) -- (3,3) -- (2,1) -- (5,1) -- (5,0) -- (4,0) -- (3,-2) -- (1,-2) -- (2,0) -- cycle;
        \foreach \i in {1,2,3,4}{
            \draw[gray,dashed] (\i,0) -- (\i,1);
            \node[gray] at (\i-0.5,0.5) {$\i$};
        }
        \node[gray] at (4.5,0.5) {$5$};
        \node[gray] at (0.5,1.5) {$6$};
        \node[gray] at (1.5,1.5) {$7$};
        \node[gray] at (2.5,2.5) {$8$};
        \node[gray] at (1.5,2.5) {$9$};
        \node[gray] at (2.5,-0.5) {$10$};
        \node[gray] at (3.5,-0.5) {$11$};
        \node[gray] at (1.5,-1.5) {$13$};
        \node[gray] at (2.5,-1.5) {$12$};
        \foreach \i in {0,1}{
            \draw[gray,dashed] (0+0.5*\i,1+\i) -- (2+0.5*\i,1+\i);
            \draw[gray,dashed] (1+\i,1) -- (1+\i,3);
        }
        \foreach \i in {0,-1}{
            \draw[gray,dashed] (2+0.5*\i,0+\i) -- (4+0.5*\i,0+\i);
            \draw[gray,dashed] (3+\i,0) -- (3+\i,-2);
        }
        \foreach \x/\y in {0/0,0/1,1/3,3/3,2/1,4/1,5/1,5/0,4/0,3/-2,1/-2,2/0}{
            \node at (\x,\y) {$\bullet$};
        }
    \end{tikzpicture}
    \caption{An origami corresponding to the prototype $(10,2,1,3,-)\in\mathcal{Q}_{169}$.}
    \label{fig:prym4-sts-example-2}
\end{figure}

\subsection{The classification proof and implied algorithm}

The classification proof of Lanneau--Nguyen can be sketched as follows:
\begin{enumerate}
    \item Prove that every translation surface $X$ in $\Omega\mathcal{E}_{D}(4)$ has a prototypical splitting of type $A\pm$ in its $\GL^{+}(2,\R)$-orbit. The associated prototype $(w,h,t,e,\varepsilon)$ lies in $\mathcal{Q}_{D}$.
    \item Prove that the prototype $(w,h,t,e)\in\mathcal{P}_{D}$ can be sent to a reduced prototype $(w',1,0,e')$ in $\mathcal{S}_{D}$ by a sequence of butterfly moves.
    \item Prove that two reduced prototypes in $\mathcal{S}_{D}$ are connected by a sequence of butterfly moves lying inside $\mathcal{S}_{D}$ (when $\mathcal{S}_{D}$ is connected) or through $\mathcal{P}_{D}$ when $\mathcal{S}_{D}$ has two components. In any case, the reduced prototypes are all connected to one another by some sequence of butterfly moves.
    \item Finally, when $D\not\equiv 1\!\!\mod 8$, find a sequence of butterfly moves of odd length that fixes a specific prototype in $\mathcal{S}_{D}$. Since each butterfly move changes $\varepsilon$ for a prototype in $\mathcal{Q}_{D}$, this allows any choice of $\varepsilon$ to be achieved. This gives one orbit for $D\not\equiv 1\!\!\mod 8$, and two orbits for $D\equiv 1\!\!\mod 8$ where $(w,h,t,e,\varepsilon)$ and $(w,h,t,e,\varepsilon')$ lie in different orbits if $\varepsilon\neq \varepsilon'$.
\end{enumerate}

For each $D\not\equiv 1\!\!\mod 8$, we choose the target prototype to be the reduced prototype $(w,h,t,e)$ in $\mathcal{S}_{D}$ that admits the odd length sequence of butterfly moves fixing itself (see Subsection~\ref{subsec:H4-result-bound} for what these prototypes are). Let the target origamis be the cusp representatives of the cusps associated to the prototypes $(w,h,t,e,+)$ in $\mathcal{Q}_{D}$.

For $D\equiv 1 \!\!\mod 8$, we set the target prototypes to be $(\frac{D-1}{8},1,0,-1,\pm)$ (one lies in each orbit) and set the target origamis to be the cusp representatives of the associated cusps.

This implies the following algorithm for moving inside an $\SL(2,\Z)$-orbit:
\begin{enumerate}
    \item Move inside the $\SL(2,\Z)$-orbit to a three-cylinder surface that corresponds to a prototype $(w,h,t,e,\varepsilon)$ of type $A\pm$.
    \item Shadow the sequence of butterfly moves that sends $(w,h,t,e)$ to a reduced prototype $(w',1,0,e')$.
    \item Shadow the sequence of butterfly moves that sends this reduced prototype to the target reduced prototype.
    \item If required, shadow the sequence of butterfly moves of odd length to change the value of $\varepsilon$, before moving in the resulting cusp to the target origami.
\end{enumerate}

\subsection{Bounding the number of butterfly moves}

The argument is similar to that given in Subsection~\ref{subsec:butterfly-bound}.

\subsubsection{Connecting to reduced prototypes}

Similar to $\calH(2)$, we will show that a prototype $(w,h,t,e)\in\mathcal{P}_{D}$ can be sent to a reduced prototype with $O(\log n)$ applications of $B_{1}$. Indeed, the key minor we are reducing is the matrix
\[\begin{pmatrix}
    h & -e-t-2h \\
    -h & w+t
\end{pmatrix}.\]
The arguments of Subsection~\ref{subsec:butterfly-bound} work exactly the same to show that at most three applications of $B_{1}$ are required to produce a prototype $(w',h',t',e')$ with $h'\leq h/2$. We have $\log h = O(\log n)$ and we are done.

\subsubsection{Connecting reduced prototypes to the target prototype}

Similar to the case in $\calH(2)$, the reduced prototypes in $\mathcal{P}_{D}$ are parameterised by the set
\[\mathcal{S}_{D} := \{e\in\Z\,:\,e^{2}\equiv D\!\!\mod 8,\,\text{and}\,\, e^{2}, (e+4)^2 < D\}.\]
So, again, we have $|\mathcal{S}_{D}| = O(\sqrt{D}) = O(n)$.

In this setting, Lanneau and Nguyen prove the following.

\begin{theorem}[Follows from {\cite[Theorems 8.2 and 8.6]{LN14}}]\label{thm:H4-SD}
    Once $D \equiv 0,1,4\!\!\mod 8$ is large enough, the set $\mathcal{S}_{D}$ is non-empty and either
    \begin{itemize}
        \item $D\equiv 4\!\!\mod 16$ and $\mathcal{S}_{D}$ has two connected components $\{e\equiv 2\mod 8\}$ and $\{e \equiv -2\mod 8\}$; or
        \item $\mathcal{S}_{D}$ is connected.
    \end{itemize}
    However, again once $D$ is large enough, $\mathcal{P}_{D}$ is connected.
\end{theorem}

If we are in the case where $\mathcal{S}_{D}$ is connected, then we can reach the target reduced prototype from any other reduced prototype in $O(n)$ butterfly moves. In the case that $\mathcal{S}_{D}$ has two components, Lanneau--Nguyen prove that the two components can be connected as follows:
\begin{itemize}
    \item if $D=4+16k$, $k$ odd, then one can connect the two components via the following path in $\mathcal{P}_{D}$:
    \[(2k - 4,1,0,-6) \xrightarrow[]{B_{2}} (k,2,0,-2) \xrightarrow[]{B_{\infty}} (k - 2,2,0,-6) \xrightarrow[]{B_{1}} (2k,1,0,-2).\]
    \item if $D=4+32k$, $k$ odd, then one can connect the two components via the following path in $\mathcal{P}_{D}$:
    \[(4k,1,0,2) \xrightarrow[]{B_{2}} (2k-6,2,1,-10) \xrightarrow[]{B_{2}} (2k - 2,2,1,-6) \xrightarrow[]{B_{1}} (4k,1,0,-2).\]
    \item if $D=4+32k$, $k$ even, then one can connect the two components via the following path in $\mathcal{P}_{D}$:
    \[(4k - 4,1,0,-6) \xrightarrow[]{B_{4}} (k-3,4,0,-10) \xrightarrow[]{B_{\infty}} (k - 1,4,0,-6) \xrightarrow[]{B_{1}} (4k-12,1,0,-10).\]
\end{itemize}

\subsection{The resulting diameter bound}\label{subsec:H4-result-bound}

We must now bound the number of applications of $T$ or $S$ required to connect an arbitrary origami $X$ to the target origami in its orbit.

Firstly, note that the cusp representatives of one-cylinder and two-cylinder origamis have a three-cylinder decomposition of type $A\pm$ in the vertical direction. So, when starting from such an origami, we can arrive at an origami corresponding to a prototype of type $A\pm$ in $O(n)$ applications of $T$ or $S$. Indeed, the cusps here have width $O(n)$ and rotating the origami by $\pi/2$ is $O(1)$.

If we have a three-cylinder origami corresponding to a prototype of type $B$, then the proof of~\cite[Proposition 4.7]{LN14} gives us that the cusp representative of this origami has a three-cylinder decomposition of type $A\pm$ in one of the following directions:
\begin{itemize}
    \item $(t_1+t_2-w_2,h_1+h_2)$;
    \item $(2t_1+t_2,2h_1+h_2)$;
    \item $(w_1+t_1+t_2,h_1+h_2)$; or
    \item $(2t_1+t_2+y-w_1-w_2,2h_1+h_2)$, for some integer $1\leq y < w_{1}$.
\end{itemize}
In each case, this direction can be made horizontal in $O(n)$ applications of $T$ or $S$, after reaching the cusp representative in $O(n^{2})$ applications of $T$. Again, we move to the cusp representative by naively applying powers of $T$.

So, any origami can be taken to a three-cylinder origami corresponding to a prototype of type $A\pm$ in $O(n^{2})$ applications of $T$ or $S$.

We now make use of the following result of Lanneau--Nguyen and its proof.

\begin{theorem}[{\cite[Theorem 9.2]{LN14}}]\label{thm:LN-9.1}
    Let $D > 16$ be an even discriminant with $D \equiv 0, 4 \!\!\mod\! 8$. If $D \not\in \{48, 68, 100\}$, then $\mathcal{Q}_{D}$ has only one component.
\end{theorem}

Since we are interested in asymptotics, we can ignore the exceptional discriminants here. We handle different residue classes separately.

\subsubsection{$D\equiv 1 \!\!\mod\! 8$} In this case, Lanneau--Nguyen prove that $\mathcal{Q}_{D}$ has two components. Indeed, for each $(w,h,t,e)\in\mathcal{P}_{D}$, they prove that $(w,h,t,e,+)$ and $(w,h,t,e,-)$ lie in different $\GL^{+}(2,\R)$-orbits~\cite[Theorem 6.1]{LN14}.

For example, it can be shown that if $\sqrt{D} \equiv 1\!\!\mod 4$ then the prototype $(\frac{D-1}{8},1,0,-1,-)$ corresponds to an origami with $\sqrt{D}$ squares (i.e., $n$ is odd). Whereas, $(\frac{D-1}{8},1,0,-1,+)$ corresponds to an origami with $2\sqrt{D}$ squares (i.e., $n \equiv 2 \!\!\mod 4$). Conversely, when $\sqrt{D} \equiv 3\!\!\mod 4$, $(\frac{D-1}{8},1,0,-1,+)$ corresponds to an origami with $\sqrt{D}$ squares while $(\frac{D-1}{8},1,0,-1,-)$ corresponds to an origami with $2\sqrt{D}$ squares. 

In this setting, take any $(w,h,t,e,\varepsilon)\in\mathcal{Q}_{D}$ and apply the $O(\log n)$ applications of $B_{1}$ to reach a prototype of the form $(w',1,0,e',\varepsilon')$ with $e'\in \mathcal{S}_{D}$. Each application of $B_1$ requires $O(n^{2})$ applications of $T$ or $S$ to shadow, so this process costs $O(n^{2}\log n)$.

Now, by Theorem~\ref{thm:H4-SD}, $\mathcal{S}_{D}$ is connected and so we can apply $O(n)$ butterfly moves to connect the reduced prototype to the target reduced prototype. Similar to Lemma~\ref{lem:H2-red}, we have the following.

\begin{lemma}\label{lem:H4-red}
    An origami corresponding to a reduced prototype $(w,1,0,e,\varepsilon)\in\mathcal{Q}_{D}$ has $h_{2}\in\{1,2\}$ and cusp width $k\in\{\frac{w_{2}}{2},w_{2},2w_{2}\}$. So, $k=O(n)$.
\end{lemma}

\begin{proof}
    Suppose that $D = n^{2}$ and $\varepsilon = +$. The prototype has area $\lambda^{2} + 2w = n\lambda$. Suppose that in achieving the prototype the origami of area $n$ is scaled by $L$ vertically and $\frac{\lambda}{L}$ horizontally. We then have
    \[\lambda = Lh_{1},\,\,\lambda = \frac{\lambda}{L}w_{1},\,\,1 = Lh_{2},\,\,w = \frac{\lambda}{L}w_{2},\]
    from which we obtain that $1 = w_{1}h_{2}$. Hence the origami has surface parameters $w_{1} = 1$, $h_{1} = \lambda$, $h_{2} = 1$, $w_{2} = \frac{w}{\lambda}$ and cusp width
    \[k = \text{lcm}\left(\frac{w_{1}}{\gcd(w_{1},h_{1})},\frac{w_{2}}{\gcd(w_{2},h_{2})}\right) = \text{lcm}\left(\frac{1}{\gcd(1,\lambda)},\frac{w_{2}}{\gcd(w_{2},1)}\right) = w_{2}.\]

    When $D = \frac{n^{2}}{4}$ and $\varepsilon = +$, the prototype has area $\lambda^{2} + 2w = n\frac{\lambda}{2}$. As above, suppose that in achieving the prototype the origami of area $n$ is scaled by $L$ vertically and $\frac{\lambda}{2L}$ horizontally. We then have
    \[\lambda = Lh_{1},\,\,\lambda = \frac{\lambda}{2L}w_{1},\,\,1 = Lh_{2},\,\,w = \frac{\lambda}{2L}w_{2},\]
    from which we obtain that $2 = w_{1}h_{2}$. Hence,
    \[(w_{1},h_{1},w_{2},h_{2})\in\left\{\left(2,\lambda,\frac{2w}{\lambda},1\right),\left(1,2\lambda,\frac{w}{\lambda},2\right)\right\}.\]
    However, since $\gcd(h_{1},h_{2}) = 1$ for primitive origamis, we must have $$(w_{1},h_{1},w_{2},h_{2}) = \left(2,\lambda,\frac{2w}{\lambda},1\right)$$ and so the origami has cusp width
    \[k = \text{lcm}\left(\frac{w_{1}}{\gcd(w_{1},h_{1})},\frac{w_{2}}{\gcd(w_{2},h_{2})}\right) = \text{lcm}\left(\frac{2}{\gcd(2,\lambda)},\frac{w_{2}}{\gcd(w_{2},1)}\right) \in\{w_{2},2w_{2}\}.\]
    
    When $D = n^{2}$ and $\varepsilon = -$, the prototype has area $2(\frac{\lambda}{2})^{2}+w = n\frac{\lambda}{2}$. So, if the vertical scaling is $L$ and horizontal scaling is $\frac{\lambda}{2L}$, then
    \[\frac{\lambda}{2} = Lh_{1},\,\,\frac{\lambda}{2} = \frac{\lambda}{2L}w_{1},\,\,1 = Lh_{2},\,\,w = \frac{\lambda}{2L}w_{2},\]
    giving $1 = w_{1}h_{2}$, again. Hence, $w_{1} = 1$, $h_{1} = \frac{\lambda}{2}$, $w_{2} = \frac{2w}{\lambda}$, $h_{2} = 1$, and the origami has cusp width
    \[k = \text{lcm}\left(\frac{w_{1}}{\gcd(w_{1},h_{1})},\frac{w_{2}}{\gcd(w_{2},h_{2})}\right) = \text{lcm}\left(\frac{1}{\gcd(1,\frac{\lambda}{2})},\frac{w_{2}}{\gcd(w_{2},1)}\right) = w_{2}.\]

    Finally, when $D = \frac{n^{2}}{4}$ and $\varepsilon = -$, the prototype has area $2(\frac{\lambda}{2})^{2}+w = n\frac{\lambda}{4}$. So, if the vertical scaling is $L$ and horizontal scaling is $\frac{\lambda}{4L}$, then
    \[\frac{\lambda}{2} = Lh_{1},\,\,\frac{\lambda}{2} = \frac{\lambda}{4L}w_{1},\,\,1 = Lh_{2},\,\,w = \frac{\lambda}{4L}w_{2},\]
    giving $2 = w_{1}h_{2}$, again. Hence,
    \[(w_{1},h_{1},w_{2},h_{2})\in\left\{\left(1,\lambda,\frac{2w}{\lambda},2\right),\left(2,\frac{\lambda}{2},\frac{4w}{\lambda},1\right)\right\}.\] So, the cusp width is
    \[k = \text{lcm}\left(\frac{w_{1}}{\gcd(w_{1},h_{1})},\frac{w_{2}}{\gcd(w_{2},h_{2})}\right) \in\left\{\frac{w_{2}}{2},w_{2},2w_{2}\right\}.\]
\end{proof}

Therefore, these $O(n)$ butterfly moves between reduced prototypes each require $O(n)$ applications of $T$ or $S$ to shadow.

So we obtain a diameter bound of $O(n^{2}\log n)$.

\subsubsection{$D\equiv 0\!\!\mod\! 8$} By Theorem~\ref{thm:LN-9.1}, we know that $\mathcal{Q}_{D}$ is connected. The proof of connectivity argues that there exists an $e\in\mathcal{S}_{D}$ that is sent to itself by a sequence of butterfly moves of odd length. This sends $(w,h,t,e,+)$ to $(w,h,t,e,-)$ and the connectivity of $\mathcal{S}_{D}$ from Theorem~\ref{thm:H4-SD} completes the proof.

The required paths of odd length are:
\begin{itemize}
    \item $B_{2}(2k-1,1,0,-4) = (2k-1,1,0,-4)$ if $D = 8 + 16k$, $k\geq 1$;
    \item if $D = 32k$, then
    \[(4k - 2,1,0,-4) \xrightarrow{B_2} (2k - 1,2,0,-4) \xrightarrow{B_\infty} (2k - 1,2,0,-4) \xrightarrow{B_1} (4k - 2,1,0,-4);\]
    \item if $D = 16+32k$, $k > 1$, then
    \[
( 4 k - 6 , 1 , 0 , - 8 )\xrightarrow{B_2}( 2 k + 1 , 2 , 0 , 0 ) \xrightarrow{B_\infty} ( 2 k - 3 , 2 , 0 , - 8 ) \xrightarrow{B_2} ( 4 k - 6 , 1 , 0 , - 8 ).\]
\end{itemize}

As above, we use $O(\log n)$ applications of $B_{1}$ to send $(w,h,t,e,\varepsilon)$ to $(w',1,0,e',\varepsilon')$. This requires $O(n^{2}\log n)$ applications of $T$ or $S$ to shadow. Now, move to the target prototype in $\mathcal{S}_{D}$, which will be one of the starting prototypes in the list of paths above. This prototype can be reached with $O(n)$ butterfly moves each costing $O(n)$ to shadow --- so a cost of $O(n^2)$. If we need to change $\varepsilon$ in order to reach the target origami, we shadow one of the odd length sequences above. Each sequence requires at most $O(n^{2})$ applications of $T$ or $S$. Indeed, the cusp width is at worst $O(n^2)$ for each prototype and, since $q$ is bounded or equal to $\infty$ in each case, the Euclidean algorithm process costs only $O(n)$. We have now arrived at the target origami after possibly $O(n)$ applications of $T$ to move inside the cusp.

Hence, we again obtain a diameter bound of $O(n^{2}\log n)$.

\subsection{$D\equiv 4 \!\!\mod \! 8$} The proof of Theorem~\ref{thm:LN-9.1} deals with this case using the path $$B_{1}\left(\frac{D-4}{8},1,0,-2\right) = \left(\frac{D-4}{8},1,0,-2\right).$$ That is, this is the odd length path required to swap $\varepsilon$. This requires $O(n)$ applications of $T$ or $S$ to shadow --- indeed, it is a single application of $B_1$ on a reduced prototype.

If $D\not\equiv 4\!\!\mod \!16$, then $\mathcal{S}_{D}$ is connected and a similar argument to those above gives a diameter bound of $O(n^{2}\log n)$.

Otherwise, $\mathcal{S}_{D}$ is disconnected. In this case, to reach the above target prototype, we may be required to use one of the paths given below Theorem~\ref{thm:H4-SD} above in order to change component. Each such path costs $O(n^{2})$ to shadow --- indeed, each sequence is a bounded length sequence of butterfly moves with bounded or infinite $q$ between prototypes corresponding to cusps of width at most $O(n^2)$. All other costing is the same as before. Hence, we once again obtain a diameter bound of $O(n^{2}\log n)$.

This completes the proof of Theorem~\ref{thm:main-H4-H6} for the case of $\calH(4)$.

\section{Prym loci in $\calH(6)$}\label{sec:H6}

Similar to the previous section, the Prym locus $\Omega\mathcal{E}_{D}(6)$ is the subset of $\calH(6)$ consisting of those $(M,\omega)$ for which $M$ admits a holomorphic involution $\iota$ with 2 fixed points taking $\omega$ to $-\omega$, and admitting real multiplication by $\mathcal{O}_{D}$ with $\mathcal{O}_{D}\cdot\omega\subset\C\cdot\omega$. Lanneau--Nguyen~\cite{LN20} classify the $\GL^{+}(2,\R)$ connected components of $\Omega\mathcal{E}_{D}(6)$ and we again direct the reader to their paper for more details on Prym loci. They prove that, for $D\equiv 0,1\!\!\mod\! 4$ with $D\neq 4,9$,  $\Omega\mathcal{E}_{D}(6)$ is non-empty and connected. Moreoever, they show that $\Omega\mathcal{E}_{4}(6)$ and $\Omega\mathcal{E}_{9}(6)$ are both empty.

The relevant result for origamis is the following.

\begin{theorem}[{\cite[Theorem 1.2]{LN20}}]
    Fix an even $n\geq 8$ and let $X$ be a primitive origami that is a Prym eigenform in $\calH(6)$. Then there is a single $\SL(2,\Z)$-orbit of such origamis and the eigenforms have discriminant $D = \frac{n^{2}}{4}$.

    If $n$ is odd, then there are no such origamis.
\end{theorem}

Here, the origamis have HLK-invariant $(1,[0,0,0])$.

\subsection{Surface parameters}

In the Prym locus in $\calH(6)$, an origami can have two, or four cylinders (since there can only be one fixed point that is not the zero of the differential and a fixed cylinder gives rise to two fixed points). The possible cusp representatives for two-cylinder origamis are shown in Figure~\ref{fig:H6-2-cyl}. Here, $(a,b,c)$ is a triple of positive integers with $2a+4b+2c = n$ in the first case, $2a+4b+4c = n$ in the second, and $\gcd(a,b,c) = 1$ in both cases. Two-cylinder cusps have width equal to $\frac{n}{2}$.

Origamis with four cylinders have one of the two forms shown in Figure~\ref{fig:H6-params}. As above, the cusp width of such an origami is
\[\text{lcm}\left(\frac{w_{1}}{\gcd(w_{1},h_{1})},\frac{w_{2}}{\gcd(w_{2},h_{2})}\right) = O(n^2).\]

Moreover, we see that the orbits will have size $O(n^{3})$.

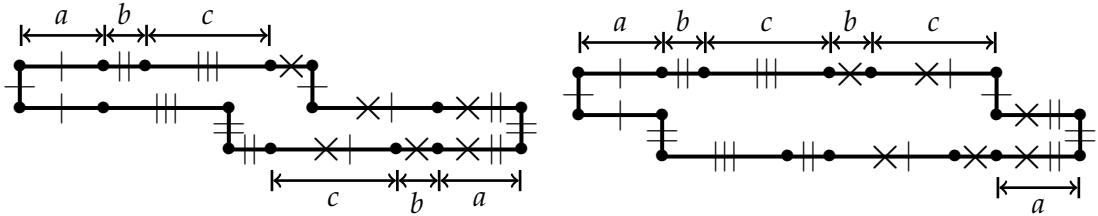
\begin{figure}[b]
    \centering
    \begin{tikzpicture}[scale = 0.55, line width = 1.5pt]
        \draw (0,0) -- node[rotate = 90]{$|$}(0,1) -- node{$|$}(2,1) -- node{$||$}(3,1) -- node{$|||$}(6,1) -- node{$\bigtimes$}(7,1) -- node[rotate = 90]{$|$}(7,0) -- node{$\bigtimes|$}(10,0) -- node{$\bigtimes||$}(12,0) -- node[rotate = 90]{$||$}(12,-1) -- node{$\bigtimes||$}(10,-1) -- node{$\bigtimes$}(9,-1) -- node{$\bigtimes|$}(6,-1) -- node{$||$}(5,-1) -- node[rotate = 90]{$||$}(5,0) -- node{$|||$}(2,0) -- node{$|$} cycle;
        \draw[|<->,line width = 1pt] (0,1.75) -- node[above]{$a$}(2,1.75);
        \draw[|<->,line width = 1pt] (2,1.75) -- node[above]{$b$}(3,1.75);
        \draw[|<->|,line width = 1pt] (3,1.75) -- node[above]{$c$}(6,1.75);
        \draw[|<->,line width = 1pt] (6,-1.75) -- node[below]{$c$}(9,-1.75);
        \draw[|<->,line width = 1pt] (9,-1.75) -- node[below]{$b$}(10,-1.75);
        \draw[|<->|,line width = 1pt] (10,-1.75) -- node[below]{$a$}(12,-1.75);
        \foreach \x/\y in {0/0,0/1,2/1,3/1,6/1,7/1,7/0,10/0,12/0,12/-1,10/-1,9/-1,6/-1,5/-1,5/0,2/0}{
            \node at (\x,\y) {$\bullet$};
        }
    \end{tikzpicture}
    \begin{tikzpicture}[scale = 0.55, line width = 1.5pt]
        \draw (0,0) -- node[rotate = 90]{$|$}(0,1) -- node{$|$}(2,1) -- node{$||$}(3,1) -- node{$|||$}(6,1) -- node{$\bigtimes$}(7,1) -- node{$\bigtimes|$}(10,1) -- node[rotate = 90]{$|$}(10,0) -- node{$\bigtimes||$}(12,0) -- node[rotate = 90]{$||$}(12,-1) -- node{$\bigtimes||$}(10,-1) -- node{$\bigtimes$}(9,-1) -- node{$\bigtimes|$}(6,-1) -- node{$||$}(5,-1) -- node{$|||$}(2,-1) -- node[rotate = 90]{$||$}(2,0) -- node{$|$} cycle;
        \draw[|<->,line width = 1pt] (0,1.75) -- node[above]{$a$}(2,1.75);
        \draw[|<->,line width = 1pt] (2,1.75) -- node[above]{$b$}(3,1.75);
        \draw[|<->,line width = 1pt] (3,1.75) -- node[above]{$c$}(6,1.75);
        \draw[|<->,line width = 1pt] (6,1.75) -- node[above]{$b$}(7,1.75);
        \draw[|<->|,line width = 1pt] (7,1.75) -- node[above]{$c$}(10,1.75);
        \draw[|<->|,line width = 1pt] (10,-1.75) -- node[below]{$a$}(12,-1.75);
        \foreach \x/\y in {0/0,0/1,2/1,3/1,6/1,7/1,10/1,10/0,12/0,12/-1,10/-1,9/-1,6/-1,5/-1,2/-1,2/0}{
            \node at (\x,\y) {$\bullet$};
        }
        \node at (0,1.75){$\,$};
        \node at (0,-1.75){$\,$};
    \end{tikzpicture}
    \caption{Two-cylinder cusp representatives in the Prym locus of $\calH(6)$.}
    \label{fig:H6-2-cyl}
\end{figure}

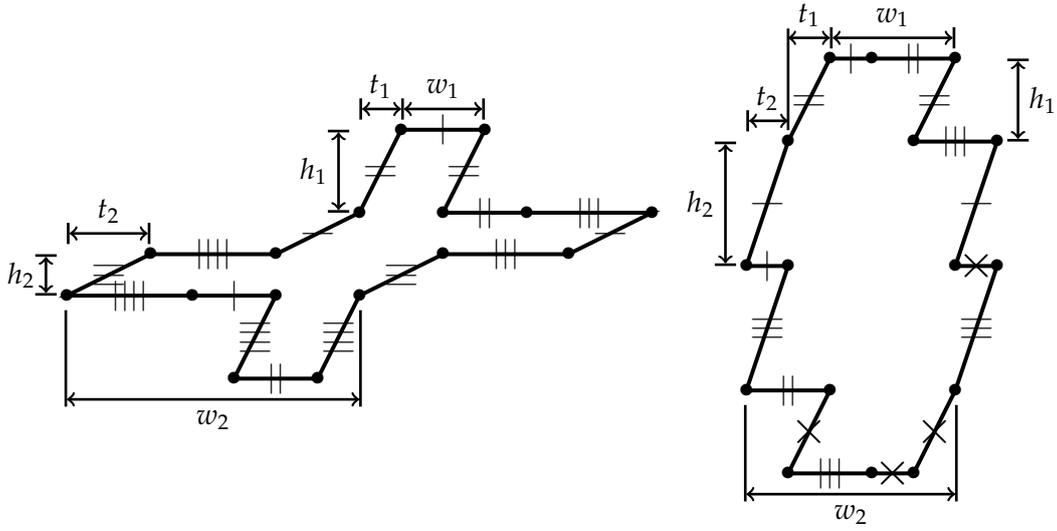
\begin{figure}[t]
    \centering
    \begin{tikzpicture}[scale = 0.55, line width = 1.5pt]
        \draw (0,0) -- node[rotate = 90]{$|$}(2,1) -- node[rotate = 90]{$||$}(3,3) -- node{$|$}(5,3) -- node[rotate = 90]{$||$}(4,1) -- node{$||$}(6,1) -- node{$|||$}(9,1) -- node[rotate = 90]{$|$}(7,0) -- node{$|||$}(4,0) -- node[rotate = 90]{$|||$}(2,-1) -- node[rotate = 90]{$||||$}(1,-3) -- node{$||$}(-1,-3) -- node[rotate = 90]{$||||$}(0,-1) -- node{$|$}(-2,-1) -- node{$||||$}(-5,-1) -- node[rotate = 90]{$|||$}(-3,0) -- node{$||||$} cycle;
        \draw[|<->|,line width = 1pt] (1.5,1) -- node[left]{$h_{1}$}(1.5,3);
        \draw[|<->|,line width = 1pt] (-5.5,-1) -- node[left]{$h_{2}$}(-5.5,0);
        \draw[|<->|,line width = 1pt] (3,3.5) -- node[above]{$w_{1}$}(5,3.5);
        \draw[<->,line width = 1pt] (-5,-3.5) -- node[below]{$w_{2}$}(2,-3.5);
        \draw[|<->,line width = 1pt] (2,3.5) -- node[above]{$t_{1}$}(3,3.5);
        \draw[|<->|,line width = 1pt] (-5,0.5) -- node[above]{$t_{2}$}(-3,0.5);
        \draw[line width = 1pt] (-5.025,-3.675) -- (-5.025,-1.35);
        \draw[line width = 1pt] (2.025,-3.675) -- (2.025,-1.35);
        \foreach \x/\y in {0/0,2/1,3/3,5/3,4/1,6/1,9/1,7/0,4/0,2/-1,1/-3,-1/-3,0/-1,-2/-1,-5/-1,-3/0}{
            \node at (\x,\y) {$\bullet$};
        }
        \node at (2,6.5) {$\,$};
        \node at (2,-6.5) {$\,$};
    \end{tikzpicture}
    \begin{tikzpicture}[scale = 0.55, line width = 1.5pt]
        \draw (0,0) -- node[rotate = 90]{$|$}(1,3) -- node[rotate = 90]{$||$}(2,5) -- node{$|$}(3,5) -- node{$||$}(5,5) -- node[rotate = 90]{$||$}(4,3) -- node{$|||$}(6,3) -- node[rotate = 90]{$|$}(5,0) -- node{$\bigtimes$}(6,0) -- node[rotate = 90]{$|||$}(5,-3) -- node[rotate = 90]{$\bigtimes$}(4,-5) -- node{$\bigtimes$}(3,-5) -- node{$|||$}(1,-5) -- node[rotate = 90]{$\bigtimes$}(2,-3) -- node{$||$}(0,-3) -- node[rotate = 90]{$|||$}(1,0) -- node{$|$} cycle;
        \draw[|<->|,line width = 1pt] (6.5,3) -- node[right]{$h_{1}$}(6.5,5);
        \draw[|<->|,line width = 1pt] (-0.5,0) -- node[left]{$h_{2}$}(-0.5,3);
        \draw[|<->|,line width = 1pt] (2,5.5) -- node[above]{$w_{1}$}(5,5.5);
        \draw[<->,line width = 1pt] (0,-5.5) -- node[below]{$w_{2}$}(5,-5.5);
        \draw[<->,line width = 1pt] (1,5.5) -- node[above]{$t_{1}$}(2,5.5);
        \draw[|<->,line width = 1pt] (0,3.5) -- node[above]{$t_{2}$}(1,3.5);
        \draw[line width = 1pt] (-0.025,-5.675) -- (-0.025,-3.35);
        \draw[line width = 1pt] (5.025,-5.675) -- (5.025,-3.35);
        \draw[line width = 1pt] (1,3.5-0.225) -- (1,5.725);
        \foreach \x/\y in {0/0,1/3,2/5,3/5,5/5,4/3,6/3,5/0,6/0,5/-3,4/-5,3/-5,1/-5,2/-3,0/-3,1/0}{
            \node at (\x,\y) {$\bullet$};
        }
    \end{tikzpicture}
    \caption{Four-cylinder surface parameters in $\calH(6)$ corresponding to prototypes of type $A$ and $B$, respectively.}
    \label{fig:H6-params}
\end{figure}

\subsection{Prototypes and butterfly moves}

\begin{figure}
    \centering
    \begin{tikzpicture}[scale = 0.45, line width = 1.5pt]
        \draw (0,0) -- (2,1) -- (2,3) -- (4,3) -- (4,1) -- (6,1) -- (10.25,1) -- (8.25,0) -- (4,0) -- (2,-1) -- (2,-3) -- (0,-3) -- (0,-1) -- (-2,-1) -- (-6.25,-1) -- (-4.25,0) -- cycle;
        \draw[|<->|,line width = 1pt] (1.25,1) -- node[left]{$\frac{\lambda}{2}$}(1.25,3);
        \draw[|<->|,line width = 1pt] (2,3.75) -- node[above]{$\frac{\lambda}{2}$}(4,3.75);
        \node[left] at (-6.25,-1) {$(0,0)$};
        \node[below right] at (2,-0.75) {$(\frac{w}{2},0)$};
        \node[above] at (-4.25,0) {$(\frac{t}{2},\frac{h}{2})$};
        \foreach \x/\y in {0/0,2/1,2/3,4/3,4/1,6/1,10.25/1,8.25/0,4/0,2/-1,2/-3,0/-3,0/-1,-2/-1,-6.25/-1,-4.25/0}{
            \node at (\x,\y) {$\bullet$};
        }
        \draw[purple] (0,0) -- (10.25,1);
        \node at (2,-5.25) {$\,$};
        \node at (2,6) {$\,$};
    \end{tikzpicture}
    \begin{tikzpicture}[scale = 0.45, line width = 1.5pt]
        \draw (0,0) -- (1,2) -- (1,5) -- (2,5) -- (4,5) -- (4,2) -- (6,2) -- (5,0) -- (6,0) -- (5,-2) -- (5,-5) -- (4,-5) -- (2,-5) -- (2,-2) -- (0,-2) -- (1,0) -- cycle;
        \node[below left] at (0,-2) {$(0,0)$};
        \node[right] at (5,-2) {$(\frac{w}{2},0)$};
        \node[right] at (6,0) {$(\frac{w}{2}+\frac{t}{2},\frac{h}{2})$};
        \draw[|<->|,line width = 1pt] (0.25,2) -- node[left]{$\frac{\lambda}{2}$}(0.25,5);
        \draw[|<->|,line width = 1pt] (1,5.75) -- node[above]{$\frac{\lambda}{2}$}(4,5.75);
        \foreach \x/\y in {0/0,1/2,1/5,2/5,4/5,4/2,6/2,5/0,6/0,5/-2,5/-5,4/-5,2/-5,2/-2,0/-2,1/0}{
            \node at (\x,\y) {$\bullet$};
        }
    \end{tikzpicture}
    \caption{Four cylinder prototypes of type $A$ and $B$, respectively. The direction corresponding to the butterfly move with $q = 1$ is shown.}
    \label{fig:H6-protos}
\end{figure}
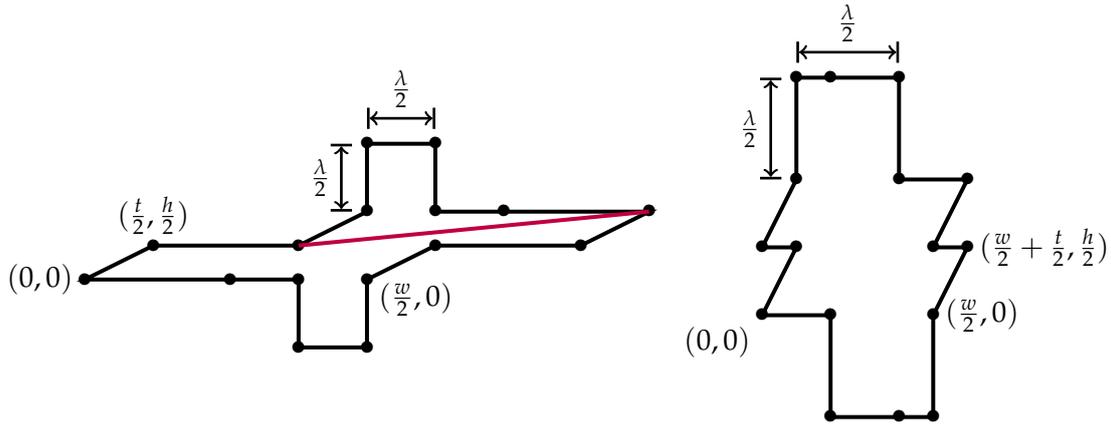

Similar to the previous two sections, Lanneau--Nguyen define prototypes of discriminant $D$ in this setting. They define prototypes of type $A$ and type $B$ as shown in Figure~\ref{fig:H6-protos}. Both prototypes are parameterised by the set
\begin{multline*}
    \mathcal{P}_{D} := \left\{(w,h,t,e)\in\Z^{4}\,:\,w>0,h>0,0\leq t<\gcd(w,h)\right.,
    \\ \left.\gcd(w,h,t,e) = 1, D = e^{2}+4wh, 0< \lambda := \frac{e+\sqrt{D}}{2} < w, \lambda \neq \frac{w}{2}\right\}
\end{multline*}
with those of type $A$ being specifically parameterised by the set
\[\mathcal{P}^{A}_{D} := \left\{(w,h,t,e)\in\mathcal{P}_{D}\,:\,\lambda < \frac{w}{2}\right\}\]
and those of type $B$ parameterised by
\[\mathcal{P}^{B}_{D} := \left\{(w,h,t,e)\in\mathcal{P}_{D}\,:\,\frac{w}{2}< \lambda < w\right\}.\]

In this setting we still have reduced prototypes $(w,1,0,e)$ parameterised by
\[\mathcal{S}^{1}_{D} := \{e\in\Z\,:\,e^2\equiv D\!\!\mod\! 8, e^{2},(e+4)^{2} < D\}.\]
For $D\equiv 1\!\!\mod\!8$, we also have \textit{almost-reduced prototypes} of the form $(w,2,0,e)$, with $w$ even, parameterised by the set
\[\mathcal{S}^{2}_{D} := \{e\in\Z\,:\,e^2\equiv D\!\!\mod\! 16, e^{2},(e+8)^{2} < D\}.\]

Butterfly moves can be defined on the prototypes of type $A$ and map such prototypes among themselves. Here, a value of $q$ is admissible if $q = \infty$, or if $(e+4qh)^{2} < D$. We have the following result.

\begin{proposition}[{\cite[Propositions 2.6 and 2.7]{LN20}}]
    Let $(w,h,t,e)\in\mathcal{P}^{A}_{D}$ and let $q$ be admissible. We have 
    $B_{q}(w,h,t,e) = (w',h',t',e')$ with
    \begin{align*}
        e' &= -e - 4qh \\
        h' &= \gcd(w+qt,qh)
    \end{align*}
    for finite $q$, and
    \begin{align*}
        e' &= -e - 4h \\
        h' &= \gcd(t,h)
    \end{align*}
    for $q = \infty$. In each case, $w'$ is determined by $D = e^2 + 4wh = (e')^{2}+4w'h'$.
\end{proposition}

Similar to the previous cases, the proof relies on the reduction of the matrices
\[\begin{pmatrix}
    h & -t-2e-4qh \\
    -qh & w + qt
\end{pmatrix}\]
for $q < \infty$, or
\[\begin{pmatrix}
    0 & w-2e-4h \\
    -h & t
\end{pmatrix}\]
for $q = \infty$, which appear as minors of a larger matrix acting on homology.

\begin{figure}[t]
    \centering
    \begin{tikzpicture}[scale = 0.8, line width = 1.5pt]
        \draw (0,0) -- (0,2) -- (1,2) -- (1,1) -- (6,1) -- (6,0) -- (2,0) -- (2,-2) -- (1,-2) -- (1,-1) -- (-4,-1) -- (-4,0) -- cycle;
        \draw[gray,dashed] (0,1) -- (1,1);
        \draw[gray,dashed] (0,0) -- (2,0);
        \draw[gray,dashed] (1,-1) -- (2,-1);
        \foreach \i in {2,3,4,5,6}{
            \draw[gray,dashed] (\i-1,0) -- (\i-1,1);
            \node[gray] at (\i-1.5,0.5) {$\i$};
        }
        \foreach \i in {9,10,11,12,13}{
            \draw[gray,dashed] (\i-12,-1) -- (\i-12,0);
            \node[gray] at (\i-11.5,-0.5) {$\i$};
        }
        \node[gray] at (0.5,1.5) {$1$};
        \node[gray] at (5.5,0.5) {$7$};
        \node[gray] at (-3.5,-0.5) {$8$};
        \node[gray] at (1.5,-1.5) {$14$};
        \foreach \x/\y in {0/0,0/1,0/2,1/2,1/1,2/1,6/1,6/0,2/0,2/-1,2/-2,1/-2,1/-1,0/-1,-4/-1,-4/0}{
            \node at (\x,\y) {$\bullet$};
        }
        \draw[line width = 0pt, fill = blue, opacity = 0.3] (0,0) -- (3,1) -- (2,1) -- (0,1/3) -- cycle;
        \draw[line width = 0pt, fill = blue, opacity = 0.3] (2,0) -- (3,0) -- (6,1)-- (5,1) -- (2,0);
        \draw[line width = 0pt, fill = blue, opacity = 0.3] (5,0) -- (6,0) -- (6,1/3) -- (5,0);
        \draw[color = purple] (0,0) -- (3,1);
        \draw[color = purple] (3,0) -- (6,1);
    \end{tikzpicture}
    \caption{An origami corresponding to the prototype $(6,1,0,-5)\in\mathcal{P}^{A}_{49}$. The direction corresponding to the butterfly move $B_{2}$ and one of the resulting simple cylinders are shown in purple and blue, respectively.}
    \label{fig:prym6-sts-example-1}
\end{figure}
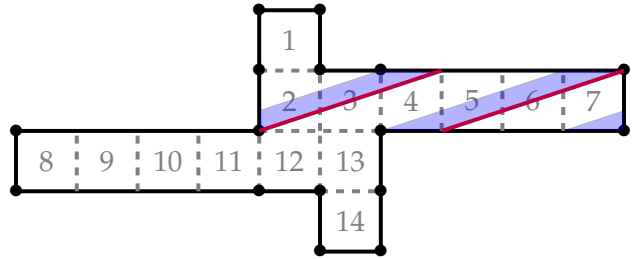

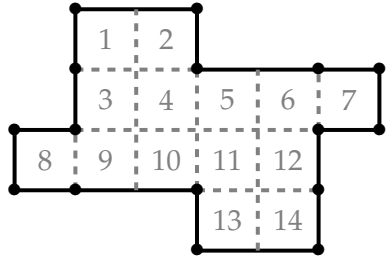
\begin{figure}[t]
    \centering
    \begin{tikzpicture}[scale = 0.8, line width = 1.5pt]
        \draw (0,0) -- (0,2) -- (2,2) -- (2,1) -- (5,1) -- (5,0) -- (4,0) -- (4,-2) -- (2,-2) -- (2,-1) -- (-1,-1) -- (-1,0) -- cycle;
        \foreach \i in {1,2}{
            \node[gray] at (\i-0.5,1.5) {$\i$};
        }
        \foreach \i in {3,4,5,6,7}{
            \node[gray] at (\i-2.5,0.5) {$\i$};
        }
        \foreach \i in {8,9,10,11,12}{
            \node[gray] at (\i-8.5,-0.5) {$\i$};
        }
        \foreach \i in {13,14}{
            \node[gray] at (\i-10.5,-1.5) {$\i$};
        }
        \foreach \i in {0,1}{
            \draw[gray,dashed] (4-4*\i,-\i) -- (4-4*\i,1-\i);
            \draw[gray,dashed] (1+2*\i,-1-\i) -- (1+2*\i,2-\i);
            \draw[gray,dashed] (0+2*\i,1-2*\i) -- (2+2*\i,1-2*\i);
        }
        \draw[gray,dashed] (0,0) -- (4,0);
        \draw[gray,dashed] (2,1) -- (2,-1);
        \foreach \x/\y in {0/0,0/1,0/2,2/2,2/1,4/1,5/1,5/0,4/0,4/-1,4/-2,2/-2,2/-1,0/-1,-1/-1,-1/0}{
            \node at (\x,\y) {$\bullet$};
        }
    \end{tikzpicture}
    \caption{An origami corresponding to the prototype $(5,2,0,-3)\in\mathcal{P}_{49}$.}
    \label{fig:prym6-sts-example-2}
\end{figure} 

Consider the origami $((2,3,4,5,6,7)(8,9,10,11,12,13),(1,12,2)(3,13,4))$ shown in Figure~\ref{fig:prym6-sts-example-1}. It can be checked that this origami corresponds to the prototype $(6,1,0,-5)\in\mathcal{P}^{A}_{49}$. In particular, $q = 2$ is admissible. The butterfly move direction $(w_{2}+qt_{2},qh_{2}) = (6,2)$ can be seen in Figure~\ref{fig:prym6-sts-example-1}. Shadowing the butterfly move, we perform $S^{-1}\circ T^{-2}$. The cusp representative of the resulting origami is $$((1,2)(3,4,5,6,7)(8,9,10,11,12)(13,14),(1,9,3)(2,10,4)(5,13,11)(6,14,12)),$$ shown in Figure~\ref{fig:prym6-sts-example-2}, and it can be checked that it does indeed correspond to the prototype $B_{2}(6,1,0,-5) = (5,2,0,-3)$.

As in the genus two case above, we see that shadowing the butterfly move costs $O(n^{2})$ to move to the cusp representative followed by a Euclidean algorithm operation costing $O(\max\{w_{2}+qt_{2},qh_{2}\})$, for finite $q$, or $O(\max\{t_{2},h_{2}\})$ for $q = \infty$.

\subsection{The classification proof and implied algorithm}

The classification proof in this setting can be sketched as follows:
\begin{enumerate}
    \item Prove that every translation surface $X$ in $\Omega\mathcal{E}_{D}(6)$ has a prototype of type $A$ in its $\GL^{+}(2,\R)$-orbit corresponding to a prototype $(w,h,t,e)$ in $\mathcal{P}_{D}^{A}$.
    \item Prove that every prototype in $\mathcal{P}_{D}^{A}$ can be sent by a sequence of butterfly moves to a reduced or almost-reduced prototype.
    \item Prove that reduced or almost-reduced prototypes can be connected to one another to form at most two components inside $\mathcal{P}_{D}^{A}$. Hence, $\mathcal{P}_{D}^{A}$ has at most two components.
    \item Prove that the two connected components of $\mathcal{P}_{D}^{A}$ can be connected by an element of $\GL^{+}(2,\R)$ by finding a surface with two different cylinder directions corresponding to a prototype in each component, respectively. The element of $\GL^{+}(2,\R)$ is the element that moves from one prototype to the intermediate surface and then to the second prototype. We shall call this intermediate surface the \textit{bridge surface}.
\end{enumerate}

In this setting, we can choose the target prototypes $(\frac{D}{4},1,0,0)$ if $D \equiv 0\!\!\mod 4$, or $(\frac{D-1}{4},1,0,1)$ if $D \equiv 1\!\!\mod 4$. The target origamis will be the cusp representatives of the cusps associated to these prototypes.

Since $\mathcal{P}_{D}^{A}$ has two connected components, we will also need a \textit{pseudo-target} prototype in the components not containing the target. For $D$ even, this will be $(\frac{D-4}{4},1,0,2)$, and for $D$ odd this will be $(\frac{D-1}{8},2,0,1)$. The associated cusp representatives will be called \textit{pseudo-target origamis}.

The algorithm for connecting origamis to the target origami is then as follows:
\begin{enumerate}
    \item Move to a four-cylinder origami corresponding to a prototype in $\mathcal{P}_{D}^{A}$.
    \item Shadow the butterfly moves taking you to a reduced or almost-reduced prototype.
    \item If the resulting prototype is in the same connected component of $\mathcal{P}_{D}^{A}$ as the target prototype, then shadow the sequence of butterfly moves that takes you to the target prototype and reach the target origami.
    \item If the reduced prototype is not in the same connected component of $\mathcal{P}_{D}^{A}$ as the target prototype, then shadow the sequence of butterfly moves that takes you to the cusp corresponding to one of the prototypes of the bridge surface. To do this, travel via the pseudo-target origami. Move through the origami corresponding to the bridge surface into the cusp corresponding to the other prototype of the bridge surface. Now, as above, repeat steps (2) and (3) to reach the target origami.
\end{enumerate}

\subsection{Bounding the number of butterfly moves}

Again, the argument is similar to that given in Subsection~\ref{subsec:butterfly-bound}.

\subsubsection{Connecting to (almost-)reduced prototypes} Given a prototype $(w,h,t,e)\in\mathcal{P}^{A}_{D}$ we can apply the butterfly move $B_{1}$ $O(\log n)$ times to achieve a reduced prototype $(w',1,0,e')$ or an almost-reduced prototype $(w',2,0,e')$ when $D\equiv 1\!\!\mod\! 8$. Indeed, the key minor we are reducing is the matrix
\[\begin{pmatrix}
    h & -t-2e-4h \\
    -h & w+t
\end{pmatrix}.\]
If $\gcd(w+t,h) < h$, then we have $h' \leq \frac{h}{2}$. Otherwise, if $\gcd(w+t,h) = h$, we obtain
\[\begin{pmatrix}
    h & -t-2e-4h \\
    -h & w+t
\end{pmatrix}\xrightarrow[]{\text{basis reduction}} \begin{pmatrix}
    w-2e-4h & 0\\
    0 & h
\end{pmatrix}\]
corresponding to the prototype $(w',h',t',e') = (w-2e-4h,h,0,-e-4h)$.

Applying $B_{1}$ again, if $\gcd(w'+t',h') = \gcd(w',h') < h'$, we obtain $(w'',h'',t'',e'')$ with $h'' \leq \frac{h'}{2} = \frac{h}{2}$, or we have $\gcd(w',h') = h' = h$. In the latter case, similar to the argument in Subsection~\ref{subsec:butterfly-bound}, we obtain
\[\begin{pmatrix}
    h' & -2e'-4h' \\
    -h' & w'
\end{pmatrix}\xrightarrow[]{\text{basis reduction}} \begin{pmatrix}
    w'-2e'-4h' & 0 \\
    0 & h'
\end{pmatrix}\]
so that, since $\gcd(w''+t'',h'') = \gcd(w'-2e'-4h',h') = \gcd(2e',h') \in \{1,2\}$, a further application of $B_{1}$ achieves $(w''',1,0,e''')$ or $(w''',2,0,e''')$.

In the latter case, if $D$ is even, then $e'''$ is even and so we must have $\gcd(w''',2) = 1$. Hence, a further application of $B_{1}$ arrives at a reduced prototype. If $D\equiv 1\!\!\mod\! 8$, then $w'''$ is either even or odd. In the former case, $(w''',2,0,e''')$ is almost-reduced and we cannot reach a reduced prototype. In the latter case, $\gcd(w''',2) = 1$ and we achieve a reduced prototype after a further application of $B_{1}$.

Therefore, since we again have that $\log h = O(\log n)$, we achieve a reduced prototype or an almost-reduced prototype in $O(\log n)$ applications of $B_{1}$.

\subsubsection{Connecting (almost-)reduced prototypes}

Lanneau--Nguyen proved the following results (we have again ignored the exceptional discriminants since we are interested in asymptotics).

\begin{theorem}[{\cite[Theorem A.2]{LN20}}]\label{thm:LN-A.2}
    For $D$ a large enough discriminant, $S^{1}_{D}$ is non-empty and has
    \begin{itemize}
        \item three components, $\{e\in S^{1}_{D}: e\equiv 0\,\,\text{or}\,\,4\!\!\mod\!8\},\{e\in S^{1}_{D}: e\equiv 2\!\!\mod\!8\}$, and $\{e\in S^{1}_{D}: e\equiv -2\!\!\mod\!8\}$, if $D\equiv 4\!\!\mod\!8$.
        \item two components,
        \begin{itemize}
            \item[-] $\{e\in S^{1}_{D}: e\equiv 1\,\,\text{or}\,\,3\!\!\mod\!8\}$ and $\{e\in S^{1}_{D} : e\equiv -1\,\,\text{or}\,\,-3\!\!\mod\!8\}$, if $D\equiv 1\!\!\mod\!8$,
            \item[-] $\{e\in S^{1}_{D}: e\equiv 0\,\,\text{or}\,\,4\!\!\mod\!8\}$ and $\{e\in S^{1}_{D} : e\equiv 2\,\,\text{or}\,\,-2\!\!\mod\!8\}$, if $D\equiv 0\!\!\mod\!8$,
        \end{itemize}
        \item only one component, otherwise.
    \end{itemize}
    Moreover, if $D\equiv 1\!\!\mod\!8$, then $S^{2}_{D}$ is non-empty and connected.
\end{theorem}

\begin{theorem}[{\cite[Theorem 3.4]{LN20}}]\label{thm:LN-3.4}
    For $D$ a large enough discriminant, $\mathcal{P}^{A}_{D}$ is non-empty and has
    \begin{itemize}
        \item one component if $D\equiv 5\!\!\mod\!8$,
        \item two components, $\{(w,h,t,e)\in\mathcal{P}^{A}_{D}\,:\, e\equiv 0\!\!\mod\! 4\}$ and $\{(w,h,t,e)\in\mathcal{P}^{A}_{D}\,:\,e\equiv 2\!\!\mod\!4\}$ if $D\equiv 0,4\!\!\mod\!8$,
        \item two components, $\mathcal{P}^{A_{i}}_{D} := \{(w,h,t,e)\in\mathcal{P}^{A}_{D}\,:\,e\in\mathcal{S}^{i}_{D}$\}, for $i = 1,2$, if $D\equiv 1\!\!\mod\!8$.
    \end{itemize}
\end{theorem}

In all cases, $|S^{i}_{D}| = O(n)$.

If $D\equiv 0\!\!\mod\!8$, we can connect to the target or pseudo-target prototype in one of the two components of $S^{1}_{D}$ in $O(n)$ butterfly moves. If $D\equiv 4\!\!\mod\!8$, then the proof of Theorem~\ref{thm:LN-3.4} demonstrates that the components $\{e\in S^{1}_{D}: e\equiv 2\!\!\mod\!8\}$, and $\{e\in S^{1}_{D}: e\equiv -2\!\!\mod\!8\}$ are connected by paths in $\mathcal{P}^{A}_{D}$:
\begin{itemize}
    \item if $D = 12 + 16k, k\geq 2$
    $$(4k+2,1,0,-2)\xrightarrow{B_{2}}(2k-3,2,0,-6)\xrightarrow{B_{\infty}}(2k+1,2,0,-5)\xrightarrow{B_{1}}(4k-6,1,0,-6)$$
    \item if $D = 4 + 32k, k\geq 4$
    $$(8k,1,0,2)\xrightarrow{B_{2}}(4k-12,2,1,-10)\xrightarrow{B_{2}}(4k-4,2,1,-6)\xrightarrow{B_{1}}(8k,1,0,-2)$$
    \item if $D = 20 + 32k, k\geq 3$
    $$(8k+4,1,0,2)\xrightarrow{B_{2}}(4k-10,2,1,-10)\xrightarrow{B_{2}}(2k-1,4,0,-6)\xrightarrow{B_{1}}(8k-20,1,0,-10).$$
\end{itemize}
So, we can connect to the target or pseudo-target prototype in one of the components of $\mathcal{P}^{A}_{D}$ listed in Theorem~\ref{thm:LN-A.2} using $O(n)$ butterfly moves between reduced prototypes and possibly one of the paths in the list above. Finally, if $D\equiv 1\!\!\mod\!8$ then the proof of Theorem~\ref{thm:LN-3.4} demonstrates that the two components of $S^{1}_{D}$ can be connected by a path in $\mathcal{P}^{A_{1}}_{D}$:
\begin{itemize}
    \item if $D = 1 + 16k, k\geq 3$
    $$(4k-6,1,0,-5)\xrightarrow{B_{2}}(2k-1,2,0,-3)\xrightarrow{B_{\infty}}(2k-3,2,0,-5)\xrightarrow{B_{1}}(4k-2,1,0,-3)$$
    \item if $D = 9 + 16k, k\geq 3$
    $$(4k-10,1,0,-7)\xrightarrow{B_{2}}(2k+1,2,0,-1)\xrightarrow{B_{\infty}}(2k-5,2,0,-7)\xrightarrow{B_{1}}(4k+2,1,0,-1).$$
\end{itemize}
So reduced prototypes can be connected to the target reduced prototype using $O(n)$ butterfly moves between reduced prototypes and possibly one of the two paths above, and the almost-reduced prototypes can be connected to the pseudo-target almost-reduced prototype in $O(n)$ butterfly moves between almost-reduced prototypes.

So, to summarise the above, we can apply the butterfly move $B_1$ $O(\log n)$ times to reach a reduced or almost-reduced prototype and then a further $O(n)$ butterfly moves between prototypes (possibly with the addition of one of the bounded length paths above) to reach the target or psuedo-target prototype.

\subsection{The resulting diameter bound}

We must now bound the number of applications of $T$ or $S$ required to connect to the target origamis.

Firstly, observe that the cusp representative of a two-cylinder origami has a four-cylinder direction corresponding to a prototype of type $A$ in its vertical direction. Hence, we can reach such a four-cylinder origami in $O(n)$ applications of $T$ or $S$.

If we have a four-cylinder origami corresponding to a prototype of type $B$, then the proof of~\cite[Proposition 2.4]{LN20} gives us that the cusp representative of this origami has a four-cylinder decomposition of type $A$ in one of the following directions:
\begin{itemize}
    \item $(t_1+t_2,h_1+h_2)$;
    \item $(t_1+t_2-w_1,h_1+h_2)$; or
    \item $(t_1+t_2-w_2,h_1+h_2)$.
\end{itemize}
In each case, this direction can be made horizontal in $O(n)$ applications of $T$ or $S$, after reaching the cusp representative at a cost of $O(n^{2})$ applications of $T$.

So, any origami can be taken to a four-cylinder origami corresponding to a prototype of type $A$ at a cost of $O(n^{2})$.

Similar to Lemmas~\ref{lem:H2-red} and~\ref{lem:H4-red}, we will make use of the following lemma when we connect (almost-)reduced prototypes to a target (almost-)reduced prototype by butterfly moves.

\begin{lemma}\label{lem:H6-red}
    An origami corresponding to a reduced prototype $(w,1,0,e)$ in $\mathcal{P}_{D}$ has $h_{2} = 1$ and cusp width $w_{2} = O(n)$. An origami corresponding to an almost-reduced prototype $(w,2,0,e)$ in $\mathcal{P}_{D}$ has $h_{2}\in\{1,2\}$ and cusp width $k\in\{\frac{w_{2}}{2},w_{2},2w_{2}\}$. So, $k = O(n)$.
\end{lemma}

\begin{proof}
    Let $(w,1,0,e)$ be a reduced prototype. The prototype has area $2(\frac{\lambda}{2})^{2} + \frac{w}{2} = n\frac{\lambda}{4}$. Suppose that in achieving the prototype the origami of area $n$ is scaled by $L$ vertically and $\frac{\lambda}{4L}$ horizontally. We then have
    \[\frac{\lambda}{2} = Lh_{1},\,\,\frac{\lambda}{2} = \frac{\lambda}{4L}w_{1},\,\,\frac{1}{2} = Lh_{2},\,\,\frac{w}{2} = \frac{\lambda}{4L}w_{2},\]
    giving $1 = w_{1}h_{2}$. So the origami has surface parameters $w_{1} = 1$, $h_{1} = \lambda$, $w_{2} = \frac{w}{\lambda}$, $h_{2} = 1$ and cusp width
    \[k = \text{lcm}\left(\frac{w_{1}}{\gcd(w_{1},h_{1})},\frac{w_{2}}{\gcd(w_{2},h_{2})}\right) = \text{lcm}\left(\frac{1}{\gcd(1,\lambda)},\frac{w_{2}}{\gcd(w_{2},1)}\right) = w_{2}.\]

    Let $(w,2,0,e)$ be an almost-reduced prototype. The prototype has area $2(\frac{\lambda}{2})^{2} + w = n\frac{\lambda}{4}$. Suppose that in achieving the prototype the origami of area $n$ is scaled by $L$ vertically and $\frac{\lambda}{4L}$ horizontally. We then have
    \[\frac{\lambda}{2} = Lh_{1},\,\,\frac{\lambda}{2} = \frac{\lambda}{4L}w_{1},\,\,1 = Lh_{2},\,\,\frac{w}{2} = \frac{\lambda}{4L}w_{2},\]
    giving $2 = w_{1}h_{2}$. So the origami has surface parameters
    \[(w_{1},h_{1},w_{2},h_{2})\in\left\{\left(1,\lambda,\frac{w}{\lambda},2\right),\left(2,\frac{\lambda}{2},\frac{2w}{\lambda},1\right)\right\},\]
    and cusp width
    \[k = \text{lcm}\left(\frac{w_{1}}{\gcd(w_{1},h_{1})},\frac{w_{2}}{\gcd(w_{2},h_{2})}\right) \in \left\{\frac{w_{2}}{2},w_{2},2w_{2}\right\}.\]
\end{proof}

As in the previous sections, this means that shadowing a butterfly move between reduced or almost-reduced prototypes only costs $O(n)$.

We now handle the possible residue classes of $D$ separately. Recall that, since we have $D$ being a square, we can ignore the $D\equiv 5\!\!\mod 8$ case.

\subsubsection{$D\equiv 0,4\!\!\mod\!8$} In this case, we can reach the target origami or the pseudo-target origami in each component of $\mathcal{P}^{A}_{D}$ using $O(n^{2}\log n)$ application of $T$ or $S$ --- we apply $B_{1}$ $O(\log n)$ times at a cost of $O(n^2\log n)$ followed by $O(n)$ butterfly moves between reduced prototypes at a cost of $O(n^2)$ and possibly one of the paths listed below Theorem~\ref{thm:LN-3.4} which each cost $O(n^{2})$.

To connect the two components of $\mathcal{P}^{A}_{D}$, Lanneau--Nguyen (see~\cite[Theorem 7.1]{LN20}) construct an origami (with height 2) as shown in Figure~\ref{fig:lABC-origami} and demonstrate that the cylinders $\mathcal{C}$ and $\mathcal{C}'$ give rise to prototypes in differing components of $\mathcal{P}^{A}_{D}$. This is the bridge surface.

\begin{figure}[t]
    \centering
    \begin{tikzpicture}[scale = 0.8, line width = 1.5pt]
        \draw (0,0) -- (0,2) -- node[above]{$A$}(3,2) -- node[above]{$B$}(4,2) -- node[above]{$C$}(5.5,2) -- node[above]{$\overline{B}$}(6.5,2) -- node[above]{$\overline{C}$}(8,2) -- (8,0) -- node[above]{$\overline{A}$}(11,0) -- (11,-2) -- node[below]{$\overline{A}$}(8,-2) -- node[below]{$\overline{B}$}(7,-2) -- node[below]{$\overline{C}$}(5.5,-2) -- node[below]{$B$}(4.5,-2) -- node[below]{$C$}(3,-2) -- (3,0) -- node[below]{$A$} cycle;
        \foreach \x/\y in {0/0,0/2,3/2,4/2,5.5/2,6.5/2,8/2,8/0,11/0,11/-2,8/-2,7/-2,5.5/-2,4.5/-2,3/-2,3/0}{
            \node at (\x,\y) {$\bullet$};
        }
        \draw[line width = 0pt, fill = red, opacity = 0.3] (3,2) -- (4,2) --(5.5,-2) -- (4.5,-2) -- cycle;
        \draw[line width = 0pt, fill = blue, opacity = 0.3] (4.5,-2) -- (8,-2+84/26) -- (8,-2+60/26) -- (5.5,-2) -- cycle;
        \draw[line width = 0pt, fill = blue, opacity = 0.3] (0,-2+84/26) -- (0,-2+60/26) -- (1+5/6,2) -- (5/6,2) -- cycle;
        \draw[line width = 0pt, fill = blue, opacity = 0.3] (5/6,0) -- (3,2) -- (4,2) -- (1+5/6,0) -- cycle;
        \node at (0.7,1.35) {$\mathcal{C}$};
        \node at (2.45,1) {$\mathcal{C}$};
        \node at (6.5,-0.65) {$\mathcal{C}$};
        \node at (4.25,0) {$\mathcal{C}'$};
    \end{tikzpicture}
    \caption{An origami with cylinders giving rise to prototypes in different components of $\mathcal{P}^{A}_{D}$. Cylinder $\mathcal{C}$ is shown in blue while cylinder $\mathcal{C}'$ is shown in red.}
    \label{fig:lABC-origami}
\end{figure}
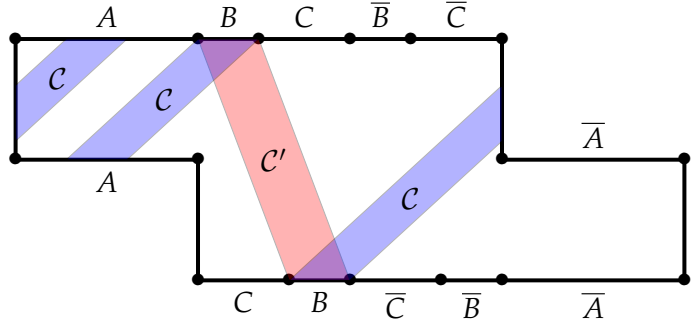

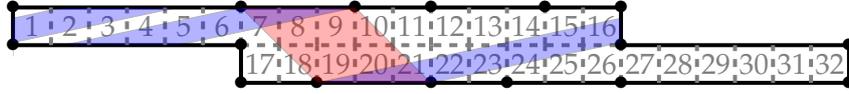
\begin{figure}[t]
    \centering
    \begin{tikzpicture}[scale = 0.5, line width = 1.5pt]
        \draw (0,0) -- (0,1) -- (16,1) -- (16,0) -- (22,0) -- (22,-1) -- (6,-1) -- (6,0) -- cycle;
        \draw[gray,dashed] (6,0) -- (16,0);
        \node[gray] at (0.5,0.5) {$1$};
        \node[gray] at (21.5,-0.5) {$32$};
        \foreach \i in {2,...,16}{
            \draw[gray,dashed] (\i-1,0) -- (\i-1,1);
            \node[gray] at (\i-0.5,0.5) {$\i$};
        }
        \foreach \i in {17,...,31}{
            \draw[gray,dashed] (\i-10,-1) -- (\i-10,0);
            \node[gray] at (\i-10.5,-0.5) {$\i$};
        }
        \foreach \x/\y in {0/0,0/1,6/1,9/1,11/1,14/1,16/1,16/0,22/0,22/-1,16/-1,13/-1,11/-1,8/-1,6/-1,6/0}{
            \node at (\x,\y) {$\bullet$};
        }
        \draw[line width = 0pt, fill = red, opacity = 0.3] (6,1) -- (9,1) -- (11,-1) -- (8,-1) -- (6,1);
        \draw[line width = 0pt, fill = blue, opacity = 0.3] (8,-1) -- (16,-1+24/14) -- (16,-1+15/14) -- (11,-1) -- (8,-1);
        \draw[line width = 0pt, fill = blue, opacity = 0.3] (0,-1+12/7) -- (4/3,1) -- (4/3+3,1) -- (0,-1+15/14) -- (0,-1+12/17);
        \draw[line width = 0pt, fill = blue, opacity = 0.3] (4/3,0) -- (6,1) -- (9,1) -- (4/3+3,0) -- (4/3,0);
    \end{tikzpicture}
    \caption{An origami with cylinder $\mathcal{C}$, shown in blue, giving rise to prototype $(63,1,0,2)$, and cylinder $\mathcal{C}'$, shown in red, giving rise to prototype $(30,2,1,-4)$.}
    \label{fig:lABC-example-1}
\end{figure}

\begin{figure}[t]
    \centering
    \begin{tikzpicture}[scale = 0.45, line width = 1.5pt]
        \draw (0,0) -- (0,10) -- (1,10) -- (1,1) -- (7,1) -- (7,0) -- (2,0) -- (2,-10) -- (1,-10) -- (1,-1) -- (-5,-1) -- (-5,0) -- cycle;
        \draw[gray,dashed] (0,0) -- (2,0);
        \node[gray] at (0.5,9.5) {$1$};
        \node[gray] at (1.5,-9.5) {$32$};
        \foreach \i in {2,...,10}{
            \draw[gray,dashed] (0,11-\i) -- (1,11-\i);
            \node[gray] at (0.5,10.5-\i) {$\i$};
        }
        \foreach \i in {11,...,16}{
            \draw[gray,dashed] (\i-10,0) -- (\i-10,1);
            \node[gray] at (\i-9.5,0.5) {$\i$};
        }
        \foreach \i in {17,...,22}{
            \draw[gray,dashed] (\i-21,-1) -- (\i-21,0);
            \node[gray] at (\i-21.5,-0.5) {$\i$};
        }
        \foreach \i in {23,...,31}{
            \draw[gray,dashed] (1,22-\i) -- (2,22-\i);
            \node[gray] at (1.5,22.5-\i) {$\i$};
        }
        \foreach \x/\y in {0/0,0/1,0/10,1/10,1/1,2/1,7/1,7/0,2/0,2/-1,2/-10,1/-10,1/-1,0/-1,-5/-1,-5/0}{
            \node at (\x,\y) {$\bullet$};
        }
    \end{tikzpicture}
    \begin{tikzpicture}[scale = 0.45, line width = 1.5pt]
        \draw (0,0) -- (0,1) -- (1,4) -- (3,4) -- (2,1) -- (10,1) -- (10,0) -- (4,0) -- (4,-1) -- (3,-4) -- (1,-4) -- (2,-1) -- (-6,-1) -- (-6,0) -- cycle;
        \draw[gray,dashed] (0,0) -- (4,0);
        \foreach \i in {0,1}{
            \draw[gray,dashed] (1+2*\i,4-5*\i) -- (1+2*\i,1-5*\i);
            \draw[gray,dashed] (2,4-5*\i) -- (2,1-5*\i);
        }
        \foreach \i in {0,1,2}{
            \draw[gray,dashed] (\i/3,\i+1) -- (2+\i/3,\i+1);
        }
        \foreach \i in {0,1,2}{
            \draw[gray,dashed] (2-\i/3,-\i-1) -- (4-\i/3,-\i-1);
        }
        \node[gray] at (2.5,3.5) {$1$};
        \node[gray] at (1.5,3.5) {$2$};
        \node[gray] at (0.5,2.5) {$3$};
        \node[gray] at (1.5,2.5) {$4$};
        \node[gray] at (0.5,1.5) {$5$};
        \node[gray] at (1.5,1.5) {$6$};
        \node[gray] at (9.5,0.5) {$16$};
        \node[gray] at (-5.5,-0.5) {$17$};
        \node[gray] at (2.5,-1.5) {$27$};
        \node[gray] at (3.5,-1.5) {$28$};
        \node[gray] at (2.5,-2.5) {$29$};
        \node[gray] at (3.5,-2.5) {$30$};
        \node[gray] at (2.5,-3.5) {$31$};
        \node[gray] at (1.5,-3.5) {$32$};
        \foreach \i in {7,...,15}{
            \draw[gray,dashed] (\i-6,0) -- (\i-6,1);
            \node[gray] at (\i-6.5,0.5) {$\i$};
        }
        \foreach \i in {18,...,26}{
            \draw[gray,dashed] (\i-23,-1) -- (\i-23,0);
            \node[gray] at (\i-22.5,-0.5) {$\i$};
        }
        \foreach \x/\y in {0/0,0/1,1/4,3/4,2/1,4/1,10/1,10/0,4/0,4/-1,3/-4,1/-4,2/-1,0/-1,-6/-1,-6/0}{
            \node at (\x,\y) {$\bullet$};
        }
        \node at (2,10) {$\,$};
        \node at (2,-10) {$\,$};
    \end{tikzpicture}
    \caption{The origami on the left corresponds to prototype $(63,1,0,2)$. The origami on the right corresponds to prototype $(30,2,1,-4)$.}
    \label{fig:lABC-example-2}
\end{figure}
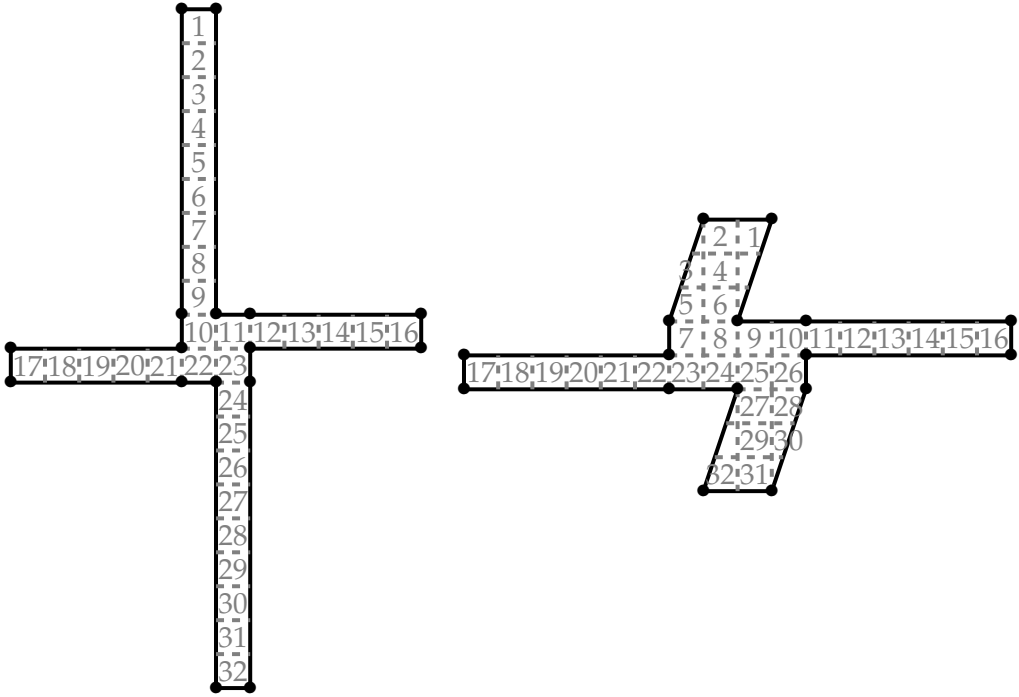

For example, when $n = 32$, the origami
\begin{multline*}
    ((1,2,3,\ldots,16)(17,18,19,\ldots,32),\\
    (7,19,9,21,11,18,8,20,10,17)(12,24,14,26,16,23,14,25,15,22))
\end{multline*}
shown in Figure~\ref{fig:lABC-example-1} has a cylinder $\mathcal{C}$ in direction $(14,3)$ giving rise (after an application of $T^{2}\circ S^{-1}\circ T^{-1}\circ S^{-1}\circ T^{-4}$) to the origami
\[
((10,11,\ldots,16)(17,18,\ldots,23),(1,22,10,9,\ldots,2)(11,32,31,\ldots,23))
\]
on the left of Figure~\ref{fig:lABC-example-2} corresponding to the prototype $(63,1,0,2)$, while the cylinder $\mathcal{C}'$ in direction $(-2,2)$ (after applying $S$) gives rise to the origami
\begin{multline*}
    ((1,2)(3,4)(5,6)(7,8,\ldots,16)(17,18,\ldots,26)(27,28)(29,30)(31,32),\\
    (1,24,8,6,4,2,23,7,5,3)(9,32,30,28,26,10,31,29,27,25))
\end{multline*}
on the right of Figure~\ref{fig:lABC-example-2} corresponding to the prototype $(30,2,1,-4)$.

Letting $l_{A},l_{B}$ and $l_{C}$ denote the lengths of the sides $A,B$ and $C$, respectively, we see that cylinder $\mathcal{C}$ has direction $(l_{A}+2l_{B}+l_{C},3)$ and so can be made horizontal in $O(\max\{l_{A}+2l_{B}+l_{C},3\}) = O(n)$ applications of $T$ or $S$. Similarly, the cylinder $\mathcal{C}'$ has direction $(-l_{C},2)$ and so can also be made horizontal in $O(n)$ applications of $T$ or $S$. Let the resulting origamis be called the \textit{bridge-end origamis}.

So, to connect to the target origami, we either apply $O(n^2\log n)$ applications of $T$ or $S$ to reach it directly, or we first use $O(n^2\log n)$ applications of $T$ or $S$ to reach the pseudo-target origami, then $O(n^2\log n)$ applications of $T$ or $S$ to reach one of the bridge-end origamis, then $O(n)$ applications to transition through the bridge origami to the other bridge-end origami, and finally $O(n^2\log n)$ applications of $T$ or $S$ to reach the the target origami.

Therefore, we obtain the diameter bound $O(n^{2}\log n)$.

\subsubsection{$D\equiv 1\!\!\mod\!8$} As above, the target or pseudo-target origami can be reached at a cost of $O(n^{2}\log n)$. Here, Lanneau--Nguyen connect the two components of $\mathcal{P}^{A}_{D}$ by proving the existence of a prototype in $\mathcal{P}^{A_{2}}_{D}$ that contains a simple cylinder in a non-horizontal direction which gives rise to a prototype in $\mathcal{P}^{A_{1}}_{D}$. To bound the number of applications of $T$ or $S$ required to do this, we must determine the direction of this cylinder inside the corresponding origami.

Lanneau and Nguyen prove the following.

\begin{proposition}[See {\cite[Section 8]{LN20}}]\label{prop:LN20-sec8}
    For any $D \equiv 1\!\!\mod\!8$, with $D$ large enough, there exists $(w,2,0,e)\in \mathcal{S}^{2}_{D}$ such that there is a simple cylinder in the direction
    $$\left(\frac{w}{2}+\left\lfloor\frac{\lambda}{2}\right\rfloor\frac{\lambda}{2},j+1+\frac{\lambda}{2}\right),$$
    for some $j\in\N$, for which the associated prototype lies in $\mathcal{P}^{A_{1}}_{D}$.
\end{proposition}

Now, by Lemma~\ref{lem:H6-red}, we have
\[(w_{1},h_{1},w_{2},h_{2})\in\left\{\left(1,\lambda,\frac{w}{\lambda},2\right),\left(2,\frac{\lambda}{2},\frac{2w}{\lambda},1\right)\right\},\]
where in the first case the prototype is obtained from the origami by scaling horizontally by $\frac{\lambda}{2}$ and vertically by $\frac{1}{2}$, and in the second case by scaling horizontally by $\frac{\lambda}{4}$. Hence, on the origami, the direction of the cylinder becomes
\[\left(\frac{w}{\lambda}+\left\lfloor\frac{\lambda}{2}\right\rfloor,2j+2+\lambda\right) = \left(w_{2}+\left\lfloor\frac{h_{1}}{2}\right\rfloor,2j+2+h_{1}\right),\]
in the first case, or
\[\left(\frac{2w}{\lambda}+2\left\lfloor\frac{\lambda}{2}\right\rfloor,j+1+\frac{\lambda}{2}\right) = (w_{2}+2h_{1},j+1+h_{1}),\]
in the second case.

Moreover, \cite[Lemma 8.3]{LN20} guarantees that
\[\frac{\lambda}{2}+j+1 < \frac{w}{\lambda} + \left\lfloor\frac{\lambda}{2}\right\rfloor \Rightarrow j < w_{2} = O(n).\]
Hence, it requires only $O(n)$ applications of $T$ or $S$ to make the cylinder direction horizontal in the origami.

\begin{figure}[b]
    \centering
    \begin{tikzpicture}[scale = 0.8, line width = 1.5pt]
        \draw (0,0) -- (0,3) -- (1,3) -- (1,2) -- (4,2) -- (4,0) -- (2,0) -- (2,-3) -- (1,-3) -- (1,-2) -- (-2,-2) -- (-2,0) -- cycle;
        \draw[gray, dashed] (0,0) -- (2,0);
        \node[gray] at (0.5,2.5) {$1$};
        \node[gray] at (1.5,-2.5) {$18$};
        \foreach \i in {-2,2}{
            \draw[gray, dashed] (\i/4+1/2,-\i) -- (1+\i/4+1/2,-\i);
            \draw[gray, dashed] (\i/2-1,\i/2) -- (\i/2+3,\i/2);
        }
        \foreach \i in {1,2,3}{
            \draw[gray, dashed] (\i,0) -- (\i,2);
            \draw[gray, dashed] (-2+\i,-2) -- (-2+\i,0);
        }
        \foreach \i in {2,...,5}{
            \node[gray] at (\i-1.5,1.5) {$\i$};
        }
        \foreach \i in {6,...,9}{
            \node[gray] at (\i-5.5,0.5) {$\i$};
        }
        \foreach \i in {10,...,13}{
            \node[gray] at (\i-11.5,-0.5) {$\i$};
        }
        \foreach \i in {14,...,17}{
            \node[gray] at (\i-15.5,-1.5) {$\i$};
        }
        \foreach \x/\y in {0/0,0/2,0/3,1/3,1/2,2/2,4/2,4/0,2/0,2/-2,2/-3,1/-3,1/-2,0/-2,-2/-2,-2/0}{
            \node at (\x,\y) {$\bullet$};
        }
        \draw[line width = 0pt, fill = blue, opacity = 0.3] (-8/7,-2) -- (8/7,2) -- (10/7,2) -- (-6/7,-2) -- cycle;
        \draw[line width = 0pt, fill = blue, opacity = 0.3] (8/7,-3) -- (2,-1.5) -- (2,-2) -- (10/7,-3) -- cycle;
        \draw[line width = 0pt, fill = blue, opacity = 0.3] (-2,-1.5) -- (-8/7,0) -- (-6/7,0) -- (-2,-2) -- cycle;
    \end{tikzpicture}
    \begin{tikzpicture}[scale = 0.8, line width = 1.5pt]
        \draw (0,0) -- (0,2) -- (2,2) -- (2,1) -- (7,1) -- (7,0) -- (4,0) -- (4,-2) -- (2,-2) -- (2,-1) -- (-3,-1) -- (-3,0) -- cycle;
        \draw[gray, dashed] (0,0) -- (4,0);
        \node[gray] at (6.5,0.5) {$9$};
        \node[gray] at (-2.5,-0.5) {$10$};
        \foreach \i in {-1,1}{
            \draw[gray, dashed] (\i+1,-\i) -- (\i+3,-\i);
            \draw[gray, dashed] (\i+2,-\i) -- (\i+2,-2*\i);
        }
        \foreach \i in {1,2}{
            \node[gray] at (-0.5+\i,1.5) {$\i$}; 
        }
        \foreach \i in {17,18}{
            \node[gray] at (\i-14.5,-1.5) {$\i$}; 
        }
        \foreach \i in {3,...,8}{
            \node[gray] at (\i-2.5,0.5) {$\i$};
            \draw[gray, dashed] (\i-2,0) -- (\i-2,1);
        }
        \foreach \i in {11,...,16}{
            \node[gray] at (\i-12.5,-0.5) {$\i$};
            \draw[gray, dashed] (\i-13,0) -- (\i-13,-1);
        }
        \foreach \x/\y in {0/0,0/1,0/2,2/2,2/1,4/1,7/1,7/0,4/0,4/-1,4/-2,2/-2,2/-1,0/-1,-3/-1,-3/0}{
            \node at (\x,\y) {$\bullet$};
        }
        \draw[line width = 0pt, fill = red, opacity = 0.3] (-2.25,-1) -- (2.25,1) -- (3.75,1) -- (-0.75,-1) -- cycle;
        \draw[line width = 0pt, fill = red, opacity = 0.3] (2.25,-2) -- (4,-2+7/9) -- (4,-2+1/9) -- (3.75,-2) -- cycle;
        \draw[line width = 0pt, fill = red, opacity = 0.3] (2,-2+7/9) -- (4,-1+2/3) -- (4,-1) -- (2,-2+1/9) -- cycle;
        \draw[line width = 0pt, fill = red, opacity = 0.3] (-3,-1+2/3) -- (-2.25,0) -- (-0.75,0) -- (-3,-1) -- cycle;
        \node at (2,-3) {$\,$};
        \node at (2,3) {$\,$};
    \end{tikzpicture}
    \caption{On the left, an origami corresponding to the prototype $(4,2,0-7)$. A cylinder in the direction $(4,7)$ is shown in blue. On the right, an origami corresponding to the prototype $(7,2,0,-5)$. A cylinder in the direction $(9,4)$ is shown in red.}
    \label{fig:slope-example-1}
\end{figure}
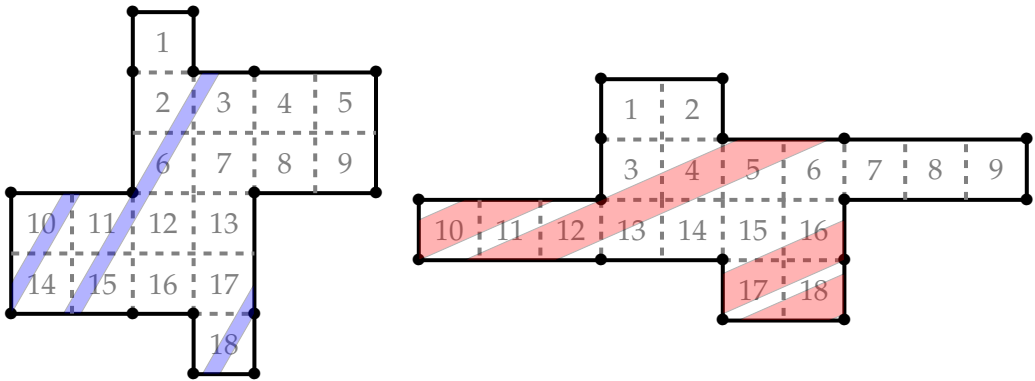

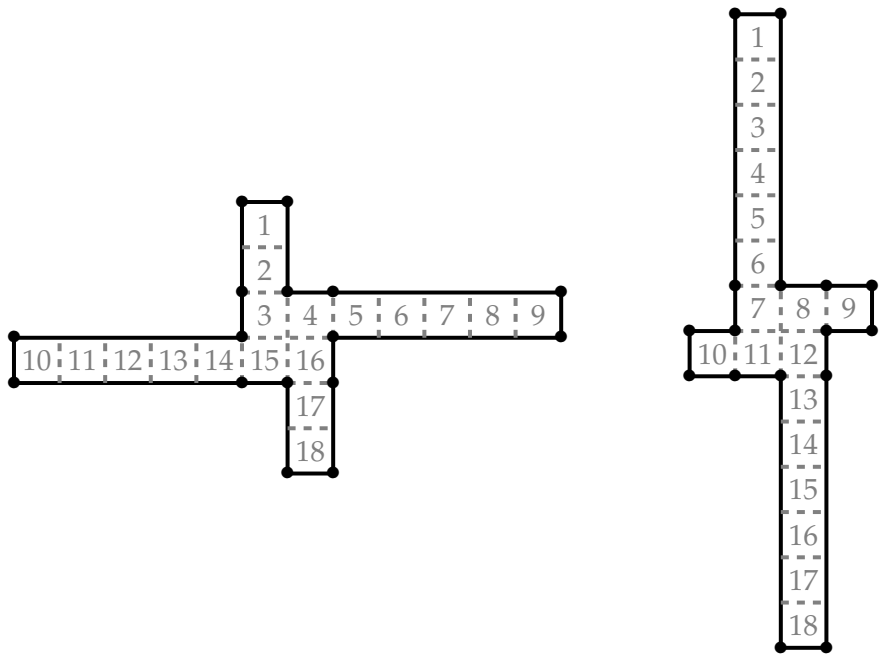
\begin{figure}[t]
    \centering
    \begin{tikzpicture}[scale = 0.6, line width = 1.5pt]
        \draw (0,0) -- (0,3) -- (1,3) -- (1,1) -- (7,1) -- (7,0) -- (2,0) -- (2,-3) -- (1,-3) -- (1,-1) -- (-5,-1) -- (-5,0) -- cycle;
        \draw[gray,dashed] (0,0) -- (2,0);
        \node[gray] at (0.5,2.5) {$1$};
        \node[gray] at (1.5,-2.5) {$18$};
        \foreach \i in {2,3}{
            \draw[gray,dashed] (0,4-\i) -- (1,4-\i);
            \node[gray] at (0.5,3.5-\i) {$\i$};
        }
        \foreach \i in {4,...,9}{
            \draw[gray,dashed] (\i-3,0) -- (\i-3,1);
            \node[gray] at (\i-2.5,0.5) {$\i$};
        }
        \foreach \i in {10,...,15}{
            \draw[gray,dashed] (\i-14,-1) -- (\i-14,0);
            \node[gray] at (\i-14.5,-0.5) {$\i$};
        }
        \foreach \i in {16,17}{
            \draw[gray,dashed] (1,15-\i) -- (2,15-\i);
            \node[gray] at (1.5,15.5-\i) {$\i$};
        }
        \node at (1,7) {$\,$};
        \node at (1,-7) {$\,$};
        \foreach \x/\y in {0/0,0/1,0/3,1/3,1/1,2/1,7/1,7/0,2/0,2/-1,2/-3,1/-3,1/-1,0/-1,-5/-1,-5/0}{
            \node at (\x,\y) {$\bullet$};
        }
    \end{tikzpicture}
    \begin{tikzpicture}[scale = 0.6, line width = 1.5pt]
        \draw (0,0) -- (0,7) -- (1,7) -- (1,1) -- (3,1) -- (3,0) -- (2,0) -- (2,-7) -- (1,-7) -- (1,-1) -- (-1,-1) -- (-1,0) -- cycle;
        \foreach \i in {8,9}{
            \draw[gray,dashed] (\i-7,0) -- (\i-7,1);
            \node[gray] at (\i-6.5,0.5) {$\i$};
        }
        \foreach \i in {1,...,7}{
            \draw[gray,dashed] (0,7-\i) -- (1,7-\i);
            \node[gray] at (0.5,7.5-\i) {$\i$};
        }
        \foreach \i in {12,...,18}{
            \draw[gray,dashed] (1,12-\i) -- (2,12-\i);
            \node[gray] at (1.5,11.5-\i) {$\i$};
        }
        \foreach \i in {10,11}{
            \draw[gray,dashed] (\i-10,-1) -- (\i-10,0);
            \node[gray] at (\i-10.5,-0.5) {$\i$};
        }
        \foreach \x/\y in {0/0,0/1,0/7,1/7,1/1,2/1,3/1,3/0,2/0,2/-1,2/-7,1/-7,1/-1,0/-1,-1/-1,-1/0}{
            \node at (\x,\y) {$\bullet$};
        }
        \node at (-3,0) {$\,$};
        \node at (3,0) {$\,$};
    \end{tikzpicture}
    \caption{On the left, an origami corresponding to the prototype $(14,1,0,-5)$. On the right, an origami corresponding to prototype $(18,1,0,3)$.}
    \label{fig:slope-example-2}
\end{figure}

For example, the origami
\begin{multline*}
    ((2,3,4,5)(6,7,8,9)(10,11,12,13)(14,15,16,17),\\
    (1,16,12,6,2)(3,18,17,13,7)(4,8)(5,9)(10,14)(11,15)),
\end{multline*}
shown on the left of Figure~\ref{fig:slope-example-1}, corresponds to the prototype $(w,h,t,e) = (4,2,0,-7)\in\mathcal{P}^{A_{2}}_{81}$. Here, $(w_1,h_1,w_2,h_2) = (1,1,4,2)$. In this case, it can be checked that setting $j = 2$ works in the construction of Proposition~\ref{prop:LN20-sec8}. That is, the origami contains a cylinder in the direction $\left(w_{2}+\left\lfloor\frac{h_{1}}{2}\right\rfloor,2j+2+h_{1}\right) = (4,7)$ giving rise to an origami corresponding to a prototype in $\mathcal{P}^{A_{1}}_{81}$. The cylinder is shown in blue in Figure~\ref{fig:slope-example-1}. We can make this horizontal by applying $S^{-3}\circ T^{-1}\circ S^{-1}$ (followed by $T^{-1}$ to achieve the cusp representative) to get the origami
\[((3,4,5,6,7,8,9)(10,11,12,13,14,15,16),(1,15,3,2)(4,18,17,16)),\]
shown on the left of Figure~\ref{fig:slope-example-2}, corresponding to the prototype $(14,1,0,-5)\in\mathcal{P}^{A_{1}}_{81}$.

Similarly, for an example of the latter possibility for $(w_{1},h_{1},w_{2},h_{2})$, we can take the origami
\begin{multline*}
    ((1,2)(3,4,5,6,7,8,9)(10,11,12,13,14,15,16)(17,18), \\
    (1,13,3)(2,14,4)(5,17,15)(6,18,16)),
\end{multline*}
shown on the right of Figure~\ref{fig:slope-example-1} and corresponding to the prototype $(7,2,0,-5)$, for which $j = 2$ again works. Here, we have $(w_{1},h_{1},w_{2},h_{2}) = (2,1,7,1)$. The cylinder in direction $(w_{2}+2h_{1},j+1+h_{1}) = (9,4)$ can be made horizontal using $S^{-4}\circ T^{-2}$ (followed by $T^{-1}$ to reach the cusp representative) achieving the origami
\[((7,8,9)(10,11,12),(1,11,7,6,\ldots,2)(8,18,17,\ldots,12)),\]
shown on the right of Figure~\ref{fig:slope-example-2} corresponding to the prototype $(18,1,0,3)$.

So, to connect to the target origami, we either use $O(n^2\log n)$ applications of $T$ or $S$ to reach it directly, or $O(n^2\log n)$ applications of $T$ or $S$ to reach the pseudo-target origami and then the bridge origami, then we apply $O(n)$ applications of $T$ or $S$ to make the appropriate direction horizontal, and finally $O(n^2\log n)$ applications of $T$ or $S$ to reach the target origami.

Hence, we obtain a diameter bound of $O(n^2\log n)$.


\appendix
\section{Potential improvements to the diameter bounds}\label{app:improvements}

Currently, we shadow a butterfly move by travelling to the cusp representative before performing the Euclidean algorithm operation. Is it possible that there exists a cheap (i.e., still $O(n)$) Euclidean algorithm operation that one can perform at any origami within the cusp that still moves to the target cusp? If so, we could improve the $O(n^{2})$ cost of a butterfly move in this setting to $O(n)$, and connecting to a reduced prototype would only cost $O(n\log n)$. Alternatively, is it possible to reach the cusp representative faster than by simply applying powers of $T$?

Reduced prototypes have the form $(0,b,1,e)$. If $e = 0$ (which can only happen if $D\equiv 0\mod 4$), then we are the prototype  $(0,\frac{D}{4},1,0)$. If $e > 0$, then apply $B_{1}$ again to obtain the prototype $(0,b-e-1,1,-e-2)$ with $-e-2 < 0$. So we can assume that $e < 0$.

For $D\equiv 0 \mod 4$, so that $e$ is even, set the target prototype to be $(0,\frac{D}{4},1,0)$. If $\gcd(b,-\frac{e}{2}) = 1$, then applying $B_{-\frac{e}{2}}$ will reach the target prototype. We also see that $q = -\frac{e}{2}$ is admissible since $e < 0$ and even, and 
\[(e+2qc)^{2} = (e+2(-\frac{e}{2}))^{2} = 0 < D.\]

If $\gcd(b,-\frac{e}{2})\neq 1$, then the computer suggests that we can find an admissible $q$ with $-e-2q < 0$ and $\gcd(b,q) = 1 = \gcd(\frac{e}{2}+q,b-eq-q^{2})$. In such a case, the composition of butterfly moves $B_{\frac{e}{2}+q}\circ B_{q}$ will send $(0,b,1,e)$ to $(0,\frac{D}{4},1,0)$.

Can we prove that such a $q$ value always exists?

If $D \equiv 1\mod 4$ then set the target prototype to be $(0,\frac{(D-1)}{4},1,\pm 1)$, depending on the HLK-invariant if $D\equiv 1 \mod 8$. The same argument as above but with $q = \frac{-e\pm 1}{2}$ should also work. Again, there is a computationally supported conjecture that $\gcd(b,\frac{-e\pm 1}{2}) = 1$ can be achieved after possibly applying another butterfly move. We also see that $q = \frac{-e\pm 1}{2}$ is admissible since
\[(e+2qc)^{2} = (e-e\pm1)^{2} = 1 < D.\]

If the existence of such $q$ can be proved then the reduced prototypes can be connected in $O(1)$ butterfly moves that each cost $O(n)$.

If all of the above improvements are achieved, then we would obtain the bound $O(n\log n) = O(|V|^{\frac{1}{3}}\log |V|)$.

\section{Bounds via the algorithm of Hubert--Leli\`evre}\label{sec:H2-HL}

Here we investigate the diameter bounds that come from a naive analysis of the classification proof of Hubert--Lelièvre. We use the language of Subsection~\ref{subsec:params}. We will see that it achieves the weaker bound $O(|V|^\frac{5}{6}) = O(n^\frac{5}{2})$.

\subsection{From two cylinders cusps to one cylinder cusps}\label{subsec:HL-2-to-1}

Let $X$ be a primitive square-tiled surface in $\mathcal{H}(2)$ tiled by a prime number $n$ of unit squares. Suppose that $X$ decomposes into two horizontal cylinders with heights $h_i$, widths $w_i$, and twists $t_i$, $i=1, 2$. Recall that $w_1 h_1 + w_2 h_2 = n$ and that, by applying $T^k$ with $0\leq k \leq w_1 w_2\leq n^2$ to $X$, we can assume that $0\leq t_i < \gcd(w_i,h_i) \leq \sqrt{n}$ for $i=1, 2$. 

Denote by $h_{total}= h_1 + h_2$ the total height of $X$. The values of the twists allow us to distinguish four cases: 

\begin{itemize}
\item Case $(I)$: $t_1, t_2\neq 0$
\item Case $(II)$: $t_1=0$, $t_2\neq 0$
\item Case $(III)$: $t_1\neq 0$, $t_2=0$
\item Case $(IV)$: $t_1 = t_2 = 0$
\end{itemize} 

We want use the inductive procedure in \cite[\S 5.2]{HL} to reduce $h_{total}$. For this sake, let us recall the algorithm introduced by Hubert and Leli\`evre. 

\subsubsection{Reduction of Case $(I)$} We apply $R$ to $X$. In this way, we get a one-cylinder surface or a two-cylinder surface with total height $h_{total}'\leq t_1+t_2 < h_1+h_2$. Note that $h_{total}'\leq 2\sqrt{n}$. 
\subsubsection{Reduction of Case $(II)$} We apply $R$ to $X$. If the resulting surface has two cylinders, its total height is $h_{total}'\leq t_2 < h_2$. Note that $h_{total}'\leq \sqrt{n}$. 
\subsubsection{Reduction of Case $(III)$} We apply $R$ to $X$ in order to obtain a surface with two cylinders: the top cylinder has zero twist and the bottom cylinder has height $\leq t_1 < h_1$. By applying $T^l$ to $R(X)$ with an adequate choice of $0\leq l\leq n^2$ in order to minimize the twist parameters in the cusp of $R(X)$, we get a surface with $t_1'=t_2'=0$ (Case (IV)), or $t_1'=0$ and $0\neq t_2'$ (Case (II)). 
\subsubsection{Reduction of case $(IV)$} By \cite[Lemma 5.3]{HL}, we have that $X$ has a single cylinder in the direction $(w_1,h_2)$. This direction can be made horizontal with a Euclidean algorithm procedure that takes $O(\max\{w_{1},h_{2}\})$, and so also $O(n)$, steps.

\subsubsection{Two-cylinder counts:}

In summary, given a square-tiled surface $X$ with two cylinders and total height $h_{total}=h_1+h_2$, after taking at most $n^2$ steps, we can assume that its twist parameters are $0\leq t_i < \gcd(w_i,h_i)\leq \sqrt{n}$. At this point, we have the following possibilities for reaching a one-cylinder surface:

\begin{itemize}
    \item \textbf{From a Case (I) surface of height $\boldsymbol{>2}$:} Each surface along the way must have been a Case I-III surface. At each step $O(n^{2})$ moves are used to rotate and normalise the twists, there are at most $O(n^{\frac{1}{2}})$ steps and so there are a total of $O(n^{\frac{5}{2}})$ moves to reach a one-cylinder surface.
    \item \textbf{From a Case (II) surface of height $\boldsymbol{>2}$:} As in the previous case, this requires $O(n^{\frac{5}{2}})$ steps.
    \item \textbf{From a surface of height 2:} We use at most $O(n^{\frac{5}{2}})$ moves to reach a surface of height 2. Following this, we use the $O(n^{2})$ steps of \cite[Lemma 5.2]{HL}. This is a total of $O(n^{\frac{5}{2}}+n^{2}) = O(n^{\frac{5}{2}})$ moves.
    \item \textbf{From a Case (IV) surface:} It takes at most $O(n^{\frac{5}{2}})$ steps to reach a Case (IV) surface. As discussed above, applying \cite[Lemma 5.3]{HL} requires an additional $O(n)$ steps. The total is $O(n^{\frac{5}{2}})$.
\end{itemize}

So any two-cylinder surface can be connected to a one-cylinder surface in at most $O(n^{\frac{5}{2}})$ moves.

\subsection{Connecting one cylinder cusps among themselves} Let $X$ be a square-tiled surface tiled by a prime number $n$ of unit squares. Assume that $X$ decomposes into a single horizontal cylinder. In this case, the surfaces in the cusp determined by $X$ are determined by the lengths $(a,b,c)$ of the saddle-connections in the bottom of the cylinder (so that $a+b+c=n$) and a twist parameter. 

\subsubsection{From $(a,b,c)$ surfaces to $(1,\ast,\ast)$ surfaces} 

By taking at most $n$ steps, we can assume that our $(a,b,c)$ surface $X$ has top saddle-connections of lengths $b$, $a$ and $c$. As it is explained in \cite[\S 5.3.1]{HL}, $R(X)$ is a $(\delta,k\delta,\gamma)$ surface in the direction $(1+t,d)$, where $d=\gcd(a,b)$, $0\leq t < (a+b)/d$ is some twist parameter, $\delta=\gcd(1+t,d)$, and $\gcd(\gamma,\delta)=1$. Hence, by applying a Euclidean algorithm procedure to make the $((1+t)/\delta,d/\delta)$ direction horizontal, we see that any $(a,b,c)$ surface $X$ can be joined to a $(\delta,k\delta,\gamma)$ surface with $\delta$ dividing $\gcd(a,b)=1$ and $\gcd(\delta,\gamma)=1$ in $O(n)$ steps. Thus, by using this procedure once more (with $a, b$ replaced by $\gamma, \delta$), we derive that any $(a,b,c)$ surface $X$ can be joined to a $(1,\ast,\ast)$ surface in $O(n)$ steps. 

\subsubsection{From $(1,b,c)$ surfaces with $b, c$ odd to $(1,1,n-2)$ surfaces} Consider the $L$-shaped surface $Y$ with arms of widths one and lengths $b, c$. Since it is a $(1,b,c)$ surface in the direction $(1,1)$, it suffices to connect $Y$ to a $(1,1,n-2)$ surface. For this sake, we observe that $R\circ T^2(Y)$ is a surface with two cylinders of heights one. By applying $T^l$ for some $0\leq l\leq n^2$, we have that $Z=T^l\circ R\circ T^2(Y)$ has two cylinders of heights one and both twist parameters equal to zero. As it turns out, the $((n-b)/2,1)$ direction in $Z$ is a $(1,1,n-2)$ surface, so that $Z$ can be joined to a $(1,1,n-2)$ surface in $O(n)$ steps. So the original $(1,b,c)$ surface is connected to a $(1,1,n-2)$ surface in $O(n^{2})$ steps.

\subsubsection{From $(1,b,c)$ surfaces with $b, c$ even to $(1,2,n-3)$ surfaces} Let $X$ be a $(1,b,c)$ surface with $b=2b'$ and $c=2c'$. 

Assume first that $b\neq c$, say $b' < c'$. The surface $X$ decomposes into two cylinders in the direction $(c'-b',1)$: the resulting surface $Y$ has a top cylinder with $h_1=b$, $w_1=2$, and a bottom cylinder with $h_2=1$, $w_2=2+\ell$. As it is explained in~\cite[\S 5.3.3]{HL}, the surface $Y$ gives rise to a $(d,2d,\ast)$ surface, $d=\gcd(\ell,b')$, in the direction $(\ell,b')$. Hence, by taking at most $O(n)$ steps, $X$ can be joined to a $(d,2d,\ast)$ surface. Finally, since any $(d,2d,\ast)$ surface gives rise to a $(1,2,\ast)$ surface in the direction $(d,1)$, we conclude that $X$ can be joined to a $(1,2,\ast)$ surface in $O(n)$ steps. 

Assume now that $b=c$. As it is explained in~\cite[\S 5.3.3]{HL}, by taking at most $n$ steps in the cusp of $X$, we get a $(1,b,c)$ surface whose $(b',1)$ leads to a $(2,2,\ast)$ surface $Y$. In the $(2,1)$ direction of $Y$, we see a two cylinders surface with $h_1=2$, $w_1=1$, $h_2=1$. By choosing an adequate $0\leq l\leq n^2$ and by acting with $T^l$, we can set the twist parameters to zero, so that we obtain a $(1,2,n-3)$ surface by looking in the direction $(1,1)$. In particular, $X$ can be joined to a $(1,2,n-3)$ surface in $O(n^2)$ steps.

\subsection{Resulting diameter bounds}

We see that all of the one-cylinder surfaces are connected to the appropriate target surface in $O(n^{2})$ steps. Any two-cylinder surface is connected to a one-cylinder surface in $O(n^{\frac{5}{2}})$ steps. Hence, the diameter bound is $O(n^{\frac{5}{2}}) = O(|V|^{\frac{5}{6}})$.

\subsection{Potential improvements}

Computer investigations suggest that a two-cylinder surface can be sent to a one-cylinder surface after only a small (maybe $O(1)$ or $O(\log n)$) number of applications of the steps discussed in Subsection~\ref{subsec:HL-2-to-1}. Each step requires at most $O(n^{2})$ steps to move inside the cusps. So this would suggest an improvement from $O(n^{\frac{5}{2}})$ to $O(n^{2})$ or $O(n^{2}\log n)$ giving the diameter bound $O(|V|^{\frac{2}{3}})$ or $O(|V|^{\frac{2}{3}}\log|V|)$. Furthermore, it is likely possible to improve the $O(n^{2})$ bound for moving in the cusp to some bound of the form $O(n^{c})$ with $c < 2$, since $O(n^{2})$ cusp width requires small height.

\end{document}